\documentclass[12pt]{article}
\usepackage{amsmath, amsthm, amssymb, amsfonts, mathtools, graphicx, enumerate}
\usepackage{times}
\usepackage{booktabs}
\usepackage{url}
\usepackage{color,soul}
\usepackage{enumerate}
\usepackage[inline]{enumitem} 
\usepackage{listings}
\usepackage{float}
\usepackage[backend=biber,style=ieee,uniquename=false,defernumbers=true]{biblatex}
\usepackage[justification=centering]{caption}
\usepackage{titlesec} % rename sections
\usepackage{tocloft}
\usepackage{algorithm,algorithmic}
\usepackage{placeins} %float barriers
\usepackage{subcaption} %subfigures
\usepackage{mathrsfs}
\usepackage{bbm}
\usepackage{hyperref,cleveref} % bi-directional proof linking
\newcommand{\R}{\mathbb{R}}

\newcommand{\expfunc}[1]{\exp\left\{  #1 \right\}}

\newcommand{\E}[0]{\mathbb{E}}

\newcommand{\T}[0]{^\top}
\newcommand{\Prob}[0]{\mathbb{P}}
\newcommand{\suml}[2]{\sum_{#1}^{#2}}

\newcommand{\indicator}[1]{1\kern-4pt 1(#1)}
\newcommand{\indep}[0]{\perp\kern-5pt\perp}
\newcommand{\tr}[0]{\textrm{tr}}

\newcommand{\mkloc}[0]{\mathcal{M}_K^{\textrm{loc}}}
\newcommand{\proofref}[1]{\hyperlink{#1}{proof}} % bi-directional linking
\newcommand{\rom}[1]{\operatorname{\uppercase\expandafter{\romannumeral #1\relax}}}
\newcommand{\dx}{\mathop{}\!\mathrm{d}x}

\newtheorem{theorem}{Theorem}
\newenvironment{delayedproof}[1]
 {\begin{proof}[\raisedtarget{#1}Proof of \Cref{#1}]}
 {\end{proof}}
\newcommand{\raisedtarget}[1]{%
  \raisebox{\fontcharht\font`P}[0pt][0pt]{\hypertarget{#1}{}}%
}
\newtheorem{lemma}[theorem]{Lemma}

\newtheorem{corollary}[theorem]{Corollary}
\newtheorem{definition}[theorem]{Definition}
\newtheorem{remark}[theorem]{Remark}

\numberwithin{theorem}{section}
\numberwithin{equation}{section}

\DeclareMathOperator*{\argmin}{arg\,min}

\definecolor{codegreen}{rgb}{0,0.6,0}
\definecolor{codegray}{rgb}{0.5,0.5,0.5}
\definecolor{codepurple}{rgb}{0.58,0,0.82}
\definecolor{backcolour}{rgb}{0.95,0.95,0.92}
\definecolor{teal}{rgb}{0,0.75,0.75}
\lstdefinestyle{mystyle}{
    backgroundcolor=\color{backcolour},   
    commentstyle=\color{codegreen},
    keywordstyle=\color{magenta},
    numberstyle=\tiny\color{codegray},
    stringstyle=\color{codepurple},
    basicstyle=\ttfamily\footnotesize,
    breakatwhitespace=false,         
    breaklines=true,                 
    captionpos=b,                    
    keepspaces=true,                 
    numbers=left,                    
    numbersep=5pt,                  
    showspaces=false,                
    showstringspaces=false,
    showtabs=false,                  
    tabsize=2
}

\newcommand{\deltafunc}{\frac{\delta^2}{\vVar^2}}

\newcommand{\thenewerboundary}{{(4D_4)^{-1}}}
\newcommand{\vVar}{\Upsilon}
\newcommand{\txi}{\tilde{\xi}}

\DeclareFieldFormat{labelalpha}{\mkbibbold{#1}}
\DeclareFieldFormat{extraalpha}{\mkbibbold{\mknumalph{#1}}}

\AtBeginBibliography{
  \DeclareFieldFormat{labelalpha}{#1}

  \DeclareFieldFormat{extraalpha}{\mknumalph{#1}}}

\usepackage{hyperref}

\title{Statistical Guarantees for Covariance Estimation with Adversarial Corruption under Constraints}
\author{Samuel Ozminkowski \and Hongmei Jiang \and Matey Neykov \and
Department of Statistics and Data Science, Northwestern University}
\begin{document}
\maketitle

\begin{abstract}
    We obtain a near-minimax rate on covariance estimation of adversarially corrupted Gaussian data in Frobenius norm when the covariance estimation is assumed to lie in a star-shaped set $K \subseteq \R^{n\times n}$. Assuming a known upper bound of $\varepsilon < 1/32$ on the fraction of the $N$ observations of data $X$ are arbitrarily corrupted by an omniscient adversary, we obtain the near-minimax rate (up to constants) of
    \begin{align*}
        \left(\varepsilon^2 \vee \delta^{\star2} \right)\wedge d^2 \lesssim \inf_{\hat\Sigma} \sup_{\Sigma \in K} \sup_{\mathcal{C}} \E \| \hat\Sigma(\mathcal{C}(X)) - \Sigma \|_F^2 \lesssim \left(\varepsilon^2\log^2(1/\varepsilon) \vee \delta^{\star2} \right)\wedge d^2
    \end{align*}
    where $d$ is the diameter of the set $K$, $N \gtrsim \sup_{\delta > 0} \log \mathcal{M}^{\operatorname{loc}}(\delta, 2c)$, and
    \begin{align*}
        \delta^*=\sup\left\{ \delta\geq 0: \frac{N}{\Upsilon^2} \delta^2\leq \log \mathcal{M}_K^{\operatorname{loc}}(\delta,c) \right\}
    \end{align*}
    for an arbitrary large constant $c$ and the local metric entropy of $K$, $\mathcal{M}^{\operatorname{loc}}_K(\delta,c)$. We believe this is the first such general result on structured robust covariance estimation in the Frobenius norm.
    % We further apply our results to the class of covariance matrices in the class
    % \begin{align*}
    %     \mathcal{U}_q(s_0(n)) \coloneqq \left\{ \Sigma: \Sigma \succ 0, \max_i \sigma_{ii} \leq K, \max_i \sum_{j=1}^n |\sigma_{ij}|^q \leq s_0(n) \right\}, \hspace{4mm} q \in [0,2)
    % \end{align*} 
    % for which we prove an upper bound on the minimax rate
    % \begin{align*}
    %     \omega \varepsilon^2 \log^2(1/\varepsilon)\bigvee n s_0(n) \omega^{1-q/2} \left(\frac{\log n}{N}\right)^{1-q/2},
    % \end{align*} and the sparse spiked covariance model, for which we prove the minimax rate 
    % \begin{align*}
    %     \varepsilon^2 \omega \bigvee \frac{\omega}{N} \mathcal{M}_K^{\textrm{loc}}\left(\sqrt{\delta},c\right) \lesssim \mathfrak{M}(X) \lesssim  \varepsilon^2 \log^2(1/\varepsilon)\omega \bigvee \frac{\omega}{N}\mathcal{M}_K^{\textrm{loc}}\left(\frac{\delta}{2K_{\max}},c\right)
    %     \label{eq:k_sparse_rate}
    % \end{align*} 
    % for the set of covariance matrices 
    % \begin{align*} 
    %     K^+ \coloneqq \left\{ \Sigma: \Sigma = I + xx\T, x \in K\textrm{ star shaped around $0$, $\|x\|_2 \leq K_{\max}$} \right\}.
    % \end{align*}
    % where $\mathcal{M}_K^{\textrm{loc}}(\delta,c)$ is the local metric entropy of $K$.
\end{abstract}

\tableofcontents

\section{Introduction}

Covariance estimation is one of the classic goals of statistics. In recent years a large body of work has been established focusing on high dimensional data as well as on corrupted data regimes, where some small $\varepsilon$ fraction of the data collected is incorrect in some way \cite{huber_robust_1964,diakonikolas_algorithmic_2023a}. We focus on the adversarial corruption regime, where an adversary with full knowledge of the true data generating mechanism can change as much as $\varepsilon$ proportion of the data in any way it pleases while we have no knowledge of which data is changed or the nature of said changes. 

Our goal is to estimate the covariance matrix in the corrupted multivariate Gaussian regime, where the uncorrupted data generating mechanism is \[\tilde{X_i} \sim \mu + \tilde{\xi}_i \textrm{ for } \tilde{\xi}_i \sim N(0,\Sigma).\] Here, $\tilde{\xi}_i \in \R^{n}$ for known dimension $n$. We develop minimax error upper and lower bounds for estimating $\Sigma$ with regards to the Frobenius norm. In the remainder of the paper we assume without loss of generality that $\mu=0$, because if $\mu \neq 0$ we can shuffle the data and then work on the difference of pairs: $X_1^\prime = \frac{1}{\sqrt{2}}(X_1 - X_2), X_2^\prime = \frac{1}{\sqrt{2}}(X_4-X_3)$, etc. In the worst case this doubles the corruption rate, hence a theorem proved for contamination \(<1/32\) yields the unknown-mean version when the original contamination is \(<1/64\). This procedure also halves the sample size but does not change the rate of increase in the sample size, so that requirement does not change. 

\subsection{Related Literature}

\subsubsection{Mean estimation}

Much of the literature on estimation in the adversarial corruption regimes focuses on mean estimation in the $\ell_2$ vector norm. Many of the algorithms and strategies are relevant to the task of covariance estimation as well. For this target, \textcite{neykov_minimax_2023} developed a local packing strategy which was integral in proving a minimax rate for estimation of the mean in the Gaussian Sequence Model. \textcite{prasadan_information_2025} then generalize this result to prove the minimax rate for mean estimation when the data is sub-Gaussian. Similar results under different data regimes have been presented by a number of researchers. For example, \textcite{lugosi_robust_2021} present robust mean estimation under contamination under the assumption that the observations are independent and identically distributed. \textcite{novikov_robust_2025} provide guarantees for robust mean estimation without moment assumptions, only assuming symmetry. They do this by adapting a filtering algorithm for moment constraints to just symmetric distributions. More similarly to our regime, \textcite{minsker_uniform_2019} shows guarantees for robust mean estimation under the assumption that random variables have low order moments (between 2nd and 3rd) and that the trace and maximum eigenvalue of the covariance matrix are known. Perhaps the most prolific researcher in this area has been Diakonikolas. For example, \textcite{diakonikolas_robust_2019} develop an algorithm for mean estimation with corrupted Gaussians, mixtures of Gaussians, binary product distributions, and balanced binary products under corruption regime. In general, Diakonikolas has focused a lot on efficient, practical estimation algorithms \cite{diakonikolas_being_2017, diakonikolas_robust_2019, cheng_robust_2018, diakonikolas_robustly_2018, diakonikolas_statistical_2017}. Finally, in a more general regime, \textcite{bellec_sharp_2016} gives error bounds for means estimated via least squares under shape-restricted sets where the model is misspecified such that the true mean is not in the shape-restricted set. 

\subsubsection{Covariance estimation}

Next, we turn to the literature for covariance estimation. The bulk of the work in this area analyzes the error of estimators under the operator norm. For example, \textcite{abdalla_covariance_2023} present a similar result to ours, but their guarantees are for estimation error in the operator norm, under the assumption that the true covariance matrix satisfies $L_p-L_2$ norm equivalence for $p \geq 4$. \textcite{fan_large_2013} give bounds in the maximal norm for covariance matrices of factor models. \textcite{chen_robust_2018} present guarantees robust covariance estimation under Huber contamination model focusing on the operator norm, and \textcite{minasyan_statistically_2025} give results for the case where the data takes the form of Gaussian vectors. Interestingly, their error rate is dependent on trace of $\Sigma$ and operator norm of $\Sigma$. Estimation error under the operator norm is a natural area of interest because the operator norm does not necessarily punish increased matrix size, while there is a direct relationship between the size of a matrix and its Frobenius norm. Furthermore, the Frobenius norm is the matrix extension of the vector mean squared error, which is a classic measure of error in statistics.

However, for this reason the non-asymptotic properties of an estimator under the Frobenius norm can be very informative. To this point some interesting results have been presented for this norm. \textcite{ma_volume_2013} prove a minimax rate for both mean and covariance estimation under the (quite broad) set of unitarily invariant norms in the uncorrupted regime, with the only assumption being that the operator norm of the true covariance matrix is bounded. Of course, given the broad set of norms under consideration, this bound can be expected to be quite loose if one is willing to make additional assumptions. \textcite{lai_agnostic_2016} developed a polynomial time algorithm for covariance estimation with Frobenius error on order $\sqrt{\log n} \lVert \Sigma\rVert_2$. \textcite{cheng_faster_2019} present an algorithm for robust covariance matrix estimation with a focus on algorithmic speed, and obtain a Frobenius error rate which is dependent on the spectral norm of the true covariance matrix. \textcite{minsker_robust_2024} present a LASSO-type estimator and prove bounds based on the rank of the true covariance under $L_4-L_2$ norm equivalence in both Frobenius and operator norms. Finally, in the case most similar to our work, \textcite{novikov_nearly_2025} gives a covariance estimation procedure with a bound on the relative Frobenius norm: $\lVert\Sigma^{-1/2}\hat\Sigma\Sigma^{-1/2} - I\rVert_F \leq O(\epsilon \log(1/\epsilon))$. This algorithm requires the data follow an elliptical distribution with a covariance matrix $\Sigma$ with effective rank greater than $C \log n$, and $N \geq C n^2 (\log^5n) / \epsilon^2$ for some constant $C$. No moment conditions on $\Sigma$ are required, but the bounds are for a given $\Sigma$ (i.e. not minimax). This algorithm also requires the condition $\epsilon \geq C \log(n)/\sqrt{n}$ for the Frobenius error bound.

\subsubsection{Relevant results from the uncorrupted regime}
We also note here some important results in the uncorrupted regime which should be extended to the adversarial corruption regime. In one of the classic results of high-dimensional covariance estimation, \textcite{bickel_covariance_2008} prove minimax rates in both the spectral and Frobenius norms when the covariance matrix is in the uniformity class of covariance matrices invariant under permutations. \textcite{cai_optimal_2021} give minimax bounds for the spiked covariance model (with normally distributed data) where the lower bound is based on the local metric entropy but the upper bound is based on the Dudley's entropy integral. \textcite{vu_minimax_2012} provide minimax bounds for the spiked covariance model with sparse covariance, where the vectors which form the spikes are in the $\ell_q$ ball. The lower bound is based on the sparsity and requires an assumption that the vectors are not too dense. These bounds differ by a factor of the radius of the $\ell_q$ ball. Another closely related framework to the adversarial corruption regime is differential privacy. In this framework, the exclusion of a single subject is guaranteed to have a pre-defined small effect on the outcome of an analysis. In this framework, \textcite{cai_optimal_2024} establish minimax rates under any Schatten-$q$ norm for covariance matrix estimation in the spiked covariance model for sub-Gaussian data, and \textcite{cai_minimax_2026} provide an estimation algorithm based on the blockwise tri-diagonal operator and show minimax rates under spectral and Frobenius norms for sub-Gaussian data.

\subsection{Notation and Definitions}

First and foremost, $\|A\|_F$ refers to the Frobenius norm of the matrix $A$. This is our main object of study, so when it is clear we may drop the $F$ and refer to the Frobenius norm of a matrix $A$ as $\|A\|$. In the vector space, the Euclidean norm (or 2-norm) is referred to as $\|a\|_2$ for a vector $a$. However, for a matrix $\|A\|_2$ refers to the spectral norm - the largest singular value of the matrix $A$ - the matrix norm induced by the vector $2-$norm. We do our best to avoid confusion between the matrix spectral norm and the vector 2-norm in the text. Next $[N]$ is used to indicate the set of integers $\{1,2,...,N\}$. $N(\mathbf{\mu},\Sigma)$ refers to the multivariate normal distribution with mean vector $\mathbf{\mu}$ and covariance matrix $\Sigma$. The indicator variable $\indicator{B}$ is equal to $1$ if condition $B$ evaluates to true and $0$ otherwise. $\mathbb{I}_p$ refers to the identity matrix, a $p\times p$ matrix with ones on the diagonal. The subscript is dropped when the dimension is clearly implied. $TV(P,Q)$ refers to the total variation distance between measures $P$ and $Q$ on probability space $(\Omega,\mathcal{F})$: $TV = \sup_{C \in \mathcal{F}} |P(C) - Q(C)|$. The maximum of two numbers $a, b \in \mathbb{R}$ is referred to as $a \vee b$ while the minimum is $a \wedge b$. $B_F(A,\delta)$ indicates the $\delta-$ball around point $A$ with respect to the Frobenius norm. When the meaning is clear we drop the subscript and just refer to it as $B(A,\delta)$. 

Next, we give key definitions required for our results. Definitions \ref{def:local_metric_entropy}-\ref{def:star_shaped} are taken directly from \textcite{prasadan_information_2025}.

\begin{definition}[Local metric entropy]\label{def:local_metric_entropy}
    For a set $K$, the $\delta$-packing number is the maximum cardinality $\mathcal{M}(\delta,K)$ of a set $\{\Sigma_1,...,\Sigma_M\} \subseteq K$ such that $\|\Sigma_i - \Sigma_j\|_F \geq \delta$ $\forall i\neq j$. Then for a fixed constant $c>0$ \textbf{the local metric entropy} for the set $K$ is $\mathcal{M}_K^{\operatorname{loc}}(\delta,c) = \sup_{\Sigma \in K} \mathcal{M}(\delta/c, B(\Sigma,\delta) \cap K)$. We refer to $\delta$ as the \textbf{local radius} and c as the \textbf{radius factor}. When the set $K$ is unambiguous, we may drop the $K$ and write $\mathcal{M}^{\operatorname{loc}}(\delta,c)$.
\end{definition}

\begin{definition}[Star-shaped sets]\label{def:star_shaped}
    A set $K$ is called \textbf{star-shaped with center $k^\star$} if $\exists k^\star \in K$ such that $\forall k \in K$ and $\alpha \in [0,1]$, $\alpha k + (1-\alpha)k^\star \in K$.  
\end{definition}

\begin{definition}[Minimax error under adversarial corruption]
    The minimax estimation under adversarial corruption is given as \[\mathfrak{M}(X) = \inf_{\hat\Sigma} \sup_{\Sigma} \sup_{\mathcal{C}_\varepsilon} \E \| \hat\Sigma(\mathcal{C}_\varepsilon(\tilde{X})) - \Sigma \|_F^2 \]
    where $\mathcal{C}_\varepsilon$ gives the adversarial corruption scheme, $\Sigma$ is the true covariance of $X$ before corruption, and $\hat\Sigma$ is the estimator of $\Sigma$ given the corrupted data $\mathcal{C}_\varepsilon(\tilde{X})$. When the corruption rate is clear we drop the $\varepsilon$ and simply write $\mathcal{C}(\tilde{X})$. 
\end{definition}

We close this section by repeating two lemmas from \textcite{prasadan_information_2025}. For proofs, see Lemma 1.3 and Lemma 1.4 in that paper. Both lemmas are proved for vectors with the Euclidean norm, but the same proofs hold exactly for matrices in the Frobenius norm. 

\begin{lemma}\label[lemma]{lemma:d_line_segment}
    Any bounded star-shaped set $K$ with diameter $d$ contains a line segment of length at least $d/3$.
\end{lemma}

\begin{lemma}\label[lemma]{lemma:loc_metric_entropy_noninc}
    For any star-shaped set $K$ and constant $c>0$, the local metric entropy is non-increasing in terms of the local radius. That is, for $\delta_0 > \delta_1$, $\mathcal{M}_K^{\operatorname{loc}}(\delta_0,c) \leq \mathcal{M}_K^{\operatorname{loc}}(\delta_1,c)$.
\end{lemma}

\subsection{Problem Set-Up}

\paragraph*{Conditions}

Next, we establish conditions under which our results hold: 

\begin{enumerate}[label={(\arabic*)},itemjoin*={ and }]
    \item The uncorrupted data $\tilde{X}_i = \mu + \tilde{\xi}_i$ where $\tilde{X}_i \in \R^n$ for $i \in [N]$ with $\tilde{X}_i$ independent, \label{cond:x}
    \item $\tilde{\xi}_i \sim N(0,\Sigma)$, \label{cond:xi}
    \item the true covariance matrix is symmetric positive definite with bounded eigenvalues, \\$\Sigma \in K \subseteq \{ \Sigma : \Sigma = \Sigma\T, \Sigma \succ 0, 0 < c_0 \leq \lambda_{\min}(\Sigma) \leq \lambda_{\max}(\Sigma) \leq c_1 < \infty \}$, \label{cond:K_set}
    \item $K$ is star shaped and bounded with diameter $d$, \label{cond:K_star}
    \item $X = \mathcal{C}_\varepsilon(\tilde{X})$ is observed, where $\varepsilon \leq 1/32$ of the datapoints are arbitrarily corrupted by an adversary with full knowledge of the underlying distribution. \label{cond:corruption}
\end{enumerate}

\begin{remark}
    We note here that the bound on $K$ is implied by the upper bound on $\lambda_{\max}(\Sigma)$. Specifically, that upper bound gives $d \leq 2\sqrt{n}c_1$. 
\end{remark}

\subsection{Organization}

The remainder of this paper is organized as follows. First, a lower bound on the minimax rate, $\mathfrak{M}(X) \gtrsim \delta^2 \bigvee (\varepsilon^2 \wedge d^2)$, is established in \Cref{sec:lower_bound}. This lower bound is proven directly based on existing and commonly used tools in the corrupted mean estimation literature. The next main challenge of the paper relies on proving a matching upper bound. In fact, we are unable to strictly do so. However, in \Cref{sec:upper_bound} we establish an upper bound which matches the lower bound up to a multiple of $\log^2(1/\varepsilon)$. We explain why this upper bound should be regarded as tight enough despite not matching the lower bound exactly. In \Cref{sec:examples}, we tackle the final main challenge of this paper: we show how the bound can be applied to real covariance matrices of interest in the literature. Most importantly, in this section we show how our bound allows us to extend the probability upper bound of \textcite{bickel_covariance_2008} on the uniformity class of matrices invariant under permutation, $\mathcal{U}_q$, to the corrupted data regime.

\section{Lower Bound}\label{sec:lower_bound}

We begin with a discussion of the lower bound. The lower bound relies heavily on Fano's inequality, so we restate it here without proof. For a proof see \textcite[pp. 502]{wainwright_highdimensional_2019}.

\begin{lemma}[Fano's Inequality] \label[lemma]{lemma:fano}
    Let $\{\Sigma^1,\dots,\Sigma^M\}$ be a $\eta$-separated set of points on the parameter space $K$ with the Frobenius norm. Suppose $J$ is uniformly distributed over the index set $[M]$ and $(X | J=j) = \mu + \xi$ for $\xi \sim N(0, \Sigma_j)$. Then we have 
    \begin{align}
        \inf_{\hat\Sigma} \sup_\Sigma \E \| \hat\Sigma(X) - \Sigma \|_F^2 \geq \frac{\eta^2}{4} \left(1 - \frac{I(X;J) + \log 2}{\log M} \right) 
    \end{align}
    where the mutual information $I(X;J)$ is defined as the average KL divergence of $\prod_{i=1}^{N}(X_i|J=j)$ as given in \textcite{wainwright_highdimensional_2019} eq. (15.30).
\end{lemma}

The lower bound is given in two parts, and each is proven separately. We begin with the more general bound which is independent of the corruption and dependent solely on the nature of the local packing parameter $\delta$. The \proofref{lemma:lower_bound_1} of this lemma is relegated to \Cref{appendix:lower_bound_pfs}.

\begin{lemma}\label[lemma]{lemma:lower_bound_1}
    Consider $\delta$ such that 
	\begin{align*}
		\log \mkloc(\delta,c) > 4\bigg(\frac{C_0}{c_0^2} N\delta^2 \vee \log 2\bigg)
	\end{align*}
    for a constant $C_0>1$ to be defined later.
	Under Conditions \ref{cond:x}-\ref{cond:corruption},

    \begin{align}
        \inf_{\hat\Sigma} \sup_{\Sigma} \sup_{\mathcal{C}} \E \| \hat\Sigma(\mathcal{C}(X)) - \Sigma \|_F^2 \geq \inf_{\hat\Sigma} \sup_{\Sigma} \E \| \hat\Sigma(X) - \Sigma \|_F^2 \gtrsim \delta^2.
    \end{align}
\end{lemma}

Next we bound the minimax error from below with a bound more specific to our case. This bound relies on the smaller of the level of corruption and the diameter of the set containing $\Sigma$. The \proofref{lemma:lower_bound_2} of this lemma is also relegated to \Cref{appendix:lower_bound_pfs}.

\begin{lemma} \label[lemma]{lemma:lower_bound_2} Let $\xi \sim N(0,\Sigma)$. Under Conditions \ref{cond:x}-\ref{cond:corruption}  and if $\varepsilon > \frac{1}{\sqrt{N}}$ then 
\begin{align}
    \inf_{\hat\Sigma} \sup_{\Sigma} \sup_{\mathcal{C}} \E \| \hat\Sigma(\mathcal{C}(X)) - \Sigma \|_F^2 &\gtrsim \varepsilon^2 \wedge d^2.
\end{align}
\end{lemma}

\section{Upper Bound} \label{sec:upper_bound}

Here we construct an upper bound on the minimax rate for bounded $K$. In general, the proof proceeds as follows: first we construct a tree structure such that each level of the tree forms a packing set on $K$. The radius of the packing sets decreases geometrically with the tree depth. Then we introduce the robust estimator which is applied to each layer of the tree. We run a tournament via the estimator which ensures the error rate is bounded at each layer. We can then apply a union bound over the tree to calculate an upper bound on the minimax rate.

\subsection{Algorithm}\label{subsec:algorithm}

\begin{algorithm}[htbp]
    \renewcommand{\algorithmicrequire}{\textbf{Input:}}
    \renewcommand{\algorithmicensure}{\textbf{Output:}}
    \caption{Directed Tree Construction}
    \label{algorithm:tree_construction}

   \begin{algorithmic}[1]
        \REQUIRE Set $ K $, Root node $ s\in K $, diameter $ d>0 $, constant $ c $ chosen large enough;
        \ENSURE A directed tree structure $ \mathcal{G} $ where each layer $ J $ is denoted $ \mathcal{L}(J) $;

        \STATE $ \mathcal{L}(1)\leftarrow \{s\} $;

        \STATE Form a maximal $ (d/c) $-packing set of $ B(s,d)\cap K $. Add directed edges from $ s $ to each node in the packing set denoted $ \mathcal{O}(s) $;
        \STATE $ \mathcal{L}(2)\leftarrow \mathcal{O}(s)   $;

        \FOR{$ J = 3,4,\ldots $}
            \FORALL{$ u \in \mathcal{L}(J-1)  $}
                \STATE Form a maximal $ \left(d/2^{J-1}c\right) $-packing set of $ B\left(u, d/2^{J-2}\right) \cap K $. 
                \STATE Add directed edges from $ u $ to each node in $ \mathcal{O}(u) $;
            \ENDFOR
            \STATE $ \mathcal{U}_J \leftarrow $ lexicographically ordered set of all offspring nodes $ \mathcal{O}(u) $ for all $ u\in \mathcal{L}(J-1) $;

            \WHILE{$ \mathcal{U}_J \neq \varnothing $}
                \STATE Pick first element $ u_l^J $, of $ \mathcal{U}_J $;
                \STATE $ \mathcal{T}_J(u_l^J) \leftarrow \{u_k^J\in \mathcal{U}_J:\, \left\| u_k^J-u_l^J \right\| \leq \frac{d}{2^{J-1}c},\, k\neq l \} $;
                \FOR{$ u_k^J \in \mathcal{T}_J(u_l^J) $}
                    \STATE Remove edge $ \mathcal{P}(u_k^J) \to u_k^J $ and node $ u_k^J $ from $ \mathcal{G} $
                    \STATE Add edge $ \mathcal{P}(u_k^J) \to u_l^J $;
                \ENDFOR
                \STATE Remove $ \{u_l^J\} \cup \mathcal{T}_J(u_l^J) $ from $ \mathcal{U}_J $;
            \ENDWHILE

            \STATE $ \mathcal{L}(J)\leftarrow \bigcup_{u\in \mathcal{L}(J-1)} \mathcal{O}(u) $;            
        \ENDFOR
        \RETURN $ \mathcal{G} $.
    \end{algorithmic} 
\end{algorithm}

We begin by stating several properties of the pruned tree whose construction is described in \Cref{algorithm:tree_construction}. This lemma is an encapsulation of lemmas 3.1-3.3 in \textcite{prasadan_information_2025}. For proofs we refer the reader to that paper.

\begin{lemma}[Properties of the pruned tree construction] \label[lemma]{lemma:tree_properties}
    Let $G$ be the pruned tree constructed from algorithm 1. Then the following properties hold: 
    \begin{enumerate}[label=(\roman*)]
        \item For any $J>3$, $\mathcal{L}(J)$ is a $\frac{d}{2^{J-1}c}$-packing and a $\frac{d}{2^{J-2}c}$-covering of $K$
        \item For any $ J\geq 2 $, and for any parent node $ \mathcal{Y}_{J-1}\in \mathcal{L}(J-1) $, its offspring $ \mathcal{O}(\mathcal{Y}_{J-1}) $ form a $ \frac{d}{2^{J-2}c} $-covering set of $ B(\mathcal{Y}_{J-1},\frac{d}{2^{J-2}}) \cap K $. Also, its cardinality is bounded as \[\mathrm{ card }(\mathcal{O}(\mathcal{Y}_{J-1})) \leq \mkloc\left(\frac{d}{2^{J-2}},2c\right) \]
        \item For any $ J\geq 2 $, we have
        \begin{align*}
            \mathrm{ card }\left( \mathcal{L}(J-1)\cap B\left(\Sigma ,\frac{d}{2^{J-2}}\right) \right) \leq   \mkloc\left(\frac{d}{2^{J-2}},c\right) \leq \mkloc\left(\frac{d}{2^{J-2}},2c\right);
        \end{align*} \label{tree_property_3}
        \item Let $ [\mathcal{Y}_1,\mathcal{Y}_2,\ldots] $ be a path in $ \mathcal{G} $, i.e. $ \mathcal{Y}_{J+1}\in \mathcal{O}(\mathcal{Y}_J) $, $ \forall J\in \mathbb{N}^+ $. Then we have
        \begin{align*}
            \left\Vert \mathcal{Y}_{J'} - \mathcal{Y}_J \right\Vert_F \leq \frac{d(2+4c)}{c2^{J'}},\quad J\geq J'\geq 1. 
        \end{align*} \label{tree_property_4}
    \end{enumerate}
\end{lemma}

The key to constructing an upper bound for the minimax estimation error of covariance matrices in the Frobenius norm is to set up a series of hypothesis-testing problems. For each of these problems, two candidate points $\left(\Sigma_1,\Sigma_2 \in \R^{n\times n}\right)$ are considered. We aim to set up a test that chooses the point which is more representative of the estimated covariance matrices. Then, we can choose points in a local packing set of each layer in the pruned tree constructed via \Cref{algorithm:tree_construction}. The test is constructed as follows: given $\Sigma_1,\Sigma_2 \in \R^{n\times n}$ which are suitably separated $\left(\|\Sigma_1 - \Sigma_2\| \geq C\delta \textrm{ for a constant } C>0\right)$, we want to formally test whether the true $\Sigma$ belongs to either $B(\Sigma_1,\delta)$ or $B(\Sigma_2,\delta)$, as follows:
\begin{align*}
    \psi=\psi_{(\Sigma_1,\Sigma_2)} = \begin{cases}
        0,& \text{if } \Sigma \in B(\Sigma_1,\delta)\\
        1,& \text{if } \Sigma \in B(\Sigma_2,\delta)
    \end{cases} 
\end{align*}
We take the opportunity here for a digression to discuss the parameter $\delta$. Choosing this parameter correctly is imperative for our results - indeed, $\delta^{2}$ is itself one of the bounds in our minimax rate. It will become clear why the value of $\delta$ should be chosen in such a way through the proofs presented later, but for now it suffices to note that $\delta$ should be the largest value such that $N \deltafunc < \log M^{\operatorname{loc}}(\delta,c)$.

In order to formally measure the difference between each candidate point and the data we introduce two maps which take the candidate points and return a real number. These maps provide the backbone of our proof of the upper bound. The map $\tilde{V}_i$ uses the uncorrupted data while $V_i$ uses the corrupted data. Notice that for subjects without corrupted data, of which there are $N(1-\varepsilon)$, $\tilde{V}_i = V_i$. Thus, we can proceed as though the $V_i$s are corrupted directly and reduce the problem to a robust mean estimation problem (though with some important characteristics of the data which will be discussed later). We define these maps as follows: 
\begin{align}\label{equation:definition_V_i}
    \tilde{V}_i:= \dfrac{ \left\Vert \tilde{\xi}_i\tilde{\xi_i}^\top -\Sigma_1 \right\Vert _F^2 - \left\Vert \tilde{\xi}_i\tilde{\xi_i}^\top -\Sigma_2 \right\Vert _F^2  }{ \left\Vert \Sigma_2-\Sigma_1 \right\Vert _F}, \qquad V_i:= \dfrac{ \left\Vert \xi_i\xi_i^\top -\Sigma_1 \right\Vert _F^2 - \left\Vert \xi_i\xi_i^\top -\Sigma_2 \right\Vert _F^2  }{ \left\Vert \Sigma_2-\Sigma_1 \right\Vert _F }
\end{align} where $i$ is taken over the set of subjects $[N]$.
Notice that the sign of $V_i$ indicates whether the data from the $i^{\textrm{th}}$ subject provides evidence for $\Sigma_1$ ($V_i < 0$) or $\Sigma_2$ ($V_i > 0$). This will be key in the construction of the two tests which we will utilize. We will apply either a median test or the trimmed mean test of \textcite{lugosi_robust_2021} based on the ratio of the local radius (see \Cref{def:local_metric_entropy}) to the sub-Exponential parameter of $\tilde{V}$ (see \Cref{lemma:V_i_bound}). We combine these estimators to define our robust estimator. We will prove error rate controls for each estimator independently first, and then combine them into a single rate for the robust estimator.

In order to simplify the notations and proofs, and with apologies for any confusion to the reader we refer to the total sample size as $2N$ for the remainder of this section, assuming that the sample size is even. We digress here to note why this assumption is reasonable. If $N$ is odd, then we can add one sample $X_{N+1} = \mathbf{0}_n$, which we consider a corrupted sample. Then our new corruption rate is at most $\frac{\varepsilon N + 1}{N + 1} = \varepsilon\big(\frac{N}{N+1}\big) + \frac{1}{N+1} < \varepsilon+\frac{1}{N}$. Therefore, if the sample size is odd we simply reduce the $\varepsilon$ for which our result holds to $\varepsilon < 1/32 - 1/(N+1)$. Due to our sample size condition ($N \geq \sup_\delta \log \mkloc(\delta,2c)$) this adjustment is negligible.

Before we can introduce our tests, we discuss the trimmed mean estimator
$\mathrm{TM}_{\delta_0}(\{\tilde{V}_i\}_{i=1}^{2N})$. We omit discussion of the bound calculation and instead refer the reader to \textcite{lugosi_robust_2021} for details. The trimmed mean estimator is a robust estimator which takes into account both the level of corruption $\varepsilon$ and a type I error bound $\delta_0 \asymp \expfunc{-N\delta^2}$ to be defined explicitly later. This estimator has the key property that with probability $1-\delta_0$ $|\mathrm{TM}_{\delta_0} - \E \tilde{V}_i|$ has a desirable bound (see \eqref{eq:tm_bound1}). Therefore $\E\tilde{V}_i \leq 0$ when $\Sigma$ is close to $\Sigma_1$, so the trimmed mean test can rely on whether $\mathrm{TM}(\{\tilde{V}_i\}_{i=1}^{2N})$ is positive or negative.

Essentially, the trimmed mean estimator is a two-step process. First, the corruption level $\varepsilon$, confidence level $\delta_0$ and half of the data $\left(\textrm{without loss of generality }\{V_i\}_{i=1}^{N}\right)$ are used to set cutoffs $\alpha$ and $\beta$. Second, the trimmed mean estimator is calculated from the remaining data 
\begin{align*}
    \mathrm{TM}_{\delta_0}(\{V_i\}_{i=1}^{2N}) = \frac{1}{N} \sum_{i=N+1}^{2N} \psi_{\alpha,\beta}(V_i)
\end{align*} for 
\begin{align*}
    \psi_{\alpha,\beta}(x) &= \begin{cases}
        \beta & x > \beta, \\ x & x \in [\alpha,\beta] \\ \alpha & x < \alpha.
    \end{cases}
\end{align*}

We now formally introduce our two tests. Notice that two equivalent definitions of the test estimator are given, one in terms of the $X_i$s and one in terms of the $V_i$s. Note that while we refer to this as the median test, this is an abuse of notation as we do not average the two middle observations. %Instead, when we refer to the median what we are actually referring to is the $N^{th}$ order statistic.

\begin{definition}[Median test]\label{definition:median_test}
    Given the data $\{X_i\}_{i=1}^{2N}$ (or $\{V_i\}_{i=1}^{2N}$), the median test returns:
    \begin{align*}
        \psi_\mathrm{Med}(X_i) &\coloneqq \indicator{\operatorname{card}\left(\left\{ i \in [2N] : \|X_iX_i\T - \Sigma_1\|_F \geq \|X_iX_i\T - \Sigma_2\|_F \right\}\right) \geq N}
    \end{align*}
\end{definition}

%{\color{red} Striclty speaking this is not a median right? Median of $2N$ numbers is the average of the two middle observations not the the previous definition... so you have to say you are abusing notation here to redefine median to something else? Did Tuorui do that also?}
%{\color{blue} I added a note, but I also changed anywhere we refer to it as the median outside the name "median test" to be the explictly correct version listed above. The explanation is similar to Tuorui's. Do you think we should get rid of the word median entirely and rename this to something like "order statistic test"? On the one hand, that's kind of vague and though median is technically incorrect it gets at the spirit of the test more precisely. On the other hand, as you noted, it's not a median...}
%{\color{red} maybe we can just call the $N$-th order statistic ``lower median'' which sounds like the proper name for this, and keep the test called median test. then we don't have to call that abuse of notation even.}

\begin{definition}[Trimmed mean test]\label{definition:trimmed_mean_test}
    Given the data $\{X_i\}_{i=1}^{2N}$ or $\{V_i\}_{i=1}^{2N}$, the trimmed mean test returns 1 if the trimmed mean of the $V_i$ is positive:
    \begin{align*}
        \psi_{\operatorname{TM}_{\delta_0}}(\{V_i\}_{i=1}^{2N}) &\coloneqq \indicator{\mathrm{TM}_{\delta_0}(\{V_i\}_{i=1}^{2N}) > 0}
    \end{align*}
\end{definition}

We take this opportunity to state that which will be later proven: we can only prove tail bounds for the median test when $\deltafunc$ is large or for the trimmed mean test when $\deltafunc$ is small. These two tests thus form a good basis for a combined test which chooses between the two sub-tests based on the value of $\deltafunc$.

\begin{definition}[Combined test]\label[definition]{definition:combined_test}
    Given the data $\{X_i\}_{i=1}^{2N}$ or $\{V_i\}_{i=1}^{2N}$ where $X = {X_1,...,X_{2N}}$ for $X_i = \mu + \xi_i$ and a pair of points $\Sigma_1,\Sigma_2 \in \R^{n\times n}$ such that $\|\Sigma_1 - \Sigma_2\|_F \geq C\delta$ for large enough constant $C > 2$, the combined robust test is
    \begin{align*}
        \psi_{\Sigma_1,\Sigma_2,\delta;\delta_0,D_4}(\{X_i\}_{i=1}^{2N}) = \begin{cases}
            \psi_{\mathrm{TM}_{\delta_0}}(\{V_i\}_{i=1}^{2N}) & \deltafunc \leq \thenewerboundary \\ 
            \psi_{\mathrm{Med}}(\{V_i\}_{i=1}^{2N}) & \deltafunc > \thenewerboundary
        \end{cases}
    \end{align*}
\end{definition}

When the context is clear, we simplify the notation for the combined test to $\psi_{\Sigma_1,\Sigma_2}$. To formalize the intuition of the combined test we define the following notation:

\begin{definition}
    Given a pair of points $\Sigma_1,\Sigma_2 \in \R^{n\times n}$, if $\psi_{\Sigma_1,\Sigma_2} = 0$ we say $\Sigma_1$ dominates $\Sigma_2$ and write $\Sigma_1 \succ \Sigma_2$. If the reverse is true we say $\Sigma_2$ dominates $\Sigma_1$ and write $\Sigma_1 \prec \Sigma_2$. 
\end{definition}

Note that the intuitive understanding of the preceding definition is that if $\Sigma_1 \succ \Sigma_2$ then $\Sigma_1$ is understood to better fit the data than $\Sigma_2$. Before we can prove the error bounds for these tests we must state one final definition and its corresponding algorithm.

\begin{definition}[Tournament operator]{\label[definition]{definition:tournament_T}}
    Let $S = \{\Sigma_1,...,\Sigma_M\} \subseteq K$ with cardinality $M$. For any $C>2$ denote $\bar{S} \coloneqq \left\{\Sigma^\prime \in S: \Sigma^\prime \succ \Sigma \textrm{ and } \|\Sigma - \Sigma^\prime\| \geq C\delta \right\}$.  Then we define the tournament operator as 
    \begin{align*}
        T(\delta,\Sigma,S) \coloneqq \begin{cases}
            \max\limits_{\Sigma^\prime \in \bar{S}} \|\Sigma - \Sigma^\prime\| & \bar{S} \neq \emptyset
            \\ 0 & \textrm{otherwise}
        \end{cases}
    \end{align*}
\end{definition}

Thus, the tournament operator gives the maximal distance between the point $\Sigma \in S$ and the points in $S \setminus B(\Sigma,C\delta)$ which dominate $\Sigma$. Therefore $\argmin_{\Sigma \in S} T(\delta,\Sigma,S)$ is the $\Sigma \in S$ with the smallest worst-case distance from the ``better" points. We utilize this operator in Algorithm \ref{algorithm:robust_algorithm}. Intuitively, in Algorithm \ref{algorithm:robust_algorithm} we work our way down the tree generated by Algorithm \ref{algorithm:tree_construction}, at each level we select the next node using the tournament operator with $S$ corresponding to the offspring of the current node. That is, we select the next step in our path as the offspring node with the smallest worst-case distance from all dominating points in the set of offspring nodes. Thus, Algorithm \ref{algorithm:robust_algorithm} generates a path along the tree generated by Algorithm \ref{algorithm:tree_construction}. At this point, the reader may ask why such a tournament operator is necessary---can't we just pick the point which is dominated by all others? The answer is, of course, no: the domination operation is not transitive. It may always be possible to find points such that $\Sigma_1 \prec \Sigma_2 \prec \Sigma_3 \prec \Sigma_1$. 

\begin{algorithm}[htbp]
    \renewcommand{\algorithmicrequire}{\textbf{Input:}}
    \renewcommand{\algorithmicensure}{\textbf{Output:}}
    \caption{Robust Tournament Algorithm for an Upper Bound}
    \label{algorithm:robust_algorithm}

    \begin{algorithmic}[1]
        \REQUIRE Directed tree $ \mathcal{G} $ with root $ \mathcal{Y}_1 $;
        \ENSURE A path $ \mathcal{Y}=[\mathcal{Y}_1,\mathcal{Y}_2,\ldots] $ in $ \mathcal{G} $;

        \STATE Initialize $ J=1 $ and $ \mathcal{Y} = [\mathcal{Y}_1] $;
        \FOR{$ J=1,2,\ldots $}
            \STATE $ \mathcal{Y}_{J+1}\leftarrow \mathop{ \argmin }\limits_{\Sigma \in \mathcal{O}(\mathcal{Y}_J)} T\left( \frac{d}{2^J(C+1)}, \Sigma ,\, \mathcal{O}(\mathcal{Y}_J) \right) $, breaking ties lexicographically;
            \STATE $ \mathcal{Y}.\text{append}(\mathcal{Y}_{J+1}) $;
        \ENDFOR
        \RETURN $ \mathcal{Y}=[\mathcal{Y}_1,\mathcal{Y}_2,\ldots] $.
    \end{algorithmic}
\end{algorithm}

\subsection{Estimation error bounds}

Now, with our algorithm fully explained, we can provide upper bounds for the error of the estimator generated by the algorithm. These bounds will be (nearly) the same as the lower bound proved previously, resulting in an minimax error rate $\left.\vphantom{\log^2}(\right.$up to a small factor of $\left.\log^2(1/\varepsilon)\right)$. Before we begin, let us provide a brief roadmap for how we will proceed. First, we will prove probabilistic bounds on the type I error rate for each of the median and trimmed mean tests. Then, we will use these to prove a bound on the type I error of the combined robust test. Finally, we can use a union-bound argument to provide a probabilistic bound on the error rate of the estimator generated by running the tournament described in \Cref{subsec:algorithm}. 

We begin by showing that the uncorrupted data $\tilde{V}_i$ is sub-exponential. This property is the key to proving the upper bound of the trimmed mean estimator.

\begin{lemma}[Sub-exponential property of $\tilde{V}_i$]\label[lemma]{lemma:V_i_bound}
    For the uncorrupted test data $\tilde{V}_i - \E\tilde{V}_i$ is sub-Exponential with constant parameter. Furthermore if the true variance $\Sigma \in B(\Sigma_1,\delta)$ and \\$\|\Sigma_1-\Sigma_2\|_F \geq C\delta$ then for $\hat{v} \coloneqq \frac{\operatorname{Vec}(\Sigma_2 -\Sigma_1)}{\|\Sigma_2-\Sigma_1\|_F}$
    \begin{align*}
        \tilde{V}_i \leq 2\hat{v}^\top\operatorname{Vec}(\tilde{\xi}_i\tilde{\xi}_i^\top - \Sigma) - (C-2)\delta.
    \end{align*}
\end{lemma}
\begin{delayedproof}{lemma:V_i_bound}
    First, note that 
    \begin{align}
    \tilde{V} - \E \tilde{V} &= \dfrac{ \left\Vert \tilde{\xi}\tilde{\xi}^\top -\Sigma_1 \right\Vert _F^2 - \left\Vert \tilde{\xi}\tilde{\xi}^\top -\Sigma_2 \right\Vert _F^2  }{ \left\Vert \Sigma_2-\Sigma_1 \right\Vert _F} - \E \tilde{V}\nonumber
    \\ &= \frac{2\tr\left(\left(\Sigma_2-\Sigma_1\right)\T\left(\txi\txi\T - \Sigma\right)\right) + \|\Sigma - \Sigma_1\|_F^2 - \|\Sigma - \Sigma_2\|_F^2}{\|\Sigma_2-\Sigma_1\|} - \E \tilde{V}\nonumber
    \\ &= \frac{2\tr\left(\left(\Sigma_2-\Sigma_1\right)\T\left(\txi\txi\T - \Sigma\right)\right)}{\|\Sigma_2-\Sigma_1\|} \nonumber
    \\ &= \frac{2\langle \Sigma^{1/2}(\Sigma_2 - \Sigma_1)\Sigma^{1/2}, (\Sigma^{-1/2} \tilde{\xi} \tilde{\xi}^{\top} \Sigma^{-1/2} - \mathbb{I})\rangle}{\|\Sigma_1 - \Sigma_2\|_F}. \label{eq:v_minus_ev} 
    \end{align} 
     
    Let $O \coloneqq \Sigma^{1/2}(\Sigma_2 - \Sigma_1)\Sigma^{1/2} / \|\Sigma_1 - \Sigma_2\|_F$ and take its eigendecomposition $O = U \Lambda U\T = \sum_i \lambda_iu_iu_i\T$. Then we can rewrite \Cref{eq:v_minus_ev} as $2\langle O, \eta\eta\T - \mathbb{I}\rangle =2 \sum_i \lambda_i((u_i\T \eta)^2 - 1)$ for $\eta = \Sigma^{-1/2} \tilde\xi$. Note that $u_i\T \eta$ are i.i.d. standard normal random variables and
    \[\sum_i \lambda_i^2 = \|O\|^2_F \leq \|\Sigma^{1/2}\|_{op}^2 \left\|\frac{\Sigma_1 - \Sigma_2}{\|\Sigma_1 - \Sigma_2\|_F}\right\|_F^2\|\Sigma^{1/2}\|_{op}^2\leq c_1^2 \lesssim 1.\] 
    Therefore, by \textcite[Proposition 5.16]{vershynin_introduction_2011} $\tilde{V} - \E \tilde{V} = 2\hat{v}\T \operatorname{Vec}\left(\tilde{\xi}_i\tilde{\xi_i}^T - \Sigma\right) =2 \langle O, \eta\eta\T - \mathbb{I}\rangle$ is sub-Exponential of constant order. Furthermore, since $\E \tilde{V}_i$ is a constant, $\tilde{V}_i$ is also sub-Exponential.
    Now, if $\Sigma \in B(\Sigma_1,\delta)$ and $\|\Sigma_1-\Sigma_2\|\geq C\delta$ then 
    \begin{align*}
        \tilde{V} &= \frac{2\tr\left(\left(\Sigma_2-\Sigma_1\right)\T\left(\txi\txi\T - \Sigma\right)\right) + \|\Sigma - \Sigma_1\|_F^2 - \|\Sigma - \Sigma_2\|_F^2}{\|\Sigma_2-\Sigma_1\|_F} 
        \\ &= 2\hat{v}\T\operatorname{Vec}\left(\txi\txi\T - \Sigma\right) + \frac{\|\Sigma - \Sigma_1\|_F^2 - \|\Sigma - \Sigma_2\|_F^2}{\|\Sigma_2 -\Sigma_1\|_F}
        \\ &\leq 2\hat{v}\T\operatorname{Vec}\left(\txi\txi\T - \Sigma\right) - (C-2)\delta
    \end{align*}
\end{delayedproof}

\begin{definition}
    For ease of notation, we define $\Upsilon \coloneqq C_{SE} c_1 \geq  \sup_{\|A\|_F=1}\|\langle A,\tilde\xi\tilde\xi\T-\Sigma\rangle\|_{\psi_1}$ as an upper bound on the sub-Exponential norm. 
\end{definition} 

In \Cref{lemma:D4_bounded} we show that $\Upsilon$ is of constant order. Now we can bound the type I error of the median test and the trimmed mean test.  The proofs of each of these lemmas are given in \Cref{appendix:upper_bound_pfs}.

\begin{lemma}[Trimmed mean test error rate]\label[lemma]{lemma:trimmed_mean}
There exist constants $ C>2, C_1>0,$ and $C_2>0$ such that: if $ \varepsilon \in (0,{1/32}) $, for any $ C\delta$- separated $ \Sigma _1,\Sigma _2\in K $ with $ \Sigma \in B(\Sigma _1,\delta ) $ and with $N^{-1}\vee C^{2}_1\varepsilon^2\log^2(1/\varepsilon) \leq \deltafunc \leq \thenewerboundary $. We have
    \begin{align*}
        \Prob_{\Sigma\in B(\Sigma_1,\delta)} \left( \mathrm{ TM } (\{V_i\}_{i=1}^{2N}) \geq 0  \right) \leq \expfunc{ -C_2N\deltafunc} .
    \end{align*}
\end{lemma}

The \proofref{lemma:trimmed_mean} of this lemma is relegated to \Cref{appendix:upper_bound_pfs}.

\newcommand{\newRHS}{\min\left(\frac{\delta^2}{\Upsilon^2}, \frac{\delta}{\Upsilon}\right)}

\begin{lemma}[Median test error rate] \label[lemma]{lemma:median}
    There exist constants $ C>2, C_4 >0$ such that: if $ \varepsilon \in (0,{1/32}) $, for any $ C\delta$- separated $ \Sigma _1,\Sigma _2\in K $ with $ \Sigma \in B(\Sigma _1,\delta ) $ and, $\thenewerboundary \leq \deltafunc$ we have
    \begin{align*}
        \Prob_{\Sigma \in B(\Sigma_1,\delta)} \big(\operatorname{card}\big(\big\{ i \in [2N] : \|X_iX_i\T - \Sigma_1\|_F \geq \|X_iX_i\T - &\Sigma_2\| \big\}\big) \geq N  \big) \\&\leq \expfunc{ -C_4N\newRHS}.
    \end{align*}
\end{lemma}

The \proofref{lemma:median} of this lemma is also relegated to \Cref{appendix:upper_bound_pfs}.

Finally, with type I error rates for our two tests given, we can define an overall hypothesis testing problem using the combined robust test $\psi_{\Sigma_1,\Sigma_2}$ and bound its type I error rate. This is given in \Cref{theorem:two_point_prob_bound}.

\begin{theorem}\label{theorem:two_point_prob_bound}
    There exist constants $C>2$ and $C_5>0$ such that the following type I error rate bound holds. Let $ \varepsilon \in (0,1/32)$ and consider the test
    \begin{align*}
        H_0:\, \Sigma \in B(\Sigma _1,\delta ) \text{ v.s. } H_A:\, \Sigma \in B(\Sigma _2,\delta ) \text{ for } \left\Vert \Sigma _1-\Sigma _2 \right\Vert \geq C\delta 
    \end{align*}
    where $ \Sigma _1,\Sigma _2\in K $. Let $\psi_{\Sigma_1,\Sigma_2}$ be the combined robust test from \Cref{definition:combined_test}. Then if $\delta$ satisfies $\deltafunc \geq N^{-1}\vee C_1^2\epsilon^2\log^2(1/\varepsilon)$ and $N \gtrsim \sup_{\delta > 0} \log \mkloc(\delta,2c)$, we have
    \begin{align*}
        \mathop{ \sup }\limits_{\Sigma :\left\Vert \Sigma-\Sigma _1 \right\Vert \leq \delta } \Prob_\Sigma \left( \psi=1 \right)  \vee\mathop{ \sup }\limits_{\Sigma :\left\Vert \Sigma-\Sigma _2 \right\Vert \leq \delta } \Prob_\Sigma \left( \psi=0 \right) \leq &  \expfunc{ -C_5N\newRHS}
    \end{align*}
\end{theorem}

\begin{proof}
    Let $ C, C_1$, and $C_2$ be the universal constants from \Cref{lemma:trimmed_mean}, $ C_4 $ from \Cref{lemma:median}, $D_5$ from \Cref{lemma:useful_constants} and let $C_5 \coloneqq C_2 \wedge C_4$.
    
    We use the robust combined estimator $ \psi= \psi_{\Sigma _1,\Sigma _2, \delta ;\delta _0,C_2} $ defined in \Cref{definition:combined_test}, substituting in the corresponding constants. We break the probability bounds into two cases based on the value of $\deltafunc$, as discussed previously.
    \begin{itemize}
        \item If $\deltafunc \leq \thenewerboundary$ then $\deltafunc \geq N^{-1}\vee C_1^2\epsilon^2\log^2(1/\varepsilon)$ ensures that \Cref{lemma:trimmed_mean} is applicable. Then 
        \begin{align*}
            \Prob_{\Sigma \in B(\Sigma_1,\delta)}\left(TM(\{V_i\}_{i=1}^{2N}\geq 0\right) &\leq \expfunc{-C_2N\deltafunc}
            \\ &\leq \expfunc{-C_5 N \newRHS}.
        \end{align*}
        \item If $\deltafunc > \thenewerboundary$ then by \Cref{lemma:median}
        \begin{align*}
            \Prob_{\Sigma\in B(\Sigma_1,\delta)} \left(\mathrm{ Median }\big( \{V_i\}_{i=1}^{2N} \big) \geq 0 \right) &\leq \expfunc{-C_4 N\newRHS}
            \\ &\leq \expfunc{-C_5 N \newRHS}
        \end{align*}
    \end{itemize}
    First notice that the two cases have the same bound and thus can be combined. Also, the case of $\Sigma \in B(\Sigma_2,\delta)$ is analogous.  Thus
    \begin{align*}
        \mathop{ \sup }\limits_{\Sigma :\left\Vert \Sigma-\Sigma _1 \right\Vert \leq \delta } \Prob_\Sigma \left( \psi=1 \right)  \vee\mathop{ \sup }\limits_{\Sigma :\left\Vert \Sigma-\Sigma _2 \right\Vert \leq \delta } \Prob_\Sigma \left( \psi=0 \right) \leq &  \expfunc{-C_5N \newRHS}
    \end{align*}
\end{proof}

Now that we have established a type-I error probability bound, we can use the union bound to bound the error for the winner of the tournament \Cref{algorithm:robust_algorithm}.  

\begin{theorem} \label{theorem:covering_bound}
    Let $ C,C_1,$ and $C_5$ be the universal constants from \Cref{theorem:two_point_prob_bound} and \\ $ S=\{\Sigma _1,\ldots,\Sigma _M\} $ be a $ \delta  $-covering of $ K'\subset K $ which has size $M$. Consider the tournament $ T(\delta ,\Sigma ,S) $ from \Cref{definition:tournament_T}. Assume $ \Sigma \in K' $ and denote $ i^*\in \mathop{ \arg\min }\limits_{i\in [M]} T(\delta ,\Sigma _i,S)  $. Also assume $ \frac{\delta^2}{\vVar^2} \geq N^{-1}\vee C_1^2\varepsilon^2\log^2(1/\varepsilon)  $, then
    \begin{align*}
        \Prob\left( \left\Vert \Sigma _{i^*} - \Sigma  \right\Vert \geq (C+1)\delta  \right) \leq M\expfunc{-C_5N\newRHS} 
    \end{align*}
\end{theorem}
\begin{proof}
    Since we have that $ S=\{\Sigma _i\}_{i=1}^M $ is a $\delta$-covering of $ K'\ni \Sigma  $, we can without loss of generality assume $ \Sigma _1 $ be the point in the covering closest to $ \Sigma  $ and thus $ \Sigma \in B(\Sigma _1, \delta ) $. Denote $ T_i:= T(\delta ,\Sigma _i,S) $. We have by definition of $ T(\, \cdot \, ) $ that 
    \begin{align*}
        \indicator{\left\Vert \Sigma _{i^*} -\Sigma _1 \right\Vert \geq C\delta } \leq \indicator{\max\{ T_{i^*},T_1 \} \geq C\delta } \leq \indicator{T_1 \geq C\delta } .
    \end{align*}
    Therefore, if $T_1 \geq C\delta$ then $\exists\Sigma_j$ such that $\Sigma_j \succ \Sigma_1$. Because $\frac{\delta^2}{\vVar^2} \geq N^{-1} \vee C_1^2\varepsilon^2\log^2(1/\varepsilon)$, we can apply \Cref{theorem:two_point_prob_bound}. Then, applying the union bound to the $M-1$ possible $\Sigma_j$
    \begin{align*}
        \Prob_\Sigma(\|\Sigma_{i^\star} - \Sigma_1\| \geq C\delta) &\leq \Prob_\Sigma(T_1 \geq C \delta) \leq (M-1) \expfunc{-C_5 N \newRHS},
    \end{align*}
    and by the triangle inequality 
    \begin{align*}
        \Prob_\Sigma(\|\Sigma_{i^\star} - \Sigma\| \geq (C+1) \delta) \leq M\expfunc{-C_5 N \newRHS}.
    \end{align*}
\end{proof}

\begin{theorem} \label{theorem:upper_bound}
    Let $C$ and $C_5$ be the universal constants from \Cref{theorem:two_point_prob_bound}. Set {$c=2(C+1)$} as the constant in the local metric entropy $\mkloc(\delta,c)$ and $\mathcal{Y}$ as the winning path from \Cref{algorithm:robust_algorithm}. Define $\delta_J \coloneqq \frac{d}{2^{J-1}(C+1)}$  and let $J^\star \geq 1$ be the largest $J$ such that 
    \begin{align}
        C_5 N \min\left(\frac{\delta_J^2}{\Upsilon^2},\frac{\delta_J}{\Upsilon}\right) \geq 4\log\mkloc(c\delta_J,2c)\bigvee \log 2. \label{eq:2log_lower_bound}
    \end{align}
    If this condition is never satisfied set $J^\star = 1$. Denote the output after some step $J^{\star\star} \geq J^\star$ as $\Sigma^{\star\star} = \mathcal{Y}_{J^{\star\star}+1}$. Finally, assume that $N \gtrsim \sup_{\delta > 0} \log \mkloc(\delta,2c)$.
    \\ Then
    \begin{align}
        \E_\xi\|\Sigma^{\star\star} - \Sigma\|_F^2 \lesssim \left(\delta^2_{J^\star}\vee\varepsilon^2\log^2(1/\varepsilon)\right) \bigwedge d^2 \label{eq:upper_bound}
    \end{align}
\end{theorem}

The \proofref{theorem:upper_bound} can be found in \Cref{appendix:upper_bound_rate}.

\section{Combined minimax rate}

Finally, we can summarize our results into one combined minimax rate. Before we do so, we note that we omit $\Upsilon$ from the rates. This is because $\Upsilon$ can be bounded by a constant, as per \Cref{lemma:D4_bounded}, and therefore is absorbed by the constants.

\begin{theorem}\label{theorem:minimax_rate}
    Let $\varepsilon \in (0,\frac{1}{32})$ and $N\gtrsim \sup_\delta \log \mkloc(2\delta,c)$.  Define
    \begin{align*}
        \delta^\star=\sup\left\{ \delta\geq 0:\, N \frac{\delta^2}{\Upsilon^2} \leq \log \mkloc(\delta,c) \right\}
    \end{align*}
    Then for sufficiently large $c$, the minimax error is bounded as
    \begin{align}
        \left(\varepsilon^2 \vee \delta^{\star2} \right)\wedge d^2 \lesssim \mathfrak{M}(X) \lesssim \left(\varepsilon^2\log^2(1/\varepsilon) \vee \delta^{\star2} \right)\wedge d^2
    \end{align}
\end{theorem}

\begin{remark}
    A single minimax rate is not achieved for this setting - the left and right bounds differ by a factor of $\log^2(1/\varepsilon)$. However, this is a small price to pay. For example, when $\varepsilon = 0$, $\varepsilon^2 = \epsilon^2\log^2(1/\varepsilon)=0$. Indeed, whenever $\varepsilon$ is a constant this difference is of course negligible as it is just a constant. If $\varepsilon \asymp 1/n$ then the rate only differs by a factor of $\log^2 n$ which is of course dominated by the $1/n^2$ in the rate. This will be explored further in \Cref{sec:examples}.
\end{remark}

\begin{remark}
    Due to the argument presented in \textcite[Lemma 5.7]{wainwright_highdimensional_2019}, for $\delta>0$ we have that $\log \mkloc(\delta,2c) \lesssim n^2$. Therefore, $N \gtrsim n^2$ is sufficient for the sample size condition $N \gtrsim \sup_\delta \log \mkloc(\delta,2c)$. 
\end{remark}

\newcommand{\boundaryA}{8 \log 2}

\begin{proof}
    Note first that for $\Sigma^{\star\star}$ as the output of \Cref{algorithm:robust_algorithm} and $J^*\geq 1$ the maximal $J$ such that $ C_5 N \min\left(\frac{\delta_J^2}{\Upsilon^2},\frac{\delta_J}{\Upsilon}\right) \geq 4\log\mkloc(c\delta_J,2c)\bigvee \log 2$. We have \begin{align} \E_\xi\|\Sigma^{\star\star} - \Sigma\|_F^2 \lesssim \max\{\delta_{J^\star}^2,\varepsilon^2\log^2(1/\varepsilon)\} \wedge d^2\label{eq:upper_bound_deltaj}\end{align} from \Cref{theorem:upper_bound}.
    Also, without loss of generality assume $\delta^\star > 0$ (for reasoning see \textcite[proof of Theorem 4.4]{peng_robust_2026}). We consider the upper and lower bounds separately.
    \paragraph{Upper bound}\phantom{.}
    \\\textbf{Case 1: $N \frac{\delta^{\star2}}{\Upsilon^2} \leq \boundaryA$.} In this case we just need to see that $d \lesssim \delta^{\star}$ which follows directly from \textcite[proof of Theorem 3.10, case 3]{prasadan_information_2025}
    because
    \begin{align*}
        8 \log 2 \geq N \frac{\delta^{\star2}}{\Upsilon^2} = \frac{N}{4}\frac{(2\delta^\star)^2}{\Upsilon^2} \geq \frac{1}{4}\log \mkloc(2\delta^\star,c).
    \end{align*}
    Therefore we have the upper bound \begin{align*}
        \inf_{\hat\Sigma} \sup_{\Sigma} \sup_{\mathcal{C}} \E \| \hat\Sigma(\mathcal{C}(X)) - \Sigma \|_F^2 \leq d^2 \lesssim  \left(\varepsilon^2\log^2(1/\varepsilon) \vee \delta^{\star2}\right) \wedge d^2.
    \end{align*}
    \\\textbf{Case 2: $N \frac{\delta^{\star2}}{\Upsilon^2} \geq \boundaryA$.}
    \\We have from \Cref{theorem:upper_bound} that $\left(\delta_{J^\star}^2 \vee \varepsilon^2\log^2(1/\varepsilon)\right) \bigwedge d^2$ upper bounds the minimax rate. Thus, we can obtain the desired bound by finding $\tilde{\delta} \asymp\delta^\star$ such that $\tilde{\delta} \gtrsim \delta_{J^\star}$. 
    First, notice that if $N \gtrsim \sup_{\delta>0}\log \mkloc(\delta,2c)$ then $\frac{\delta^\star}{\Upsilon} \leq B$ for some fixed constant $B$ and 
    \[
        \min\bigg(G^2\frac{\delta^{\star2}}{\Upsilon^2},G\frac{\delta^\star}{\Upsilon}\bigg) \geq \min\bigg(G^2,\frac{G}{B}\bigg) \frac{\delta^{\star2}}{\Upsilon^2} 
    \]
    by the process laid out in \Cref{lemma:delta_J_123_rates}. Now let $\eta_G = \min\bigg(G^2,\frac{G}{B}\bigg)$ and $\tilde\delta \coloneqq G_c\delta^\star$ for some $G_c$ such that $G_c \geq \frac{4}{c} $ and $\eta_{G_c} \geq \frac{32}{C_5}$ and notice that
    \begin{align*}
        C_5N\min\left(\frac{\tilde{\delta}^2}{\Upsilon^2}, \frac{\tilde{\delta}}{\Upsilon}\right) &= C_5N\min\left(G_c^2\frac{\delta^{\star2}}{\Upsilon^2}, G_c\frac{\delta^\star}{\Upsilon}\right)
        \\ &\geq C_5N\eta_{G_c}\frac{\delta^{\star2}}{\Upsilon^2}
        \\ &\overset{(i)}{\geq} 8N\frac{(2\delta^{\star})^2}{\Upsilon^2}
        \\ &\overset{(ii)}{\geq} 8\log\mkloc(2\delta^\star,c)
        \\ &\overset{(iii)}{\geq} 8\log\mkloc\left(\frac{c}{2}G_c\delta^\star,c\right)
        \\ &\geq 8\log\mkloc\left(\frac{c}{2}\tilde{\delta},c\right)
        \\ &\overset{(iv)}{\geq} 4\log\mkloc\left(c\tilde{\delta},2c\right)
    \end{align*}
    where $(i)$ is due to the definition of $G_c$, $(ii)$ is due to the definition of $\delta^\star$, $(iii)$ is due to the definition $G_c$, plus the monotonicity of $\mkloc$, and $(iv)$ is due to \Cref{lemma:2r_bound}. Notice that this gives 
    \begin{align*}
        C_5N\min\left(\frac{\tilde{\delta}^2}{\Upsilon^2}, \frac{\tilde{\delta}}{\Upsilon}\right) \geq 8N\frac{(2\delta^{\star})^2}{\Upsilon^2} \geq 256\log 2 > \log 2.
    \end{align*}
    Thus, \Cref{eq:2log_lower_bound} holds for $\tilde\delta$. If $\tilde\delta \geq \delta_1$, set $\delta^\prime \coloneqq \tilde\delta$. Otherwise, choose $\delta^\prime$ as the $\delta_j$ in $[\tilde{\delta},2\tilde\delta)$. Note that this is a valid choice of $\delta$ because the $\min\big(\frac{\delta^2}{\Upsilon^2},\frac{\delta}{\Upsilon}\big)$ is increasing in $\delta$ while the local metric entropy is non-increasing in $\delta$. Now, because $J^\star$ is the largest $J$ such that \Cref{eq:2log_lower_bound} holds, we have that $\delta^\prime \gtrsim \delta_{J^\star}$ and applying \Cref{theorem:upper_bound} gives the desired inequality
    \begin{align*}
        \inf_{\hat\Sigma} \sup_{\Sigma} \sup_{\mathcal{C}} \E \| \hat\Sigma(\mathcal{C}(X)) - \Sigma \|_F^2 \lesssim \left(\varepsilon^2\log^2(1/\varepsilon) \vee \delta_{J^\star}^{2}\right) \wedge d^2 \lesssim  \left(\varepsilon^2\log^2(1/\varepsilon) \vee \delta^{\star2}\right) \wedge d^2.
    \end{align*}
    \paragraph{Lower bound}\phantom{.}
    \\\textbf{Case 3: $N \frac{\delta^{\star2}}{\Upsilon^2} \leq \boundaryA$ and $\varepsilon > \frac{1}{\sqrt{N}}$.} Following the same argument as \textcite[Theorem 3.10, case 3]{prasadan_information_2025} we have $d \leq 12 \delta^\star$. Now define $\tilde{\delta} \coloneqq \frac{c_0}{25\sqrt{C_0}\Upsilon} d \leq \frac{1}{2}\delta^\star$ (note $c_0 \leq \Upsilon \leq \sqrt{C_0}\Upsilon$) to get
    \begin{align*}
        \boundaryA \geq & N \bigg(\frac{\delta^\star}{\Upsilon}\bigg)^2  \geq N \bigg(\frac{d}{12\Upsilon}\bigg)^2 = N\bigg(\frac{25\Upsilon\tilde{\delta}\sqrt{C_0}/c_0}{12\Upsilon}\bigg)^2 > 4\bigg(\frac{\tilde{\delta}}{c_0}\bigg)^2NC_0 
    \end{align*}
    Then if we choose $c$ large enough such that $ \log M^{\operatorname{ loc }}\Big(\frac{c_0}{24\Upsilon}d,c\Big) \geq \boundaryA $, we have
    \begin{align*}
        \log M^{\operatorname{ loc }}(\tilde{\delta },c) =&\log M^\mathrm{ loc }\bigg(\frac{c_0}{25\sqrt{C_0}\Upsilon}d,c\bigg) \geq \boundaryA \geq \boundaryA \vee 4\frac{ C_0 }{ c_0 ^2 }N\tilde{\delta }^2  > 4\bigg( \dfrac{ C_0}{ c_0 ^2 }N\tilde{\delta }^2  \vee \log 2  \bigg)
    \end{align*}
    then by Lemma \ref{lemma:lower_bound_1} we have
    \begin{align*}
        \inf_{\hat\Sigma} \sup_{\Sigma} \sup_{\mathcal{C}} \E \lVert  \hat\Sigma(\mathcal{C}(X)) - \Sigma \rVert_F^2 \gtrsim \tilde{\delta }^2 \gtrsim d^2 \gtrsim d^2\wedge \delta ^{\star2},
    \end{align*}
    and by Lemma \ref{lemma:lower_bound_2} we have 
    \begin{align*}
        \inf_{\hat\Sigma} \sup_{\Sigma} \sup_{\mathcal{C}} \E \lVert  \hat\Sigma(\mathcal{C}(X)) - \Sigma \rVert_F^2 \gtrsim  \varepsilon^2 \wedge d^2.
    \end{align*}
    \textbf{Case 4:  $N \frac{\delta^{\star2}}{\Upsilon^2} \leq \boundaryA$ and $\varepsilon \leq \frac{1}{\sqrt{N}}$.}
       The bound on $\delta^{\star2}$ holds but the bound on $\varepsilon^2$ must be shown. We can utilize Lemma \ref{lemma:lower_bound_1} to do so by showing that $\tilde\delta = \frac{1}{3\sqrt{2}}\big(\frac{c_0}{\sqrt{C_0}}\varepsilon \wedge d\big)$ meets the conditions of that lemma, specifically that $\log \mkloc(\tilde\delta,c) > 4\bigg(\frac{C_0}{c_0^2}N\tilde{\delta}^2 \vee \log 2\bigg)$. We have that $4\frac{C_0}{c_0^2}N\tilde{\delta}^2 \leq \frac{4}{18}N\varepsilon^2 \leq \frac{2N}{9}\varepsilon^2 \leq \frac{2}{9}$. From \cite[proof of Theorem 3.10, case 4]{prasadan_information_2025} we have that we can obtain a $\frac{d}{3\sqrt{2}c}-$packing set of cardinality greater than $\exp(4 \log 2)$ for sufficiently large c. Thus, 
        \begin{align*}
            \log \mkloc\Bigg(\frac{1}{3\sqrt{2}}\bigg(\frac{c_0}{\sqrt{C_0}}\varepsilon \wedge d\bigg),c\Bigg) &\geq  \log \mkloc\bigg(\frac{d}{3\sqrt{2}},c\bigg) 
            \\ &> 4 \log 2 
            \\ & \geq 4 \log 2 \vee \frac{2}{9} 
            \\ &\geq 4 \bigg(\log 2 \vee \frac{C_0}{c_0^2}N\tilde\delta^2\bigg)
        \end{align*}
        and we have 
        \begin{align*}
            \inf_{\hat\Sigma} \sup_{\Sigma} \sup_{\mathcal{C}} \E \lVert  \hat\Sigma(\mathcal{C}(X)) - \Sigma \rVert_F^2 \gtrsim   \tilde\delta^2 \asymp \varepsilon^2 \wedge d^2.
        \end{align*}
    \\\textbf{Case 5: $N \frac{\delta^{\star2}}{\Upsilon^2} \geq \boundaryA$ and $\varepsilon > \frac{1}{\sqrt{N}}$.}
    \\ Let $\tilde\delta \coloneqq \Big(\frac{c_0}{4\sqrt{C_0}\Upsilon} \wedge \frac{1}{2}\Big)\delta^\star$. Note that
    \begin{align*}
        \log \mkloc\bigg(\tilde{\delta},c\bigg) &\geq \log \mkloc\Big(\frac{\delta ^\star}{2},c\Big)
        \\ &\geq \lim_{h \downarrow 0}\log \mkloc(\delta^\star - h,c)
        \\ &\geq \lim_{h \downarrow 0} N \frac{(\delta^\star - h)^2}{\Upsilon^2}
        \\ &= N \frac{\delta^{\star2}}{\Upsilon^2}
        \\ &\geq 8 \log 2 \vee N\frac{\delta^{\star2}}{\Upsilon^2} 
        \\ &> 4\bigg( \frac{C_0}{c_0 ^2}N\tilde{\delta}^{2} \vee \log 2 \bigg).
    \end{align*} 
    Then by Lemma \ref{lemma:lower_bound_1} we have
    \begin{align*}
        \inf_{\hat\Sigma} \sup_{\Sigma} \sup_{\mathcal{C}} \E \lVert  \hat\Sigma(\mathcal{C}(X)) - \Sigma \rVert_F^2 \gtrsim \tilde{\delta}^{2} \gtrsim \delta^{\star2} \gtrsim d^2\wedge \delta ^{\star2}.  
    \end{align*}
    And by Lemma \ref{lemma:lower_bound_2} we have 
    \begin{align*}
        \inf_{\hat\Sigma} \sup_{\Sigma} \sup_{\mathcal{C}} \E \lVert  \hat\Sigma(\mathcal{C}(X)) - \Sigma \rVert_F^2 \gtrsim  \varepsilon^2  \wedge d^2.
    \end{align*}
    \textbf{Case 6: $N \frac{\delta^{\star2}}{\Upsilon^2} \geq \boundaryA$ and $\varepsilon \leq \frac{1}{\sqrt{N}}$.}
    Similar to case 4, the bound on $\varepsilon^2$ must be shown. 
    Take $\tilde\delta = \Upsilon \varepsilon$. Since $\varepsilon \leq \frac{1}{\sqrt{N}}$, $N\frac{\tilde\delta^2}{\Upsilon^2} \leq 1$. On the other hand, $N \frac{\delta^{\star2}}{\Upsilon^2} \geq \boundaryA$. Thus, $\varepsilon \asymp \varepsilon \Upsilon = \tilde\delta < \delta^\star$. Since $\delta^{\star2}$ is a lower bound on the rate, $\varepsilon^2$ must be, too.
    \\ Combining the six cases gives the desired bound.
\end{proof}

\section{Examples}\label{sec:examples}

We have established the minimax rate and now illustrate its relevance to existing results in the literature through a series of examples. We begin with a corollary to our main theorem, in which we establish an upper bound on estimation for the uniformity class of covariance matrices with the $\ell_q$ ball. 

\subsection{Covariance matrices in the uniformity class with the $\ell_q$ ball} \label{subsec:lq_ball}

\begin{corollary}\label[corollary]{cor:lq_ball}
    Consider the uniformity class of matrices with the $\ell_q$ ball with bounded eigenvalues, $\mathcal{U}_q$. In this class each column of the covariance matrix $\Sigma$ is in the $\ell_q$ ball for $q \in [0,1]$: 
    \begin{align*}
        \mathcal{U}_q(s_0(n)) \coloneqq \left\{ \Sigma: \max_i \sum_{j=1}^n |\sigma_{ij}|^q \leq s_0(n) , 0 < c_0 \leq \lambda_{\min}(\Sigma) \leq \lambda_{\max}(\Sigma) \leq c_1 \right\} \hspace{2mm}
    \end{align*} for $s_0(n) \geq c_0^q$, $N \gtrsim n^2$, $n \geq 2$, and $X \sim N(0,\Sigma)$ such that $\Sigma \in U_q(s_0(n))$. 
    \\Then 
    \begin{align*}
        \frac{1}{n} \mathfrak{M}(X) \lesssim \frac{1}{n} \varepsilon^2 \log^2(1/\varepsilon)\bigvee \left(s_0(n) \left(\frac{\log n}{N}\right)^{1-q/2} \bigwedge \frac{n}{N}\right).
    \end{align*}
\end{corollary}

Readers may note that in the uncorrupted regime, this rate is analogous to that established by \textcite{bickel_covariance_2008}, where that paper established a rate of $s_0(n)\left(\frac{\log n}{N}\right)^{1-q/2}$.

\begin{lemma}\label[lemma]{lemma:Uq_star_shaped}
    For $q \in [0,2]$, the uniformity class of matrices with the $\ell_q$ ball with bounded eigenvalues, $\mathcal{U}_q$, is star-shaped around $c_0\mathbb{I}$, where $c_0$ is the eigenvalue lower bound. 
\end{lemma}

\begin{lemma}\label[lemma]{lemma:delta_properties}
    Assume $c_0^q \leq s_0(n) \lesssim nN^{-q/2}\log^{q/2-1}n$, (i.e. the sparse branch of the rate holds) and $N \gtrsim n^2$. For $q \in (0,1]$, let $R=(ns_0(n))^{1/q}$ and $\delta$ such that \[\delta^2 \asymp ns_0(n)\left(\frac{\log(n)}{N}\right)^{\frac{2-q}{2}}
    %\bigwedge \frac{n^2}{N}
    .\] Then 
    \begin{align}
        \log\left( n^{2-\frac{2}{2-q}} s_0(n)^{-\frac{2}{2-q}} \delta^{\frac{2q}{2-q}} \right) {\lesssim} \log (n) \label{eq:n_to_log_n}
        \shortintertext{and}
         R^{-1}\delta \gtrsim n^{1 - 2/q}. \label{eq:Rdelta_bound}
    \end{align}
\end{lemma}

The \proofref{lemma:delta_properties} is in \ref{appendix:ex_proofs}.

\begin{delayedproof}{cor:lq_ball}
We have that $\mathcal{U}_q$ is star-shaped around the $c_0 \mathbb{I}$ matrix by \Cref{lemma:Uq_star_shaped}

By \Cref{theorem:minimax_rate}, the sampling component of the rate for estimating $\Sigma$ is $\delta^{\star 2} \wedge d^2$, for \[\delta^\star = \sup\left\{ \delta\geq 0:\, \frac{N}{\Upsilon^2} \delta^2\leq \log M^{\operatorname{loc}}_{\mathcal{U}_q}(\delta,c) \right\}.\]
Now note that 
\begin{align*}
    \mathcal{U}_q(s_0(n)) &\subseteq \left\{ \|\Sigma\|_q^q \leq n s_0(n),  0 < c_0 \leq \lambda_{\min}(\Sigma) \leq \lambda_{\max}(\Sigma) \leq c_1 \right\} 
    \\ &\subseteq \left\{ \|\Sigma\|_q^q \leq n s_0(n) \right\}  \eqqcolon U^{+}_q(ns_0(n)) 
\end{align*} and thus 
\[\delta^\star \lesssim \sup\bigg\{\delta \geq 0 : N \delta^2 \leq \log M_{\mathcal{U}_q^+}^{\operatorname{loc}}(\delta,c)\bigg\}.\]
Further note that $U^{+}$ is the $\ell_q$ ball with dimension $n^2$ and radius $(ns_0(n))^{1/q}$. 

We now split into cases based on the value of $q$.

\medskip
\noindent
\textbf{Case 1: $q \in (0,1]$.}
We split this case further into two cases based on the relative magnitude of $s_0(n) \big(\frac{\log n}{N}\big)^{1-q/2} \lessgtr \frac{n}{N}$.
\\\textbf{Case 1A: $s_0(n) \big(\frac{\log n}{N}\big)^{1-q/2} \leq \frac{n}{N}$.}
By the argument of \textcite{neykov_minimax_2023} (see also \textcite{guedon_euclidean_2000}) we see that the log cardinality of a maximal $\delta-$packing set of $U_q^+(R)$ is upper-bounded up to constants depending on $B_1$ and $q$ by $(R^{-1}\delta)^{-\frac{2q}{2-q}}\log\left(n^2(R^{-1}\delta)^{\frac{2q}{2-q}}\right)$ for $R^{-1}\delta \gtrsim (n^2)^{1/2-1/q}$ and $n^2(R^{-1}\delta)^{2q/(2-q)} > B_1 > 1$. 
Define $\tilde\delta^2 = n s_0(n) \big(\frac{\log n}{N}\big)^{1-q/2}$. Note that 
\begin{align*}
    (R^{-1}\tilde\delta)^{-\frac{2q}{2-q}} &= \bigg(n^{1/2-1/q}s_0(n)^{1/2-1/q} \Big(\frac{\log n}{N}\Big)^{1/2-q/4} \bigg)^{-\frac{2q}{2-q}}
    \\ &= ns_0(n)\Big(\frac{\log n}{N}\Big)^{-q/2}.
\end{align*}
Let $z \coloneqq n^2\big(\tilde\delta (cR)^{-1}\big)^{\frac{2q}{2-q}}$ and let $B_1 > 1$ be a positive constant. Split further into cases based on $z \lessgtr B_1$.

\medskip
\noindent
\textbf{Case 1A.1: $z > B_1$.}
First see that \begin{align*}
\log \mathcal{M}_{\mathcal{U}_q}^{\operatorname{loc}}(\tilde\delta,c) &\lesssim (R^{-1}\tilde\delta/c)^{-\frac{2q}{2-q}}\log\left(n^2(R^{-1}\tilde\delta/c)^{\frac{2q}{2-q}}\right).
\end{align*}
Recall that we assume $s_0(n) \big(\frac{\log n}{N}\big)^{1-q/2} \leq \frac{n}{N}$ and $N \gtrsim n^2$ and apply \Cref{lemma:delta_properties} to see that $R^{-1}\tilde\delta \gtrsim n^{1-2/q}$ and
\begin{align*}
\log \mathcal{M}_{\mathcal{U}_q}^{\operatorname{loc}}(\tilde\delta,c) &\lesssim (R^{-1}\tilde\delta/c)^{-\frac{2q}{2-q}}\log\left(n^2(R^{-1}\tilde\delta/c)^{\frac{2q}{2-q}}\right) 
\\ &\lesssim (R^{-1}\tilde\delta)^{-\frac{2q}{2-q}}\log n.
\\ &= ns_0(n)\Big(\frac{\log n}{N}\Big)^{-q/2} \log n
\\ &= N\tilde\delta^2.
\end{align*} 
\medskip
\noindent
\textbf{Case 1A.2: $z \leq B_1$.}
Note that $(R^{-1}\tilde\delta\big)^{-2q/(2-q)} = \frac{n^2}{zc^{2q/(2-q)}}$ and as in the previous case we still have
\begin{align*}
(R^{-1}\tilde\delta)^{-\frac{2q}{2-q}}\log n=  N\tilde\delta^2.
\end{align*} 
So we have 
\begin{align*}
    N\tilde\delta^2 &= (R^{-1}\tilde\delta)^{-\frac{2q}{2-q}}\log n
    \\ &= \frac{n^2}{zc^{2q/(2-q)}}\log n
    \shortintertext{or}
    n^2 &= \frac{N\tilde\delta^2}{\log n}zc^{2q/(2-q)}
\end{align*}
Now, from  \textcite[Example 5.8]{wainwright_highdimensional_2019} we have that, for $z$ bounded by constant $B_1$,
\begin{align*}
    \log \mathcal{M}_{\mathcal{U}_q}^{\operatorname{loc}}(\tilde\delta,c) \leq n^2 \log(1 + 2c) = \frac{N\tilde\delta^2}{\log n}zc^{2q/(2-q)}\log(1+2c)\lesssim N\tilde{\delta}^2
\end{align*}

Now, for both cases define $B_2$ such that $\log \mathcal{M}_{\mathcal{U}_q}^{\operatorname{loc}}(\tilde\delta,c) \leq B_2N\tilde\delta^2$. Then for $A > \sqrt{B_2}\Upsilon \vee 1$
\begin{align*}
    \log \mathcal{M}_{\mathcal{U}_q}^{\operatorname{loc}}(A \tilde\delta,c) \le\log \mathcal{M}_{\mathcal{U}_q}^{\operatorname{loc}}(\tilde\delta,c) \leq B_2N\tilde\delta^2 < \frac{N(A\tilde\delta)^2}{\Upsilon^2}.
\end{align*}
Therefore, $\delta^\star \leq A\tilde\delta$ and $\delta^{\star 2} \lesssim n s_0(n)\big(\frac{\log n}{N}\big)^{1-q/2}$.
\\\textbf{Case 1B: $s_0(n) \big(\frac{\log n}{N}\big)^{1-q/2} \geq \frac{n}{N}$.}
Recall that $N\frac{\delta^2}{\Upsilon^2} \leq \log \mkloc(\delta,c) \lesssim n^2$ and thus we immediately get $N\delta^2 \lesssim n^2 \implies \delta^2 \lesssim \frac{n^2}{N}$.

\medskip
\noindent
\textbf{Case 2: $q=0$.}
Set $s = \min\{n,\lfloor s_0(n)\rfloor\}$ and our set of matrices becomes:
\begin{align*}
    \mathcal{U}_0(s) \coloneqq \bigg\{\Sigma: \max_i \sum_j\indicator{|\sigma_{ij}| > 0}\leq s,\lambda(\Sigma) \in [c_0,c_1] \bigg\}. 
\end{align*}
Using the eigenvalue upper bound, discarding the lower bound, and loosening the row-wise sparsity to overall matrix sparsity we get that 
\begin{align*}
    \mathcal{U}_0(s) \subseteq \big\{\Sigma: \|\Sigma\|_0 \leq ns, \|\Sigma\|_F \leq \sqrt{n}c_1 \big\}
\end{align*}
Notice that 
\begin{align*}
    \mathcal{U}_0(s) \subseteq \bigg\{\Sigma: \|\Sigma\|_0 \leq ns, \|\Sigma\|_F \leq \sqrt{n}c_1 \bigg\} = \bigcup_{|S|=ns}\mathbb{B}_2^{S}(\sqrt{n}c_1).
\end{align*}
That is, this set is a union of $\binom{n^2}{ns}$ $\ell_2^{n^2}$ balls of radius $\sqrt{n}c_1$. 
Now for a local center $\Sigma_0$ we consider the local metric entropy $B_F(\Sigma_0,\delta) \cap \mathcal{U}_0(s)$. Fix the support $S$. From \textcite[Example 5.8]{wainwright_highdimensional_2019} we have that the local metric entropy of this set is 
\begin{align*}
    \log \mkloc(\delta,c) &\leq ns\log\frac{en}{s} + ns\log(1 + 2c)
    \\ &\lesssim ns \log \frac{en}{s}
\end{align*}
due to the fact that
\begin{align*}
    \binom{n^2}{ns} &= \frac{n^2!}{(ns)!(n^2-ns)!} 
    \\ &= \frac{n^2 \times ... \times (n^2 - ns + 1)}{(ns)!}
    \\ &\leq \frac{n^{2ns}}{(ns/e)^{ns}} = \bigg(\frac{en^2}{ns}\bigg)^{ns} = \bigg(\frac{en}{s}\bigg)^{ns}
\end{align*}
with the inequality due to Stirling's approximation. Combining with the bound \\$N\frac{\delta^2}{\Upsilon^2}\leq \log\mkloc(\delta,c) \lesssim n^2$ gives the required bound. 

% Again following \textcite{neykov_minimax_2023} we can use this to approximate the local metrix entropy.
% Setting $N\delta^2 = ns\log\frac{en}{s} + ns\log\bigg(1 + \frac{2\sqrt{n}c_1}{\delta}\bigg)$ gives
% \begin{align*}
%     N\delta^2 &= ns\log\bigg(\frac{en}{s}\bigg(1 + \frac{2\sqrt{n}c_1}{\delta}\bigg)\bigg).
% \end{align*}
% We will show that this gives $\delta \asymp ns\frac{\log n}{N}$. First let $\delta_1^2 = cns\log n / N$. Then we get 
% \begin{align*}
%     N\delta_1^2 &= ns \log \frac{en}{s}\bigg(1 + 2\frac{\sqrt{n}c_1}{\sqrt{cns\log n/N}}\bigg)
%     \\ &= ns \log\Bigg(\frac{en}{s}\bigg( 1 + 2\sqrt{\frac{N}{cs\log n}}\bigg)\Bigg) 
%     \\ &\overset{(i)}{\gtrsim} ns \log\Bigg(\frac{en}{s}\bigg( 1 + \sqrt{\frac{n}{\log n}}\bigg)\Bigg) 
%     \\ &\overset{(ii)}{\gtrsim} ns \log n
% \end{align*}
% where $(i)$ is due to $N \asymp n^2$ and $s \lesssim n$ and $(ii)$ is due to $\log(1 + x) \geq \frac{1}{4}\log x$. On the other hand,
% \begin{align*}
%     N\delta_1^2 &= ns \log \frac{en}{s}\bigg(1 + 2\frac{\sqrt{n}c_1}{\sqrt{cns\log n/N}}\bigg)
%     \\ &= ns \log\Bigg(\frac{en}{s}\bigg( 1 + 2\sqrt{\frac{N}{cs\log n}}\bigg)\Bigg) 
%     \\ &\overset{(i)}{\lesssim} ns \log\Bigg(\frac{en}{s}\bigg( 1 + n\bigg)\Bigg) 
%     \\ &\overset{(ii)}{\lesssim}  ns \log n
% \end{align*}
% where $(i)$ is due to $N \asymp n^2$  and $(ii)$ is due to $\log(1 + x) \leq \log x$.
\end{delayedproof}

\subsection{Sparse Spiked-Covariance model}

\begin{corollary}\label[corollary]{cor:sparse_spiked}
    Consider the set 
    \begin{align*} 
        K^+ \coloneqq \left\{ \Sigma: \Sigma = I + xx\T, x \in K^0\textrm{ star shaped around $0$, $\|x\|_2 \leq K_{\max}$} \right\}
    \end{align*}
    for constant $K_{\max}$. Let $\mathcal{M}_{K^0}^{\operatorname{loc}}(\cdot,\cdot)$ denote local
metric entropy on $K^0$ with respect to the Euclidean norm, and let
$\mathcal{M}_{K^+}^{\operatorname{loc}}(\cdot,\cdot)$ denote local
metric entropy on $K^+$ with respect to the Frobenius norm. Define 
    \[
    \delta_+^\star
    \coloneqq
    \sup_\delta \left\{
        \delta>0:
        N\frac{\delta^2}{\Upsilon^2}
        \leq
        \log
        \mathcal{M}_{K^+}^{\operatorname{loc}}(\delta,c)
    \right\},
    \]
    and
    \[
        \delta^{\star} \coloneqq \sup_\delta \Big\{ N\frac{\delta^2}{\Upsilon^2}
        \leq
        \log 2
        +
        \log
        \mathcal{M}_{K^0}^{\operatorname{loc}}
        \left(
            \frac{\delta}{K_{\max}},
            Cc
        \right)\Big\}
    \] for some universal constant $C >0$ and $N \gtrsim \sup_{\delta} \log\mathcal{M}_{K^+}^{\operatorname{loc}}(\delta,c)$. 
    Then $\delta_{+}^\star \leq \delta^\star$ and we can upper bound the minimax rate as
    \begin{align}
        \mathfrak{M}(X) \lesssim  \varepsilon^2 \log^2(1/\varepsilon) \bigvee \delta^{\star2}.
        \label{eq:k_sparse_rate}
    \end{align} 
\end{corollary}

\begin{lemma}[Local entropy of rank-one spike classes]
\label[lemma]{lemma:spiked_local_entropy}
Let $K^0\subseteq \mathbb{R}^n$ be star-shaped around $0$ and suppose that
\[
    \sup_{x\in K^0}\|x\|_2\leq K_{\max}.
\]
Define
\[
    K^+
    \coloneqq
    \left\{
        I+xx^\top:\ x\in K^0
    \right\}.
\]
Let $\mathcal{M}_{K^0}^{\operatorname{loc}}(\cdot,\cdot)$ denote local
metric entropy on $K^0$ with respect to the Euclidean norm, and let
$\mathcal{M}_{K^+}^{\operatorname{loc}}(\cdot,\cdot)$ denote local
metric entropy on $K^+$ with respect to the Frobenius norm. Then there
is a universal constant $C>0$ such that, for every $c>0$ and every
$0<\delta\leq K_{\max}^2$,
\begin{equation}
    \mathcal{M}_{K^+}^{\operatorname{loc}}(\delta,c)
    \leq
    2\,
    \mathcal{M}_{K^0}^{\operatorname{loc}}
    \left(
        \frac{\delta}{K_{\max}},
        Cc
    \right).
    \label{eq:spiked_local_entropy}
\end{equation}
Consequently,
\begin{equation}
    \log \mathcal{M}_{K^+}^{\operatorname{loc}}(\delta,c)
    \leq
    \log 2
    +
    \log
    \mathcal{M}_{K^0}^{\operatorname{loc}}
    \left(
        \frac{\delta}{K_{\max}},
        Cc
    \right).
    \label{eq:spiked_local_entropy_log}
\end{equation}
\end{lemma}

\begin{remark}
\label{rem:spiked_local_entropy_fixed_point}
The preceding lemma gives the local, rather than merely global,
entropy comparison needed for the sparse spiked-covariance example.
In particular, if
\[
    \delta_+^\star
    \coloneqq
    \sup\left\{
        \delta>0:
        N\delta^2
        \leq
        \log
        \mathcal{M}_{K^+}^{\operatorname{loc}}(\delta,c)
    \right\},
\]
then, up to changing the fixed radius factor $c$ by a universal
constant, $\delta_+^\star$ is bounded above by the largest $\delta$
satisfying
\[
    N\delta^2
    \leq
    \log 2
    +
    \log
    \mathcal{M}_{K^0}^{\operatorname{loc}}
    \left(
        \frac{\delta}{K_{\max}},
        Cc
    \right).
\]
Thus the local entropy of the rank-one covariance class is controlled
by the local Euclidean entropy of the spike class at scale
$\delta/K_{\max}$.
\end{remark}
\begin{delayedproof}{cor:sparse_spiked}
    Since $K^0$ is star-shaped, $K^+$ is also star shaped around $I$, since if $I + x x\T$ is in $K^+$, then for any $\alpha \in [0,1]$, $\alpha(I + x x\T)  + (1-\alpha)I = I + \alpha xx\T$ is in $K^+$ since $K^0$ is star-shaped around $0$. The identity matrix in the construction implies $\lambda_{\min} \geq 1$. The corruption part of the rate is immediate. We thus show a bound on the local metric entropy to complete the proof.
    Consider an arbitrary point $\Sigma_0 \in K^+$. We need to show that we can form a set of points $\Sigma_1 = I + x_1x_1\T,...,\Sigma_M = I + x_Mx_M\T$ in $K^+$ such that for any $\Sigma_0 \in K^+$,  $\exists i$ such that $\|\Sigma_i - \Sigma_0\| \leq \delta^\star$. 

    \textbf{Case 1: $\delta_+^\star \leq K_{\max}^2$. } The conclusion follows by \Cref{lemma:spiked_local_entropy}. 
    \\\textbf{Case 2: $\delta_+^\star > K_{\max}^2$. } 
    Note that 
    \begin{align*}
        \| xx\T - yy\T\|_F &\leq \|x(x-y)\T\|_F + \|(x-y)y\T\|_F
        \\ &= \|x\|_2\|x-y\|_2 + \|x-y\|_2\|y\|_2
        \\ &\leq 2K_{\max}\|x-y\|_2
    \end{align*}
    and therefore $x \mapsto xx\T$ is $2K_{\max}$-Lipschitz.
    Fix an arbitrary
    \[
        \Sigma_0=I+x_0x_0^\top\in K^+,
    \]
    and let
    \[
        \{\Sigma_i=I+x_ix_i^\top\}_{i=1}^M
        \subseteq
        B_F(\Sigma_0,\delta)\cap K^+
    \]
    be a $\delta/c$-packing in Frobenius norm.  We will bound $M$
    uniformly over $\Sigma_0$. 
    Now, the Lipschitz inequalty gives that for all $i \neq j$
    \begin{align*}
        \frac{\delta}{c} \leq \|\Sigma_i - \Sigma_j\|_F \leq 2K_{\max}\|x_i-x_j\|_2 
    \end{align*}
    and thus the $x_i$ form a $\frac{\delta}{2K_{\max}c}$-packing. Now, since $\delta > K_{\max}^2$ we have that $\frac{\delta}{K_{\max}} > K_{\max}$ and thus for all $x_i \in K^0$ we have $x_i \in B_2\big(0,\frac{\delta}{K_{\max}}\big) \cap K^0$. Furthermore, $0 \in K^0$ by star-shapedness, so we can bound $M \leq \mathcal{M}_{K^0}^{\operatorname{loc}}\big(\frac{\delta}{K_{\max}},2c\big)$. Taking the supremum over $\Sigma_0$ gives the desired result.
\end{delayedproof}

A good example of a set of sparse-spiked covariance matrices is \begin{align*}
    K^1 \coloneqq \{x \in \mathbb{R}^n : \|x\|_2 \leq 1, x \mbox{ is } s-\mbox{sparse}\}.
\end{align*} 
This is a star shaped set around $0$, which means that $K_1^+$ is star-shaped around $I$. From \textcite{prasadan_information_2025} we have that the log local metric entropy of ($K^1$) admits the uniform upper bound (up to constants) of $s\log\left(1+\frac{n}{2s}\right)$ we get tighter bounds on the minimax rate:

\begin{align*}
     \mathfrak{M}(X) \lesssim \varepsilon^2 \log^2(1/\varepsilon)\bigvee \frac{s}{N}\log\left(1 + \frac{n}{2s}\right).
\end{align*}

\subsection{Gene Co-Expression Data}

In many biological applications, such as gene co-expression estimation, data collection can suffer from technical problems such as failed libraries or degraded RNA corrupting a small fraction of samples. We can look at this situation through the lens of adversarial corruption. Assume the true log-transformed gene expression data is Gaussian and say the adversary picks $\varepsilon = 1/32 \wedge k/n$ proportion of the samples (for some arbitrary $k$) and changes their lowest-value points to zeroes. Assume the conditions of Corollary \ref{cor:lq_ball} hold additionally. Then we can similarly upper bound the worst-case Frobenius error as \begin{align*}\frac{1}{n} \mathfrak{M}(X) \lesssim \frac{1}{n}\log^2(1/32)\frac{1}{32^2} \bigvee \left( s_0(n)\left(\frac{\log n}{N}\right)^{1-q/2} \bigwedge \frac{n}{N} \right).\end{align*}

\section{Conclusions and Future Work}

We have proven a near-minimax rate in the Frobenius norm of the covariance matrix in a star-shaped set for the multivariate Gaussian model under adversarial corruption. We then used this rate to prove results for several example situations, most notably when the covariance matrices are weakly sparse (extending the results of \textcite{bickel_covariance_2008}), for zero-inflated high-dimensional data, and the spiked covariance model. Unfortunately the algorithm needed to prove the upper bound is not computationally tractable, so the natural extension of this work is to find an implementable algorithm for which these results (or some subset of them) hold. Additionally, in the future we seek to extend these results to multi-spiked covariance models as well.

\subsection*{AI usage disclosure}
ChatGPT Astra was used to assist the authors in proofreading the results and proofs for correctness, consistency, and readability.

\printbibliography

@article{bickel_covariance_2008,
  title = {Covariance Regularization by Thresholding},
  author = {Bickel, Peter J. and Levina, Elizaveta},
  date = {2008-12},
  journaltitle = {The Annals of Statistics},
  volume = {36},
  number = {6},
  pages = {2577--2604},
  publisher = {Institute of Mathematical Statistics},
  issn = {0090-5364, 2168-8966},
  doi = {10.1214/08-AOS600},
}

@article{fan_large_2013,
  title = {Large {{Covariance Estimation}} by {{Thresholding Principal Orthogonal Complements}}},
  author = {Fan, Jianqing and Liao, Yuan and Mincheva, Martina},
  date = {2013-09-01},
  journaltitle = {Journal of the Royal Statistical Society Series B: Statistical Methodology},
  shortjournal = {Journal of the Royal Statistical Society Series B: Statistical Methodology},
  volume = {75},
  number = {4},
  pages = {603--680},
  issn = {1369-7412},
  doi = {10.1111/rssb.12016},
}

@online{abdalla_covariance_2023,
  title = {Covariance {{Estimation}}: {{Optimal Dimension-free Guarantees}} for {{Adversarial Corruption}} and {{Heavy Tails}}},
  shorttitle = {Covariance {{Estimation}}},
  author = {Abdalla, Pedro and Zhivotovskiy, Nikita},
  date = {2023-07-20},
  eprint = {2205.08494},
  eprinttype = {arXiv},
  eprintclass = {math},
  doi = {10.48550/arXiv.2205.08494},
  url = {http://arxiv.org/abs/2205.08494},
  urldate = {2025-07-10}
}

@online{bellec_sharp_2016,
  title = {Sharp Oracle Inequalities for {{Least Squares}} Estimators in Shape Restricted Regression},
  author = {Bellec, Pierre C.},
  date = {2016-08-08},
  eprint = {1510.08029},
  eprinttype = {arXiv},
  eprintclass = {math},
  doi = {10.48550/arXiv.1510.08029},
  url = {http://arxiv.org/abs/1510.08029},
  urldate = {2025-07-10}
}

@article{chen_robust_2018,
  title = {Robust Covariance and Scatter Matrix Estimation under {{Huber}}’s Contamination Model},
  author = {Chen, Mengjie and Gao, Chao and Ren, Zhao},
  date = {2018-10-01},
  journaltitle = {The Annals of Statistics},
  shortjournal = {Ann. Statist.},
  volume = {46},
  number = {5},
  publisher = {Institute of Mathematical Statistics},
  issn = {0090-5364},
  doi = {10.1214/17-aos1607},
  url = {https://projecteuclid.org/journals/annals-of-statistics/volume-46/issue-5/Robust-covariance-and-scatter-matrix-estimation-under-Hubers-contamination-model/10.1214/17-AOS1607.full},
  urldate = {2025-07-14}
}

@article{lugosi_robust_2021,
  title = {Robust Multivariate Mean Estimation: {{The}} Optimality of Trimmed Mean},
  shorttitle = {Robust Multivariate Mean Estimation},
  author = {Lugosi, Gábor and Mendelson, Shahar},
  date = {2021-02-01},
  journaltitle = {The Annals of Statistics},
  shortjournal = {Ann. Statist.},
  volume = {49},
  number = {1},
  publisher = {Institute of Mathematical Statistics},
  issn = {0090-5364},
  doi = {10.1214/20-aos1961},
  url = {https://projecteuclid.org/journals/annals-of-statistics/volume-49/issue-1/Robust-multivariate-mean-estimation-The-optimality-of-trimmed-mean/10.1214/20-AOS1961.full},
  urldate = {2025-07-18}
}

@online{ma_volume_2013,
  title = {Volume {{Ratio}}, {{Sparsity}}, and {{Minimaxity}} under {{Unitarily Invariant Norms}}},
  author = {Ma, Zongming and Wu, Yihong},
  date = {2013-06-15},
  eprint = {1306.3609},
  eprinttype = {arXiv},
  eprintclass = {math},
  doi = {10.48550/arXiv.1306.3609},
  url = {http://arxiv.org/abs/1306.3609},
  urldate = {2025-07-19}
}

@article{neykov_minimax_2023,
  title = {On the {{Minimax Rate}} of the {{Gaussian Sequence Model Under Bounded Convex Constraints}}},
  author = {Neykov, Matey},
  date = {2023-02},
  journaltitle = {IEEE Transactions on Information Theory},
  volume = {69},
  number = {2},
  pages = {1244--1260},
  issn = {1557-9654},
  doi = {10.1109/TIT.2022.3213141},
  url = {https://ieeexplore.ieee.org/document/9914630},
  urldate = {2025-07-19}
}

@article{prasadan_information_2025,
  title = {Information Theoretic Limits of Robust Sub-{{Gaussian}} Mean Estimation under Star-Shaped Constraints},
  author = {Prasadan, Akshay and Neykov, Matey},
  year = 2026,
  month = feb,
  journal = {The Annals of Statistics},
  volume = {54},
  number = {1},
  pages = {490--515},
  publisher = {Institute of Mathematical Statistics},
  issn = {0090-5364, 2168-8966},
  doi = {10.1214/25-AOS2576},
  urldate = {2026-04-21},
  langid = {english}
}

@book{wainwright_highdimensional_2019,
  title = {High-{{Dimensional Statistics}}: {{A Non-Asymptotic Viewpoint}}},
  shorttitle = {High-{{Dimensional Statistics}}},
  author = {Wainwright, Martin J.},
  date = {2019-02-21},
  edition = {1},
  publisher = {Cambridge University Press},
  doi = {10.1017/9781108627771},
  url = {https://www.cambridge.org/core/product/identifier/9781108627771/type/book},
  urldate = {2025-07-21},
  isbn = {978-1-108-62777-1 978-1-108-49802-9}
}

@online{novikov_robust_2025,
  title = {Robust {{Mean Estimation Without Moments}} for {{Symmetric Distributions}}},
  author = {Novikov, Gleb and Steurer, David and Tiegel, Stefan},
  date = {2025-11-10},
  eprint = {2302.10844},
  eprinttype = {arXiv},
  eprintclass = {cs},
  doi = {10.48550/arXiv.2302.10844},
  url = {http://arxiv.org/abs/2302.10844},
  urldate = {2025-07-24}
}

@article{minasyan_statistically_2025,
  title = {Statistically Optimal Robust Mean and Covariance Estimation for Anisotropic {{Gaussians}}},
  author = {Minasyan, Arshak and Zhivotovskiy, Nikita},
  year = 2025,
  month = {May},
  journal = {Mathematical Statistics and Learning},
  volume = {8},
  number = {1},
  pages = {33--69},
  issn = {2520-2316},
  doi = {10.4171/msl/48},
  urldate = {2026-08-07},
  langid = {english}
}

@online{minsker_uniform_2019,
  title = {Uniform Bounds for Robust Mean Estimators},
  author = {Minsker, Stanislav},
  date = {2019-05-03},
  eprint = {1812.03523},
  eprinttype = {arXiv},
  eprintclass = {math},
  doi = {10.48550/arXiv.1812.03523},
  url = {http://arxiv.org/abs/1812.03523},
  urldate = {2025-07-24}
}

@article{diakonikolas_robust_2019,
  title = {Robust {{Estimators}} in {{High Dimensions}} without the {{Computational Intractability}}},
  author = {Diakonikolas, Ilias and Kamath, Gautam and Kane, Daniel and Li, Jerry and Moitra, Ankur and Stewart, Alistair},
  date = {2019},
  journaltitle = {Society for Industrial and Applied Mathematics},
  volume = {48},
  number = {2},
  pages = {742--864},
  doi = {10.1137/17M1126680},
  url = {http://arxiv.org/abs/1604.06443}
}

@article{diakonikolas_being_2017,
  title = {Being {{Robust}} (in {{High Dimensions}}) {{Can Be Practical}}},
  author = {Diakonikolas, Ilias and Kamath, Gautam and Kane, Daniel M. and Li, Jerry and Moitra, Ankur and Stewart, Alistair},
  date = {2017},
  journaltitle = {Proceedings of Machine Learning Research},
  volume = {70},
  pages = {999--1008},
  doi = {10.48550/arXiv.1703.00893},
  url = {http://arxiv.org/abs/1703.00893}
}

@online{cheng_robust_2018,
  title = {Robust {{Learning}} of {{Fixed-Structure Bayesian Networks}}},
  author = {Cheng, Yu and Diakonikolas, Ilias and Kane, Daniel and Stewart, Alistair},
  date = {2018-10-29},
  eprint = {1606.07384},
  eprinttype = {arXiv},
  eprintclass = {cs},
  doi = {10.48550/arXiv.1606.07384},
  url = {http://arxiv.org/abs/1606.07384},
  urldate = {2025-07-28}
}

@article{diakonikolas_robustly_2018,
  title = {Robustly {{Learning}} a {{Gaussian}}: {{Getting Optimal Error}}, {{Efficiently}}},
  author = {Diakonikolas, Ilias and Kamath, Gautam and Kane, Daniel M. and Li, Jerry and Moitra, Ankur and Stewart, Alistair},
  date = {2018-01},
  journaltitle = {Proceedings of the Twenty-Ninth Annual ACM-SIAM Symposium on Discrete Algorithms},
  pages = {2683--2702},
  doi = {10.1137/1.9781611975031},
  url = {https://epubs.siam.org/doi/book/10.1137/1.9781611975031}
}

@inproceedings{diakonikolas_statistical_2017,
  title = {Statistical {{Query Lower Bounds}} for {{Robust Estimation}} of {{High-Dimensional Gaussians}} and {{Gaussian Mixtures}}},
  booktitle = {2017 {{IEEE}} 58th {{Annual Symposium}} on {{Foundations}} of {{Computer Science}} ({{FOCS}})},
  author = {Diakonikolas, Ilias and Kane, Daniel M. and Stewart, Alistair},
  date = {2017-10},
  pages = {73--84},
  publisher = {IEEE},
  location = {Berkeley, CA},
  doi = {10.1109/FOCS.2017.16},
  url = {http://ieeexplore.ieee.org/document/8104048/},
  urldate = {2025-07-28},
  eventtitle = {2017 {{IEEE}} 58th {{Annual Symposium}} on {{Foundations}} of {{Computer Science}} ({{FOCS}})},
  isbn = {978-1-5386-3464-6}
}

@online{lai_agnostic_2016,
  title = {Agnostic {{Estimation}} of {{Mean}} and {{Covariance}}},
  author = {Lai, Kevin A. and Rao, Anup B. and Vempala, Santosh},
  date = {2016-08-14},
  eprint = {1604.06968},
  eprinttype = {arXiv},
  eprintclass = {cs},
  doi = {10.48550/arXiv.1604.06968},
  url = {http://arxiv.org/abs/1604.06968},
  urldate = {2025-07-24}
}

@online{novikov_nearly_2025,
  title = {Nearly {{Optimal Robust Covariance}} and {{Scatter Matrix Estimation Beyond Gaussians}}},
  author = {Novikov, Gleb},
  date = {2025-04-12},
  eprint = {2502.06564},
  eprinttype = {arXiv},
  eprintclass = {cs},
  doi = {10.48550/arXiv.2502.06564},
  url = {http://arxiv.org/abs/2502.06564},
  urldate = {2025-07-28}
}

@article{vershynin_introduction_2011,
  title = {Introduction to the Non-Asymptotic Analysis of Random Matrices},
  author = {Vershynin, Roman},
  date = {2011-11-23},
  eprint = {1011.3027},
  eprinttype = {arXiv},
  eprintclass = {math},
  doi = {10.48550/arXiv.1011.3027},
  url = {http://arxiv.org/abs/1011.3027},
  urldate = {2025-08-22},
  langid = {english},
  pubstate = {prepublished},
}

@inproceedings{cheng_faster_2019,
  title = {Faster {{Algorithms}} for {{High-Dimensional Robust Covariance Estimation}}},
  booktitle = {Proceedings of the {{Thirty-Second Conference}} on {{Learning Theory}}},
  author = {Cheng, Yu and Diakonikolas, Ilias and Ge, Rong and Woodruff, David P.},
  date = {2019-06-25},
  pages = {727--757},
  publisher = {PMLR},
  issn = {2640-3498},
  url = {https://proceedings.mlr.press/v99/cheng19a.html},
  urldate = {2026-02-13},
  eventtitle = {Conference on {{Learning Theory}}}
}

@article{minsker_robust_2024,
  title = {Robust {{Estimation}} of {{Covariance Matrices}}: {{Adversarial Contamination}} and {{Beyond}}},
  shorttitle = {Robust {{Estimation}} of {{Covariance Matrices}}},
  author = {Minsker, Stanislav and Wang, Lang},
  date = {2024},
  journaltitle = {Statistica Sinica},
  volume = {34},
  number = {3},
  eprint = {27318966},
  eprinttype = {jstor},
  pages = {1565--1583},
  publisher = {Institute of Statistical Science, Academia Sinica},
  issn = {1017-0405},
  url = {https://www.jstor.org/stable/27318966},
  urldate = {2026-02-13}
}

@article{cai_optimal_2021,
  title = {Optimal {{Structured Principal Subspace Estimation}}: {{Metric Entropy}} and {{Minimax Rates}}},
  author = {Cai, Tony and Li, Hongzhe and Ma, Rong},
  date = {2021},
  journaltitle = {Journal of Machine Learning Research},
  volume = {22},
  pages = {1--45},
  langid = {english}
}

@inproceedings{vu_minimax_2012,
  title = {Minimax {{Rates}} of {{Estimation}} for {{Sparse PCA}} in {{High Dimensions}}},
  booktitle = {Proceedings of the {{Fifteenth International Conference}} on {{Artificial Intelligence}} and {{Statistics}}},
  author = {Vu, Vincent and Lei, Jing},
  date = {2012-03-21},
  pages = {1278--1286},
  publisher = {PMLR},
  issn = {1938-7228},
  url = {https://proceedings.mlr.press/v22/vu12.html},
  urldate = {2026-03-12},
  eventtitle = {Artificial {{Intelligence}} and {{Statistics}}},
  langid = {english}
}

@article{huber_robust_1964,
  title = {Robust {{Estimation}} of a {{Location Parameter}}},
  author = {Huber, Peter J.},
  date = {1964-03},
  journaltitle = {The Annals of Mathematical Statistics},
  shortjournal = {Ann. Math. Statist.},
  volume = {35},
  number = {1},
  pages = {73--101},
  issn = {0003-4851},
  doi = {10.1214/aoms/1177703732},
  langid = {english}
}

@book{diakonikolas_algorithmic_2023a,
  title = {Algorithmic {{High-Dimensional Robust Statistics}}},
  author = {Diakonikolas, Ilias and Kane, Daniel M.},
  date = {2023-08-31},
  edition = {1},
  publisher = {Cambridge University Press},
  doi = {10.1017/9781108943161},
  isbn = {978-1-108-94316-1 978-1-108-83781-1}
}

@misc{cai_minimax_2026,
  title = {Minimax and {{Adaptive Covariance Matrix Estimation}} under {{Differential Privacy}}},
  author = {Cai, T. Tony and Li, Yicheng},
  year = 2026,
  month = mar,
  number = {arXiv:2603.19703},
  eprint = {2603.19703},
  primaryclass = {math},
  publisher = {arXiv},
  doi = {10.48550/arXiv.2603.19703},
  urldate = {2026-04-13},
  archiveprefix = {arXiv}
}

@misc{cai_optimal_2024,
  title = {Optimal {{Differentially Private PCA}} and {{Estimation}} for {{Spiked Covariance Matrices}}},
  author = {Cai, T. Tony and Xia, Dong and Zha, Mengyue},
  year = 2024,
  month = sep,
  number = {arXiv:2401.03820},
  eprint = {2401.03820},
  primaryclass = {math},
  publisher = {arXiv},
  doi = {10.48550/arXiv.2401.03820},
  urldate = {2026-04-13},
  archiveprefix = {arXiv}
}

@article{peng_robust_2026,
  title = {Robust Mean Estimation under Star-Shaped Constraints with Heavy-Tailed Noise},
  author = {Peng, Tuorui and Prasadan, Akshay and Neykov, Matey},
  year = 2026,
  month = apr,
  number = {arXiv:2604.05063},
  eprint = {2604.05063},
  primaryclass = {math},
  publisher = {arXiv},
  doi = {10.48550/arXiv.2604.05063},
  urldate = {2026-04-15},
  archiveprefix = {arXiv}
}

@incollection{guedon_euclidean_2000,
  title = {Euclidean Projections of a P-Convex Body},
  booktitle = {Geometric {{Aspects}} of {{Functional Analysis}}},
  author = {Gu{\'e}don, O. and Litvak, A. E.},
  editor = {Milman, Vitali D. and Schechtman, Gideon},
  year = 2000,
  volume = {1745},
  pages = {95--108},
  publisher = {Springer Berlin Heidelberg},
  address = {Berlin, Heidelberg},
  doi = {10.1007/BFb0107210},
  urldate = {2026-09-11},
  isbn = {978-3-540-41070-6 978-3-540-45392-5},
  langid = {english}
}

\newpage
\appendix

\section{Proofs for \Cref{sec:lower_bound} (lower bound)} \label[appendix]{appendix:lower_bound_pfs}

Before progressing to the proofs of \Cref{lemma:lower_bound_1} and \Cref{lemma:lower_bound_2}, we must obtain some intermediate results which are required for those proofs. 

\begin{lemma} \label[lemma]{lemma:min_max_eigen}
    For positive definite matrix A with eigenvalues $0 < \lambda_1 \leq \dots \leq \lambda_n$, 
    \begin{align}
        \lambda_{n}(A^{-1}) &= \lambda_{1}(A)^{-1}
    \end{align}
\end{lemma}

\begin{delayedproof}{lemma:min_max_eigen}\label{proof:min_max_eigen} Begin with the eigen decomposition of A: 
    \begin{align*}
        A v &= \lambda v
        \\ A^{-1} Av &= \lambda A^{-1}v
        \\ v &= \lambda A^{-1} v 
        \\ \lambda^{-1} v &= A^{-1}v.
    \end{align*}
    Therefore each eigenvalue of $A^{-1}$ is the inverse of an eigenvalue of $A$. Since all eigenvalues are positive this further implies the desired result.
\end{delayedproof}

\begin{lemma} \label[lemma]{lemma:kl_bound}
    Let $\Sigma_i$ and $\Sigma_j$ be covariance matrices in $K$.  Then there exists a sufficiently large constant $C_0 > 1$ depending only on $c_0$ and $c_1$ such that we can bound the KL divergence of normal random variables with mean 0 and these covariances as follows: 
    \begin{align}
        KL(N(0,\Sigma_i) \| N(0, \Sigma_j)) &\leq C_0\| \Sigma_j^{-1} \Sigma_i - \mathbb{I} \|_F^2  \leq \frac{C_0}{c_0^2}\|\Sigma_i - \Sigma_j\|_F^2\label{eq:kl_f_ineq}
    \end{align}
\end{lemma}

\begin{delayedproof}{lemma:kl_bound}
    Define $A \coloneqq \Sigma_j^{-1} \Sigma_i$ with eigenvalues $\lambda_{\min} \leq \dots \leq \lambda_{\max}$. Note that $\Sigma_j^{1/2} (\Sigma_j^{-1}\Sigma_i) \Sigma_j^{-1/2} = \Sigma_{j}^{-1/2}\Sigma_i\Sigma_j^{-1/2}$ so $\Sigma_j^{-1}\Sigma_i$ and $\Sigma_j^{-1/2}\Sigma_i\Sigma_{j}^{-1/2}$ are similar. Thus the eigenvalues $\lambda_i$ of $\Sigma_j^{-1}\Sigma_i$ are equal to those of the symmetric $\Sigma_j^{-1/2}\Sigma_i\Sigma_j^{-1/2}$.
    We use \Cref{lemma:min_max_eigen} and Condition \ref{cond:K_set} to bound $\lambda_{\min}(A)$  as follows
    \begin{align*}
        \lambda_{\max}((\Sigma_j^{-1/2}\Sigma_i\Sigma_j^{-1/2})^{-1})  &= \lambda_{\max}(\Sigma_i^{-1/2}\Sigma_j\Sigma_i^{-1/2}) 
        \\ &\leq \lambda_{\max}(\Sigma_{i}^{-1/2})\lambda_{\max}(\Sigma_j)\lambda_{\max}(\Sigma_{i}^{-1/2})
        \\ &\leq \frac{c_1}{c_0}
    \end{align*} 
    We apply \Cref{lemma:min_max_eigen} to this inequality and get 
    \begin{align}
        \frac{1}{\lambda_{\min}(A)} \leq \frac{c_1}{c_0} \implies \lambda_{\min}(A) \geq \frac{c_0}{c_1}, \label{eq:eigen_bound_sigsig}
    \end{align}
    On the other hand, 
    \begin{align*}
        \lambda_{\max}(\Sigma_j^{-1/2}\Sigma_i\Sigma_j^{-1/2}) &\leq \|\Sigma_j^{-1/2}\|_{\operatorname{op}}^2\|\Sigma_i\|_{\operatorname{op}} \leq \frac{c_1}{c_0}.
    \end{align*}
    Now we show \Cref{eq:kl_f_ineq}:
    \begin{align}
        KL(N(0,\Sigma_i) \| N(0, \Sigma_j)) &= \frac{1}{2} \left(tr(\Sigma_j^{-1}\Sigma_i) - n + \log \frac{|\Sigma_j|}{|\Sigma_i|} \right) \notag
       \\ &= \frac{1}{2} \sum_{i=1}^{n} (\lambda_i - 1 + \log \lambda_i^{-1}) \label{eq:kl_equality}
    \end{align}
    where $\lambda_i$ denotes the $i$th eigenvalue of $\Sigma_j^{-1/2}\Sigma_i\Sigma_j^{-1/2}$.  Next, let $\mu_i \coloneqq \lambda_i - 1$ and see that
    \begin{align*}
        \lambda_i - 1 + \log(\lambda_i)^{-1} &= \mu_i - \log(\mu_i + 1) 
    \end{align*}
    We claim 
    \begin{align}
        \mu_i - \log(\mu_i + 1) \leq C_0\mu_i^2 \hspace{3mm} \text{on } \mu_i \in \bigg[\frac{c_0}{c_1}-1,\frac{c_1}{c_0}-1\bigg]\label{eq:log_sq_ineq}
    \end{align} for a sufficiently large $C_0$.
    To see this, first let $h(x)= \begin{cases} \frac{x - \log(x + 1)}{x^2} & x \neq 0 \\ 1/2 & x = 0\end{cases}$. This function is continuous on all $\Big[\frac{c_0}{c_1}-1,\frac{c_1}{c_0}-1\Big]$. This interval is closed and bounded it is compact, so $h(x)$ is guaranteed to attain a maximum somewhere on $\Big[\frac{c_0}{c_1}-1,\frac{c_1}{c_0}-1\Big]$, call it $C_0^\star \coloneqq \max h(x)$. Then let $C_0 = C_0^\star \vee 1$. Thus, $h(u_i) = \frac{u_i - \log(u_i+1)}{u_i^2} \leq C_0$ and the property holds.
    Continuing on from \Cref{eq:kl_equality}, we have
    \begin{align*}
        KL(N(0,\Sigma_i) \| N(0, \Sigma_j)) &\leq C_0\sum_{i=1}^n \mu_i^2
        \\ &= C_0\sum_{i=1}^n (\lambda_i - 1)^2
        \\ &\leq C_0\|\Sigma_j^{-1}\Sigma_i - \mathbb{I}\|_{F}^2
        \\ &\leq C_0\|\Sigma_j^{-1}\|_2^2 \|\Sigma_i - \Sigma_j\|_F^2
        \\ &\leq \frac{C_0}{c_0^2}\|\Sigma_i - \Sigma_j\|_F^2,
    \end{align*}
    where the last inequality holds due to the bound on $\lambda_{\min}$.
\end{delayedproof}

\begin{corollary} \label[corollary]{cor:info_bound} For any $\Sigma_0 \in K$ we have
    \begin{align} 
        I(X;J) \leq \frac{C_0}{c_0^2} N \max_j \|\Sigma_0 - \Sigma_j\|_F^2 
    \end{align}
\end{corollary}

\begin{delayedproof}{cor:info_bound}
    Let $\Prob_{X_i}^{\Sigma_0}$ be the distribution of $X_i \sim N(0, \Sigma_0)$ and  $\Prob_{X_i}^{\Sigma_j}$ be the distribution of $(X_i | J=j) \sim N(0, \Sigma_j)$
    \begin{align*}
    I(X;J) &\leq \frac{1}{m} \sum_{j=1}^m KL\left(\Prob_{X_i}^{\Sigma_j\otimes N} \| \Prob_{X_i}^{\Sigma_0\otimes N}\right) 
    \\ &= \frac{1}{m} \sum_{j=1}^m \sum_{i=1}^N KL\left(\Prob_{X_i}^{\Sigma_j} \| \Prob_{X_i}^{\Sigma_0}\right) 
    \\ &= \frac{1}{m} \sum_{j=1}^m \sum_{i=1}^N KL(N(0,\Sigma_j) \| N(0, \Sigma_0))
    \\ &\leq \frac{NC_0}{mc_0^2} \sum_{j=1}^m \|\Sigma_0 - \Sigma_j\|_F^2
    \\ &\leq \frac{NC_0}{c_0^2} \max_j \|\Sigma_0 - \Sigma_j\|_F^2
    \end{align*}
\end{delayedproof}

\begin{delayedproof}{lemma:lower_bound_1}
    The first inequality is trivial by assuming no corruption.
    Pick $\Sigma^\star \in K$ such that \begin{align}
        M \coloneqq \mkloc(\eta,c) \coloneqq \sup_{\Sigma^\star \in K} \mathcal{M}(\eta/c, B(\Sigma^\star, \eta) \cap K)
    \end{align}
    Let $\Sigma_1, \dots, \Sigma_M$ be a maximal $\eta/c$ packing set for $B(\Sigma^\star,\eta) \cap K$. This of course implies $\|\Sigma_i - \Sigma_j\|_F > \eta/c$ whenever $i \neq j$. Also, $\|\Sigma^\star - \Sigma_i\|_F \leq \eta$ for every $j \in [M]$ because $\Sigma_j \in B(\Sigma^\star,\eta) \cap K$. Let $J$ be uniformly distributed on $[M]$ and apply \Cref{cor:info_bound} to get 
    \begin{align}
        I(X; J) &\leq N \frac{C_0}{c_0^2} \max_j \|\Sigma^\star - \Sigma_j\|_F^2 \leq \frac{C_0}{c_0^2} N\eta^2
    \end{align}
    Apply \Cref{lemma:fano} to get 
    \begin{align}
        \inf_{\hat\Sigma} \sup_{\Sigma} \E \| \hat\Sigma(X) - \Sigma \|_F^2 &\geq \frac{\eta^2}{4c^2} \left(1 - \frac{I(X;J) + \log 2}{\log M} \right) 
        \\ &\geq \frac{\eta^2}{4c^2} \left(1 - \frac{\frac{C_0}{c_0^2} N\eta^2 + \log 2}{\log M} \right) 
    \end{align}
    Choosing $\eta$ such that $\log M > 4 \left(\frac{C_0}{c_0^2} N\eta^2 \vee \log 2 \right) \geq 2\left(\frac{C_0}{c_0^2} N\eta^2 + \log 2\right)$ gives the desired bound. 
\end{delayedproof}

\begin{delayedproof}{lemma:lower_bound_2} 
First choose $\varepsilon^\prime = \varepsilon - \frac{1}{\sqrt{2N}}$. Note that $\varepsilon \asymp \varepsilon^\prime$. Choose $\Sigma_1$, $\Sigma_2 \in K$ such that $(d/3)^2 \wedge \frac{2c_0^2\varepsilon^{\prime 2}}{C_0(1-\varepsilon^\prime)^2} \leq \|\Sigma_1 - \Sigma_2\|_F^2 \leq \frac{2c_0^2 \varepsilon^{\prime 2}}{C_0(1-\varepsilon^\prime)^2}$, which is possible by \Cref{lemma:d_line_segment}. Then \begin{align*}
TV(N(0,\Sigma_1),N(0,\Sigma_2))^2 &\overset{(i)}{\leq} \frac{1}{2} KL(N(0,\Sigma_1),N(0,\Sigma_2)) 
\\ &\overset{(ii)}{\leq} \frac{C_0}{2c_0^2}\|\Sigma_1 - \Sigma_2\|_F^2 
\\&\overset{(iii)}{\leq} \frac{\varepsilon^{\prime 2}}{(1-\varepsilon^\prime)^2} \end{align*} where $(i)$ is the Pinsker-Csiszár-Kullback inequality, $(ii)$ is due to \Cref{lemma:kl_bound}, and $(iii)$ is due to the definition of $\Sigma_1$ and $\Sigma_2$. Also, by \textcite[Proof of theorem 5.1]{chen_robust_2018} we can pick $\varepsilon^{\prime\prime} \leq \varepsilon^\prime$ such that $TV(N(0,\Sigma_1),N(0,\Sigma_2)) = \frac{\varepsilon^{\prime\prime}}{(1-\varepsilon^{\prime\prime})}$. Denote the densities of $N(0,\Sigma_1)$ and $N(0,\Sigma_2)$ as $p_1,p_2$  (and their measures $P_1$,$P_2$). Then define two measures, $Q_1$ and $Q_2$ via their densities as follows: 
\begin{align*}
    q_1 &= \frac{(p_2 - p_1) \indicator{p_2 > p_1}}{TV(P_1,P_2)} \hspace{10mm} q_2 = \frac{(p_1 - p_2) \indicator{p_1 > p_2}}{TV(P_1,P_2)}
\end{align*}
First we must confirm these are valid probability measures:
\begin{align*}
     \int (p_2 - p_1) \indicator{p_2 > p_1} = 1 - \int p_1 \wedge p_2 = \int (p_1 - p_2)\indicator{p_1 > p_2}
\end{align*}
and 
\begin{align*}
     \int (p_2 - p_1) \indicator{p_2 > p_1} + \int (p_1 - p_2)\indicator{p_1 > p_2} = 2TV(P_1,P_2)
\end{align*}
then 
\begin{align}
    \int (p_2 - p_1) \indicator{p_2 > p_1} = \int (p_1 - p_2)\indicator{p_1 > p_2} = TV(P_1,P_2) \hspace{1mm}.
\end{align}
Therefore $\int q_1 = \int q_2 = 1$ so these are valid densities.
\\ Next, we create two mixtures: 
\begin{align*}
    R_1 &= (1 - \varepsilon^{\prime\prime})P_1 + \varepsilon^{\prime\prime} Q_1 
    \\ R_2 &= (1 - \varepsilon^{\prime\prime})P_2 + \varepsilon^{\prime\prime} Q_2.
\end{align*}
First, notice that these mixtures can be defined as follows: with probability $1-\varepsilon^{\prime\prime}$ the original draw from $N(0,\Sigma_i)$ is kept and with probability $\varepsilon^{\prime\prime}$ it is corrupted to a draw from $Q_i$. 
\\ Notice that since $TV(P_1,P_2) = \frac{\varepsilon^{\prime\prime}}{1 - \varepsilon^{\prime\prime}}$, the densities of $R_1$ and $R_2$ (with respect to the Lebesgue measure) are equal:
\begin{align*}
    r_1 &= (1 - \varepsilon^{\prime\prime})p_1 + (1-\varepsilon^{\prime\prime}) (p_2 - p_1)\indicator{p_2 > p_1} 
    \\ &= (1-\varepsilon^{\prime\prime})p_1\indicator{p_1 \geq p_2} + (1-\varepsilon^{\prime\prime})p_2\indicator{p_2 > p_1}
    \\ &= (1-\varepsilon^{\prime\prime})(p_1-p_2)\indicator{p_1 \geq p_2} + (1-\varepsilon^{\prime\prime})p_2
    \\ &= r_2
\end{align*}

Thus $TV(R_1^{\otimes N},R_2^{\otimes N})=0$.
As in \cite[proof of Lemma 2.2]{prasadan_information_2025}, define the set of choices of $s$ of the $N$ indices as $\binom{[N]}{s} \coloneqq \big\{(i_1,...,i_s): i_1,\dots,i_s \in[N], i_1 < \dots < i_s \big\}$ and let
\begin{align*}
    S_J^L(B) &= \sum_{s=0}^B \sum_{I \in \binom{[N]}{s}}(1-\varepsilon^{\prime\prime})^{N-s}(\varepsilon^{\prime\prime})^s \prod_{i \in [N] \setminus I} p_J(x_i) \prod_{i \in I}q_J(x_i)
    \\S_J^U(B) &= \sum_{s=B+1}^N \sum_{I \in \binom{[N]}{s}}(1-\varepsilon^{\prime\prime})^{N-s}(\varepsilon^{\prime\prime})^s \prod_{i \in [N] \setminus I} p_J(x_i) \prod_{i \in I}q_J(x_i)
\end{align*} 
for integer $B$.
Note that
\begin{align*}
    &\frac{1}{2}\int\bigg|S_1^U(B) - S_2^U(B)\bigg|\dx 
    \\ &= \frac{1}{2}\int \bigg|\sum_{s=B+1}^N \sum_{I \in \binom{[N]}{s}}(1-\varepsilon^{\prime\prime})^{N-s}(\varepsilon^{\prime\prime})^s \bigg(\prod_{i \in [N] \setminus I} p_1(x_i) \prod_{i \in I}q_1(x_i) - \prod_{i \in [N] \setminus I} p_2(x_i) \prod_{i \in I}q_2(x_i)\bigg) \bigg|\dx
    \\ &\overset{(i)}{\leq}  \frac{1}{2}\sum_{s=B+1}^N \sum_{I \in \binom{[N]}{s}}(1-\varepsilon^{\prime\prime})^{N-s}(\varepsilon^{\prime\prime})^s \int\bigg|\bigg(\prod_{i \in [N] \setminus I} p_1(x_i) \prod_{i \in I}q_1(x_i) - \prod_{i \in [N] \setminus I} p_2(x_i) \prod_{i \in I}q_2(x_i)\bigg) \bigg|\dx
    \\ &\leq \sum_{s=B+1}^N \binom{N}{s}(1-\varepsilon^{\prime\prime})^{N-s}(\varepsilon^{\prime\prime})^s
    \\ &= \Prob(\operatorname{Bin}(N,\varepsilon^{\prime\prime}) > B)
\end{align*}
where $(i)$ is due to the triangle inequality and we can exchange the sum and integral since the interior is strictly positive.
Now,
\begin{align*}
    TV(R_1^{\otimes N},&R_2^{\otimes N}) 
    \\ &= \frac{1}{2}\int \bigg|\prod_{i \in [N]}\big((1-\varepsilon^{\prime\prime})p_{1}(x_i) + \varepsilon^{\prime\prime} q_1(x_i)\big) - \prod_{i \in [N]}\big((1-\varepsilon^{\prime\prime})p_{2}(x_i) + \varepsilon^{\prime\prime} q_2(x_i)\big)\bigg|\dx
    \\ &\geq \frac{1}{2}\int |S_1^L(B) - S_2^L(B)|\dx - \frac{1}{2}\int |S_1^U(B) - S_2^U(B)| \dx 
    \\ &\geq \frac{1}{2}\int |S_1^L(B) - S_2^L(B)|\dx -\Prob(\operatorname{Bin}(N,\varepsilon^{\prime\prime}) > B)
\end{align*}
Hoeffding's inequality gives $\Prob(\operatorname{Bin}(N,\varepsilon^{\prime\prime}) - N\varepsilon^{\prime\prime} \geq t) \leq \exp(-2t^2/N)$ and $\Prob(\operatorname{Bin}(N,\varepsilon^{\prime\prime}) - N\varepsilon^{\prime\prime} \leq t) \geq 1- \exp(-2t^2/N)$.
Let $t = \sqrt{\frac{N}{2}}$ then $\Prob\bigg(\operatorname{Bin}(N,\varepsilon^{\prime\prime})> N\big(\varepsilon^{\prime\prime} + \frac{1}{\sqrt{2N}}\big)\bigg) \leq \exp(-1)$ and $\Prob\bigg(\operatorname{Bin}(N,\varepsilon^{\prime\prime})\leq N\big(\varepsilon^{\prime\prime} + \frac{1}{\sqrt{2N}}\big)\bigg) \geq 1 - \exp(-1)$. Thus
\begin{align*}
    \frac{\frac{1}{2}\int |S_1^L\big(\lfloor N(\varepsilon^{\prime\prime}+1/\sqrt{2N})\rfloor\big) - S_2^L\big(\lfloor N(\varepsilon^{\prime\prime}+1/\sqrt{2N})\big)\rfloor|\dx}{\Prob\bigg(\operatorname{Bin}(N,\varepsilon^{\prime\prime})\leq N\big(\varepsilon^{\prime\prime} + \frac{1}{\sqrt{2N}}\big)\bigg)} \leq \frac{\exp(-1)}{1-\exp(-1)}< 1.
\end{align*}
Now, we can see that the left-hand side is the TV distance between two conditional mixtures. For $J\in\{1,2\}$ and given $N$ observations of $R_J$, let $W_J$ be the number of times $Q_J$ is drawn in place of $P_J$. Note that this is a $\operatorname{Bin}(N,\varepsilon^{\prime\prime})$ random variable. Then define $R_J^\prime$ as the distribution $R_J^{\otimes N}$ conditional on fewer than $\varepsilon^{\prime\prime} + \frac{1}{\sqrt{2N}}$ proportion of the entries being corrupted to distrubtion $Q_J$. Specifically, $R_J^\prime = R_J^{\otimes N} \big|\big\{W_J \leq N(\varepsilon^{\prime\prime} + \frac{1}{\sqrt{2N}})\big\}$. Also note that $N\big(\varepsilon^{\prime\prime} + 1/\sqrt{2N}\big) \leq N\varepsilon$. Thus, $TV(R^{\prime}_1,R^{\prime}_2) < 1$, so by Le Cam's Two-Point Lemma \cite[Section 15.2]{wainwright_highdimensional_2019} we have: 
\begin{align*}
    \inf_{\hat\Sigma} \sup_{\Sigma} \sup_{\mathcal{C}} \E \| \hat\Sigma(\mathcal{C}(X)) - \Sigma \|_F^2 & \gtrsim \|\Sigma_1 - \Sigma_2\|_F^2\frac{1}{2}(1 - TV(R^{\prime}_1,R^{\prime}_2))
    \\ &\gtrsim \frac{\varepsilon^{\prime 2}}{(1 - \varepsilon^{\prime})^2} \wedge (d/3)^2
    \\ &\gtrsim \varepsilon^{\prime 2} \wedge d^2
    \\ &\asymp \varepsilon^{ 2} \wedge d^2
\end{align*}

\end{delayedproof}

\section{Proofs for \Cref{sec:upper_bound} (upper bound)}\label[appendix]{appendix:upper_bound_pfs}

\subsection{Preliminaries}

\begin{lemma}\label[lemma]{lemma:D4_bounded}
    Let $\Upsilon \coloneqq C_{SE}c_1$ for some $C_{SE} \geq \sqrt{8}$ such that $C_{SE}c_1\geq \sup\limits_{\substack{\|A\|_F=1\\A =A\T}}\|\langle A,\tilde\xi\tilde\xi\T-\Sigma\rangle\|_{\psi_1}$ for all $\Sigma \in K$ where $\lambda_{\max}(\Sigma) \leq c_1$. Then $\operatorname{Var}(\tilde{V}) \leq \Upsilon^2 \lesssim c_1^2$.
\end{lemma}

\begin{delayedproof}{lemma:D4_bounded}
    let $\eta = \Sigma^{-1/2}\tilde\xi$ so $\eta \sim N(0,\mathbb{I})$ and $O \coloneqq \Sigma^{1/2}A\Sigma^{1/2} = U\Lambda U\T = \sum_i \lambda_i u_i u_i\T$ the eigendecomposition of $O$.
    \begin{align*}
        \langle A, \tilde\xi\tilde\xi\T - \Sigma\rangle &= \langle \Sigma^{1/2}A\Sigma^{1/2}, \eta\eta\T - \mathbb{I}\rangle 
        \\ &= \langle O, \eta\eta\T - \mathbb{I}\rangle
        \\ &= \sum_{i} \lambda_i((u_i\T \eta)^2 - 1)
    \end{align*}
    By rotational invariance, $u_i\eta$ are iid standard normal random variables so \textcite[Proposition 5.16]{vershynin_introduction_2011} gives (for $t \geq 0$ and for $K = \|N(0,1)^2-1\|_{\psi_1}$)
    \begin{align*}
        \Prob(|\langle A, \tilde\xi\tilde\xi\T - \Sigma)| \geq t) &\leq 2\expfunc{-c\min\bigg(\frac{t^2}{K^2\|\lambda\|_2^2} , \frac{t}{K\|\lambda\|_{\infty}}\bigg)}
        \\ &=  2\expfunc{-c\min\bigg(\frac{t^2}{K^2\|O\|_F^2} , \frac{t}{K\|O\|_{2}}\bigg)}
    \end{align*}
    taking $t$ large shows this is at worst case $\|\langle A, \tilde\xi\tilde\xi\T - \Sigma\rangle\|_{\psi_1} \lesssim \|O\|_F$, and  \[\|O\|_F \leq \|\Sigma^{1/2}\|_2\|A\|_F\|\Sigma^{1/2}\|_2 \leq c_1.\]
    Finally, $\operatorname{Var}(\tilde{V}) = 8 \bigg\|\Sigma^{1/2}\frac{\Sigma_1-\Sigma_2}{\|\Sigma_1 - \Sigma_2\|_F}\Sigma^{1/2}\bigg\|_F^2 \leq 8c_1^2$ and the second conclusion follows.
\end{delayedproof}

\begin{lemma}\label[lemma]{lemma:2r_bound}
    For any $\delta > 0$, $c \geq 2$, and $K$ star-shaped, 
    \[
        \log \mkloc(\delta,2c) \leq 2\log\mkloc\bigg(\frac{\delta}{2},c\bigg).
    \]
\end{lemma}
\begin{delayedproof}{lemma:2r_bound}
    Fix some $x \in K$ and choose a $\delta/(2c)$ packing $P \subseteq B(x,\delta) \cap K$. Within $P$, further choose a maximal $\delta/c$-separated subset $Q$, and note that $|Q| \leq \mkloc(\delta,c)$. Further note that since $Q$ is a maximal subset, it is a $\delta/c$-covering of $P$. Cluster $P$ by assigning each point in $P$ to a point in $Q$ less than $\delta/c$ distance from said point. For each $q \in Q$, denote its cluster $R_q \subseteq P$. 
    Now, since $c \geq 2$ we have that $\frac{\delta}{c} \leq \frac{\delta}{2}$ and thus $R_q \subseteq B\big(q,\frac{\delta}{2}\big) \cap K$. Recall that points in $R_q$ are $\delta/(2c)$ separated, and thus $R_q$ contains at most $\mkloc\big(\frac{\delta}{2},c\big)$ points. Now, over $|Q|$ clusters we have that \[|P| \leq |Q|\mkloc\bigg(\frac{\delta}{2},c\bigg) \leq \mkloc(\delta,c)\mkloc\bigg(\frac{\delta}{2},c\bigg) \leq \mkloc\bigg(\frac{\delta}{2},c\bigg)^2,\] with the last inequality from monotonicity. Taking the supremum over $P$ and $x$ then taking logs completes the proof.
\end{delayedproof}

\begin{lemma}[Useful Constants] \label[lemma]{lemma:useful_constants}
    Let $\varepsilon \in (0,1/32)$ be a known constant. Define 
    \begin{align*}
        g(t) &\coloneqq \bigg(\frac{1}{2} + t\bigg) \log\bigg(\frac{1}{2} + t\bigg) + \bigg(\frac{1}{2} - t\bigg) \log\big(1 - 2t\big)
    \end{align*}
    There exist positive constants $C$, $C_1$, $C_2$, $C_3$, $C_4$, $C_6$, $D_1$, $D_2$, $D_3$, $D_4$, $D_5$, $D_6$, and $\alpha \in \Big(0,\frac{1}{2}\Big)$ with the following properties: Define $D_1$ as some constant such that $\big(\E|\bar{V}|^p\big)^{1/p} \leq \Upsilon D_1p$ for $\bar{V} = \tilde{V} - \E\tilde{V}$ as defined in the proof of lemma \ref{lemma:trimmed_mean}. Let $D_3 \coloneqq \sqrt{\frac{8}{3}\big(12\log 4 + 12C_2\big)}$, $D_4 \coloneqq 8 + 3D_3^2/8$, $D_5 \coloneqq 2D_2 \sqrt{D_4} + C_6$, $D_6 \coloneqq \Big(1 - \frac{D_2D_4}{C_1(C-2-C_6)}\Big)$, $C_6 \coloneqq 2 \sqrt{\log 4 + C_2}$, and $\rho \coloneqq \exp\bigg\{-C_3 \min\Big(\frac{(C-2)^2\delta^2}{\Upsilon^2}, \frac{(C-2)\delta}{\Upsilon}\Big)\bigg\}$. Then
    \begin{enumerate}[label=(\roman*)]
        \item $C_6 < D_5 < C - 2$ \label{constant_property_1}
        \item $D_6 \in (0,1)$ \label{constant_property_2}
    \end{enumerate}
    Furthermore, if $\deltafunc \geq \thenewerboundary$: 
    \begin{enumerate}[label=(\roman*)]\setcounter{enumi}{2}
        \item $\exp\bigg\{-C_3^\star \min\Big(\frac{(C-2)^2\delta^2}{\Upsilon^2}, \frac{(C-2)\delta}{\Upsilon}\Big)\bigg\} \leq \frac{1}{2} \exp\bigg\{-C_3 \min\Big(\frac{(C-2)^2\delta^2}{\Upsilon^2}, \frac{(C-2)\delta}{\Upsilon}\Big)\bigg\}$, where $C_3^\star$ is the sub-exponential constant given in the proof of Lemma \ref{lemma:median}.\label{constant_property_4}
        \item $\Big(\frac{1}{2} - \alpha \Big) \log\Big(\frac{1}{\rho} \Big) \geq -2g(\alpha)$ \label{constant_property_5}
        \item $\varepsilon < \alpha(1-\rho)$ \label{constant_property_6}
    \end{enumerate}
\end{lemma}

\begin{delayedproof}{lemma:useful_constants}
	First, define the following helpful functions: 
    \begin{itemize}
        \item $f(a) \coloneqq -\frac{2g(a)}{\frac{1}{2} - a}$
        \item $h(a) \coloneqq a \Big(1 - \exp\big\{-f(a)\big\}\Big)$.
    \end{itemize}
	Next, pick the constants as follows:
    \begin{enumerate}[label=\arabic*.] 
        \item Choose $C_2 > 0$ arbitrarily.
        \item Choose $\alpha \in \Big(0,\frac{1}{2}\Big)$ s.t. $h(\alpha) > \frac{1}{32}$. Note that this is possible as $h$ is a surjection onto $\Big(0,\frac{1}{2}\Big)$, as proved in Lemma B.1 in \cite{prasadan_information_2025}. \label{choose_alpha} 
        \item Pick $C_3 = C_3^\star/2$. \label{choose_C_3}
        \item Choose $C$ large enough to satisfy
        \begin{align}
            C > 2 + \Bigg( D_5  \bigvee 2\sqrt{D_4} \max\bigg\{\sqrt{\frac{2\log 2}{C_3^\star}},\frac{2\log 2}{C_3^\star}\bigg\} \Bigg) \label{eq:choose_C_ineq}
        \end{align} 
        and
        \begin{align}
             C_3 \min\Big(\frac{(C-2)^2}{4D_4},\frac{(C-2)}{2\sqrt{D_4}}\Big) \geq f(\alpha). \label{choose_C_fa}
        \end{align}\label{choose_C}
        given the choice of $\alpha$. 
        \item Finally, choose $C_1 > \frac{D_2D_4}{C-2-C_6}$. \label{choose_C_1_star}
    \end{enumerate}
    Now that we have specified how we choose our constants, we show that they satisfy \ref{constant_property_1}-\ref{constant_property_6}.
    \begin{enumerate}[label=(\roman*)]
        \item Clearly holds from Step \ref{choose_C} above and the definition of $D_5$.
        \item Clearly holds from Steps \ref{choose_C} and \ref{choose_C_1_star} above.
        \item Let $J(\delta,C) \coloneqq \min\Big(\frac{(C-2)^2\delta^2}{\Upsilon^2}, \frac{(C-2)\delta}{\Upsilon}\Big)$
       {\small \begin{align*}
            \exp\big\{-C_3^\star J(\delta,C)\big\} &\overset{(i)}{\leq} \exp\bigg\{-\Big(\frac{1}{J(\delta,C)}\log 2 + C_3\Big)J(\delta,C)\bigg\}
            \\ &= \exp\big\{-\log2-C_3 J(\delta,C)\big\}
            \\ &= \frac{1}{2} \exp\big\{-C_3 J(\delta,C)\big\},
        \end{align*}}
        where $(i)$ holds if $0 < C_3 \leq C_3^\star - \frac{\log 2}{J(C,\delta)}$. Let us show this holds for our choice of $C$ and $C_3$. First note that setting $C_3=C_3^\star/2$ we can re-arrange this as $J(C,\delta)\geq \frac{2\log 2}{C_3^\star}$. Thus, it suffices to show that $\frac{(C-2)^2\delta^2}{\Upsilon^2} \geq \frac{2 \log 2}{C_3^\star}$ and $\frac{(C-2)\delta}{\Upsilon} \geq \frac{2\log 2}{C_3^\star}$. From the bound on $\frac{\delta^2}{\Upsilon^2}$ and choice of C in step \ref{choose_C} we have that $(C-2)\delta/\Upsilon \geq \frac{C-2}{2\sqrt{D_4}} \geq \frac{2\log2}{C_3^\star}$. The squared version holds analogously.
         \item First note we can rearrange this so the inequality we want to show is $\log\Big(\frac{1}{\rho}\Big) \geq f(\alpha)$. Substitute in $\rho \coloneqq \exp\bigg\{-C_3 \min\Big(\frac{(C-2)^2\delta^2}{\Upsilon^2}, \frac{(C-2)\delta}{\Upsilon}\Big)\bigg\}$ and calculate
        \begin{align*}
            \log(1/\rho) &= C_3\min\Big(\frac{(C-2)^2\delta^2}{\Upsilon^2}, \frac{(C-2)\delta}{\Upsilon}\Big)
            \\ & \overset{(i)}{\geq} C_3 \min\Big(\frac{(C-2)^2}{4D_4},\frac{(C-2)}{2\sqrt{D_4}}\Big)
            \\ &\overset{(ii)}{\geq} f(\alpha)
        \end{align*}
        where $(i)$ is due to the bound on $\frac{\delta^2}{\Upsilon^2}$ and $(ii)$ is due to the construction of $C$ in Step \ref{choose_C}
        \item 
        \begin{align*}
            \alpha(1-\rho) &\overset{(i)}{\geq} h(\alpha) 
            \\ &\overset{(ii)}{>} \frac{1}{32}
            \\ &\overset{(iii)}{\geq} \varepsilon
        \end{align*}
        where $(i)$ is the bound on $\rho$ calculated in the proof of Property \ref{constant_property_4}, $(ii)$ is due to the construction of $\alpha$ in Step \ref{choose_alpha} above, and $(iii)$ is due to the condition on $\varepsilon$ in Condition \ref{cond:corruption}.
    \end{enumerate}
\end{delayedproof}

With this, we can proceed to the proof of Lemma \ref{lemma:trimmed_mean}.

\subsection{Proof of Lemma \ref{lemma:trimmed_mean}}
\begin{delayedproof}{lemma:trimmed_mean}
	Let $ C, C_1 ,C_2 $ and $ C_6  $ be the constants from Lemma \ref{lemma:useful_constants}. Denote $ \delta _0:=\exp\Big[ -C_2N\deltafunc \Big] $; $ \tilde{\varepsilon } := 8\varepsilon + 12\log(4/\delta _0)/N $. We first  verify the following conditions so that we can apply the result of Lugosi and Mendelson \cite{lugosi_robust_2021}: 
    Define the confidence level as $ \delta _0 \coloneqq \exp\Big\{ -C_2N\deltafunc  \Big\} $. Thus, by the condition that $ \deltafunc \leq \thenewerboundary$:
        \begin{align*}
             \exp\bigg\{ -C_2N\deltafunc \bigg\} \geq & \exp\big\{ -C_2N\cdot \thenewerboundary \big\} > \exp\{-N\}/4.
        \end{align*}
    Now we show that the quantile level $ \tilde{\varepsilon } $ as defined in Lugosi and Mendelson \cite{lugosi_robust_2021} is valid. We have
    \begin{align*}
        \tilde{\varepsilon } &= 8\varepsilon + 12\frac{ \log \bigg(4/\exp\Big\{ -C_2N\deltafunc \Big\}\bigg) }{ N }
        \\ &= 8\varepsilon + \frac{ 12\log 4 }{ N } + 12C_2\deltafunc
        \\ &\overset{(i)}{ \leq } 8\varepsilon +  \deltafunc\cdot \big(12\log 4 + 12C_2\big)
        \\ &\overset{(ii)}{ \leq }  8\varepsilon + \thenewerboundary\cdot \big(12\log 4 + 12C_2\big)
        \\ & \overset{(iii)}{ < } \frac{ 1 }{ 2 } 
    \end{align*}
    where $ (i) $ is due to the fact that  $ \deltafunc \geq N^{-1} $, $ (ii) $ is due to the fact that $ \deltafunc \leq \thenewerboundary $, and $ (iii) $ is due to the condition that $ \varepsilon <1/32 $.
    Now, let $m = \E \tilde{V}_i$ and $|\bar{V}| \coloneqq \big|\tilde{V} - m\big|$. With probability $1-\delta_0$: 
    \begin{align}
        \big|TM(\{V\}_{i=1}^{2N})-m\big| &\leq \underbrace{\vphantom{\sqrt{\frac{\log(4/\delta_0)}{N}}} 3 \mathcal{E}\big(4\tilde{\varepsilon}\big) }_\textbf{A} + \underbrace{2\sigma_V\sqrt{\frac{\log(4/\delta_0)}{N}}}_\textbf{B} \label{eq:tm_bound1}
    \end{align}
    due to Theorem 1 in \cite{lugosi_robust_2021} and where $\sigma_V^2$ is the variance of $\tilde{V}$. Here we use $\mathcal{E}$ as defined in \cite{lugosi_robust_2021}: 
    \begin{align*} 
        \mathcal{E}(\eta) \coloneqq \max\Big\{\E\big[|\bar{V}|\indicator{\bar{V} \leq Q_{\eta/2}(\bar{V})}\big],\E\big[|\bar{V}|\indicator{\bar{V} \geq Q_{1-\eta/2}(\bar{V})}\big]\Big\}
    \end{align*}
    Now, we handle each part on the right hand side of \eqref{eq:tm_bound1} separately. 
    
    First we bound \textbf{B}
    \begin{align*}
        2\sqrt{\sigma_V^2\frac{\log\big(4/\delta_0\big)}{N}} &= 2\sqrt{\sigma_V^2\frac{\log\bigg(4/\exp\Big\{-C_2N\deltafunc\Big\}\bigg)}{N}}
        \\ &= 2\sqrt{\sigma_V^2 \bigg(\frac{\log 4}{N} + C_2\deltafunc\bigg)}
        \\ &\overset{(i)}{\leq} 2 \sqrt{\sigma_V^2\bigg(\deltafunc\log 4 + C_2\deltafunc\bigg)}
        \\ & \overset{(ii)}{\leq} 2 \sqrt{\delta^2\log 4 + C_2\delta^2} 
        \\ &= \delta \cdot 2 \sqrt{\log 4 + C_2} 
        \\ &\eqqcolon  C_6 \delta
    \end{align*} 
    where $(i)$ holds because $\deltafunc \geq N^{-1}$, $(ii)$ holds because $\sigma_V^2 \leq \Upsilon^2$ (\Cref{lemma:D4_bounded}), and we define $C_6 \coloneqq 2\sqrt{\log 4 + C_2}$.
    Next we bound \textbf{A}. Following the argument of the remark following Theorem 1 of Lugosi and Mendelson \cite{lugosi_robust_2021}, but with the sub-Exponential bound $(\E|\bar{V}|^p)^{1/p} \leq D_1 \Upsilon p$ replacing the sub-Gaussian bound and noting that for $p=\max\{2,\log(1/\tilde\varepsilon)\}$ \[3\mathcal{E}(4\tilde{\varepsilon}) \leq 3\|\bar{V}\|_{L^p} (2\tilde\varepsilon)^{1-1/p} \lesssim \Upsilon \tilde\varepsilon\log(1/\tilde{\varepsilon}),\] we have that for some constant $D_2$ (depending on $D_1$):
    \begin{align*}
        3\mathcal{E}(4\tilde{\varepsilon})&\leq D_2\Upsilon\tilde{\varepsilon}\log(1/\tilde\varepsilon)
        \\ &\overset{(i)}{\leq} D_2\Upsilon\tilde{\varepsilon}\log(1/\varepsilon)
        \\ &= D_2\Upsilon\bigg(8\varepsilon + 12\frac{\log(4/\delta_0)}{N}\bigg) \log(1/\varepsilon)
        \\ &= D_2\Upsilon\bigg(8\varepsilon + \frac{12\log4}{N} + \frac{12C_2\delta^2}{\Upsilon^2}\bigg) \log(1/\varepsilon)
        \\ &\overset{(ii)}{\leq} D_2\Upsilon\bigg(8\varepsilon + \frac{12\delta^2\log4}{\Upsilon^2} + \frac{12C_2\delta^2}{\Upsilon^2}\bigg) \log(1/\varepsilon)
        \\ &= D_2\Upsilon\bigg(8\varepsilon+ \frac{3D_3^2\delta^2}{8\Upsilon^2}\bigg) \log(1/\varepsilon)
    \end{align*}
    where $(i)$ is due to the fact that $\varepsilon \leq \tilde{\varepsilon}$ and $(ii)$ is due to $\frac{\delta^2}{\Upsilon^2} \geq N^{-1}$, defining \\$D_3 \coloneqq \sqrt{\frac{8}{3}(12\log 4 + 12C_2)}$. Now consider two sub-cases based on whether $\frac{\delta^2}{\Upsilon^2} \lessgtr \varepsilon$.
    \\\textbf{Case 1: $\frac{\delta^2}{\Upsilon^2} < \varepsilon$.}
            \begin{align*}
                3\mathcal{E}(4\tilde\varepsilon) \leq D_2\Upsilon\bigg(8\varepsilon+ \frac{3D_3^2}{8}\varepsilon\bigg) \log(1/\varepsilon) = D_2D_4\Upsilon\varepsilon\log(1/\varepsilon)
            \end{align*}
            for $D_4 \coloneqq 8 + 3D_3^2/8$
    \\\textbf{Case 2: $\frac{\delta^2}{\Upsilon^2} \geq \varepsilon$.}
            \begin{align*}
                \tilde\varepsilon = 8\varepsilon + \frac{12\log 4}{N} + \frac{12C_2\delta^2}{\Upsilon^2} \leq \big(8 + 12 \log 4 + 12C_2\big) \frac{\delta^2}{\Upsilon^2} = D_4 \frac{\delta^2}{\Upsilon^2}.
            \end{align*} Now, we have as well that $D_4 \frac{\delta^2}{\Upsilon^2}\leq D_4(\thenewerboundary)=1/4$. Let $s(x)$ be the map \\$s(x):x \mapsto x\log(1/x)$. Then $s$ is positive and real-valued, and upper bounded by 1 on $x\in(0,1)$. It is also increasing on $(0,1/e)$ and has limit $0$ as $x \to 0+$. It then follows that $\tilde\varepsilon \leq D_4\frac{\delta^2}{\Upsilon^2}\leq 1/4 \implies s(\tilde\varepsilon) \leq s\Big(D_4\frac{\delta^2}{\Upsilon^2}\Big)$ and $s\Big(\sqrt{D_4}\frac{\delta}{\Upsilon}\Big)<1$. 
            Now, we can apply the bound
            \begin{align*}
                3\mathcal{E}(4\tilde{\varepsilon}) &\leq D_2\Upsilon s(\tilde{\varepsilon})
                \\&\leq D_2\Upsilon \bigg(D_4\frac{\delta^2}{\Upsilon^2}\bigg)\log\bigg(\frac{\Upsilon^2}{\delta^2 D_4}\bigg)
                \\ &= 2D_2 \sqrt{D_4}\delta s\bigg(\sqrt{D_4}\frac{\delta}{\Upsilon}\bigg)
                \\&\leq 2D_2 \sqrt{D_4} \delta.
            \end{align*}
    Combining cases, with probability at least $1-\delta_0$ we have
            \begin{align}
                |TM_{\delta_0}(\{V_i\}_{i=1}^{2N}) - m| \leq \begin{cases}
                    C_6\delta + D_2D_4 \Upsilon\varepsilon\log(1/\varepsilon) & \text{if } \varepsilon > \frac{\delta^2}{\Upsilon^2}
                    \\ D_5\delta &\text{if }\varepsilon\leq \frac{\delta^2}{\Upsilon^2}
                \end{cases} \label{eq:tm_m_bounds}
            \end{align}
            for $D_5 \coloneqq 2D_2 \sqrt{D_4} + C_6$. 
    Now, since $C-2 > D_5 > C_6$ due to Lemma \ref{lemma:useful_constants} Property \ref{constant_property_1}, $m \leq -(C-2)\delta$, and $\frac{\delta^2}{\Upsilon^2} \leq \thenewerboundary$, we have
    \begin{align*}
        \Prob_{\Sigma \in B(\Sigma_1,\delta)} \Big( \mathrm{ TM } (\{V_i\}_{i=1}^{2N}) \geq 0  \Big) &\leq \Prob_{\Sigma \in B(\Sigma_1,\delta)} \Big( |\mathrm{ TM } (\{V_i\}_{i=1}^{2N}) - m| \geq (C-2)\delta \Big).
    \end{align*}
    We return to our cases on $\varepsilon$.
    \\\textbf{Case 1: $\varepsilon \leq \frac{\delta^2}{\Upsilon^2}$.} Applying \eqref{eq:tm_m_bounds} and $C-2> D_5$ we get
    \begin{align*}
        \Prob_{\Sigma \in B(\Sigma_1,\delta)} \Big( |\mathrm{ TM } (\{V_i\}_{i=1}^{2N}) - m| \geq (C-2)\delta \Big) &\leq  \Prob_{\Sigma \in B(\Sigma_1,\delta)} \Big( |\mathrm{ TM } (\{V_i\}_{i=1}^{2N}) - m| > D_5\delta \Big)
        \\ &\leq \delta_0
        \\ &= \exp\bigg\{-C_2N\frac{\delta^2}{\Upsilon^2}\bigg\}
    \end{align*}
    as required. 
    \\\textbf{Case 2: $\varepsilon > \frac{\delta^2}{\Upsilon^2}$.} Recall that we also have the bound $\frac{\delta^2}{\Upsilon^2} \geq C_1^2 \varepsilon^2\log^2(1/\varepsilon)$, so
    \begin{align*}
        \delta - \frac{D_2D_4}{C-2-C_6}\Upsilon s(\varepsilon) \geq \bigg(1 - \frac{D_2D_4}{C_1(C-2-C_6)}\bigg)\delta = D_6\delta
    \end{align*}
    for $D_6 \coloneqq \Big(1 - \frac{D_2D_4}{C_1(C-2-C_6)}\Big)$, due to $D_6 \in (0,1)$ from Lemma \ref{lemma:useful_constants} Property \ref{constant_property_2}.
    Now we can finally achieve the desired probability bound
    \begin{align*}
        \Prob_{\Sigma \in B(\Sigma_1,\delta)}& \Big( |\mathrm{ TM } (\{V_i\}_{i=1}^{2N}) - m| \geq (C-2)\delta \Big) 
        \\ &\overset{(i)}{=}\Prob_{\Sigma \in B(\Sigma_1,\delta)} \Big( |\mathrm{ TM } (\{V_i\}_{i=1}^{2N}) - m| \geq C_6\delta + D_2D_4\Upsilon s(\varepsilon) 
        \\ & \phantom{=} + (C-2-C_6)\delta - D_2D_4\Upsilon s(\varepsilon) \vphantom{X^2} \Big)
        \\ &=\Prob_{\Sigma \in B(\Sigma_1,\delta)} \Big(\vphantom{\frac{X}{Y}} |\mathrm{ TM } (\{V_i\}_{i=1}^{2N}) - m| \geq C_6\delta + D_2D_4\Upsilon s(\varepsilon) 
        \\ & \phantom{=} + (C-2-C_6)\Big(\delta - \frac{D_2D_4}{C-2-C_6}\Upsilon s(\varepsilon) \Big) \Big)
        \\ &\leq \Prob_{\Sigma \in B(\Sigma_1,\delta)} \Big( |\mathrm{ TM } (\{V_i\}_{i=1}^{2N}) - m| \geq C_6\delta + D_2D_4\Upsilon s(\varepsilon) 
        \\ &\phantom{=} + (C-2-C_6)D_6\delta \Big)
        \\ &\overset{(ii)}{\leq }\Prob_{\Sigma \in B(\Sigma_1,\delta)} \Big( |\mathrm{ TM } (\{V_i\}_{i=1}^{2N}) - m| > C_6\delta + D_2D_4\Upsilon s(\varepsilon) \Big)
        \\ &\overset{(iii)}{\leq} \exp\bigg\{-C_2N\frac{\delta^2}{\Upsilon^2}\bigg\}
    \end{align*}
    where $(i)$ is just a rearranging of terms, $(ii)$ is due to the fact that $C-2> C_6$, $D_6 > 0$ and $\delta > 0$, and $(iii)$ is the bound from \eqref{eq:tm_m_bounds}.
\end{delayedproof}

\begin{delayedproof}{lemma:median}
    It suffices to prove for arbitrary $N^\prime$ that
    \begin{align}
        \Prob_{\Sigma \in B(\Sigma_1,\delta)}&  \left(\left|\left\{i \in [N^\prime] : \|\xi_i\xi_i^\top - \Sigma_1\| \geq \|\xi_i\xi_i^\top - \Sigma_2\|\right\}\right| \geq \frac{N^\prime}{2}\right) \nonumber 
        \\ & \leq \expfunc{-C_4N \newRHS} 
    \end{align}
    Let constants $C$ be defined as in \Cref{lemma:trimmed_mean} and $ C_3,\alpha  $ as in \Cref{lemma:useful_constants}. We begin by proving the probability tail property of the uncorrupted $ \tilde{V_i} $ and then argue that the error rate of median estimator based on the corrupted $ V_i $ is bounded by that based on the uncorrupted $ \tilde{V_i} $.
    \\ Let $\hat{v} = \frac{\operatorname{vec}\left((\Sigma_2 - \Sigma_1)^{\top}\right)}{\|(\Sigma_2 - \Sigma_1)\|_F}$, then 
    \begin{align*}
        \Prob_{\Sigma \in B(\Sigma_1,\delta)}\left(\tilde{V}_i > 0\right) &\leq \Prob_\Sigma\left(2\hat{v}^\top\operatorname{Vec}(\xi_i\xi_i^\top - \Sigma) > (C-2)\delta \right) 
        \\ &\overset{(i)}{\leq} \expfunc{-C_3^\star \min\left(\frac{(C-2)^2\delta^2}{\Upsilon^2}, \frac{(C-2)\delta}{\Upsilon}\right)}
     \end{align*}
    where $(i)$ is due to in \textcite[Proposition 5.16]{vershynin_introduction_2011}. Let $J(\delta,C) \coloneqq \min\left(\frac{(C-2)^2\delta^2}{\Upsilon^2}, \frac{(C-2)\delta}{\Upsilon}\right)$ and note that
    \begin{align*}
        \Prob_{\Sigma \in B(\Sigma_1,\delta)}\left(\tilde{V}_i > 0\right) &\leq \expfunc{-C_3^\star J(\delta,C)}  
    \\ & \overset{(i)}{\leq} \frac{1}{2} \expfunc{-C_3J(\delta,C)} 
        \\ &= \frac{1}{2}\rho
    \end{align*}
    where $(i)$ is due to the definition of $C$ and $C_3$ in \Cref{lemma:useful_constants} and $\rho \coloneqq \expfunc{-C_3J(\delta,C)}$
    \\ Now consider two events: 
    \begin{align*}
        A_i &\coloneqq \left\{ \|\tilde{\xi}_i\tilde{\xi}_i^\top - \Sigma_1 \| \geq  \|\tilde{\xi}_i\tilde{\xi}_i^\top - \Sigma_2\| \right\}
        \\ B_i &\coloneqq \left\{\|\xi_i\xi_i^\top - \Sigma_1 \| \geq \|\xi_i\xi_i^\top - \Sigma_2\|\right\} 
    \end{align*}
    for uncorrupted data $\tilde{\xi}$ and corrupted data $\xi$. Next, define the following estimators based on these events: 
    \begin{align*}
        \psi &\coloneqq \indicator{\textrm{at least } \frac{N^\prime}{2} \textrm{ of } B_i \textrm{ are true}}
        \\ \tilde{\psi} &\coloneqq \indicator{\textrm{at least } \frac{N^\prime}{2} - N^\prime\alpha(1-\rho) \textrm{ of } A_i \textrm{ are true}}
    \end{align*}
    We will continue by showing that for each $\Sigma \in B(\Sigma_1,\delta)$, we can bound $\Prob_\Sigma(\tilde{\psi} = 1)$. Then, we will show that this is sufficient for our proof by proving that $\Prob_{\Sigma \in B(\Sigma_1,\delta)}(\psi = 1) \leq \Prob_{\Sigma \in B(\Sigma_1,\delta)}(\tilde{\psi} = 1)$.
    
    Notice that the event $\tilde{\psi}=1$ can be bounded by the tail probability of a binomial random variable, and set $\zeta = \left(\frac{1}{2} - \alpha\right)(1-\rho)$ and $p = 1 - \rho/2$, so  $p - \zeta = \frac{1}{2} + \alpha(1-\rho)$ and $\Prob(A_i^C) \geq p$. Then 
    \begin{align}
        \Prob_\Sigma(\tilde{\psi} = 1) &= \Prob\left(Bin\left(N^\prime,\Prob_\Sigma\left(A_i^C\right)\right) \leq \frac{N^\prime}{2} + N^\prime \alpha(1-\rho)\right) \nonumber
        \\ &= \Prob\left(Bin\left(N^\prime,\Prob_\Sigma\left(A_i^C\right)\right) \leq N^\prime(p - \zeta)\right) \nonumber 
        \\ &\leq \expfunc{-N^\prime D((p - \zeta)\| p)} \label{eq:kl_bin_ineq}
    \end{align}
    and thus by the same argument used in \textcite[proof of Lemma B.4]{prasadan_information_2025}: 
    \begin{align*}
        D(p - \zeta \| p) &\geq g(\alpha) + \left(\frac{1}{2} - \alpha\right) \log(1 /\rho)
        \\ &\overset{(i)}{\geq }\frac{1}{2}\left(\frac{1}{2} - \alpha\right) \log(1/\rho) 
        \\ &= \frac{1}{2}\left(\frac{1}{2} - \alpha \right) \log\left(\expfunc{C_3J(\delta,C)}\right)
        \\ &= \frac{1}{2}\left(\frac{1}{2} - \alpha \right) C_3J(\delta,C)
    \end{align*}
    where $(i)$ is due to \Cref{lemma:useful_constants} part \ref{constant_property_5}. Plugging this into (\ref{eq:kl_bin_ineq}) gives
    \begin{align*}
        \Prob_{\Sigma \in B(\Sigma_1,\delta)}(\tilde\psi = 1) &\leq \expfunc{-N^\prime \frac{1}{2}\left(\frac{1}{2} - \alpha\right) \log(1/\rho)}
        \\ &= \expfunc{-N^\prime \frac{1}{2}\left(\frac{1}{2} - \alpha \right) C_3J(\delta,C)}.
    \end{align*}
    With this, we have a bound for the uncorrupted case. Let us then show how this bounds the corrupted case. Assume $\psi = 1$ and thus the event $B_i$ evaluates to true for at least $\frac{N^\prime}{2}$ subjects. Due to the corruption, there are up to $\varepsilon N^\prime$ subjects for which $A_i$ and $B_i$ can evaluate differently, so in the worst case at least $\frac{N}{2} - N^\prime \varepsilon$ of the $A_i$ are true. Furthermore, by \Cref{lemma:useful_constants} Part \ref{constant_property_6} $\alpha(1-\rho) > \varepsilon$. Therefore, if $\psi = 1$ then at least $\frac{N^\prime}{2} - \varepsilon N^\prime > \frac{N^\prime}{2} - N^\prime\alpha(1-\rho)$ of the $A_i$ evaluate to true and thus $\tilde\psi = 1$. Therefore we can bound $\Prob(\psi = 1) \leq \Prob(\tilde\psi = 1)$ and get
    \begin{align*}
        \Prob_{\Sigma \in B(\Sigma_1,\delta)} (\psi = 1) &\leq \Prob_{\Sigma \in B(\Sigma_1,\delta)}\left(\tilde{\psi} = 1\right)
        \\ &\leq \expfunc{-N^\prime\frac{1}{2}\left(\frac{1}{2} - \alpha \right) C_3J(\delta,C)}.
        \\ &= \expfunc{-C_4NJ(\delta,C)},
    \end{align*}
    setting $N^\prime = 2N$ and $C_4 \coloneqq \Big(\frac{1}{2} - \alpha\Big)C_3$. The argument for $\Prob_{\Sigma}(\tilde{\psi}=0)$ is exactly symmetric and so is excluded. 
\end{delayedproof}

\subsection{Proof of \Cref{theorem:upper_bound}} \label[appendix]{appendix:upper_bound_rate}
Before we can prove \Cref{theorem:upper_bound} we must establish a number of lemmas. We begin by bounding the distance between each element in the ``winning'' path of \Cref{algorithm:robust_algorithm} and the true covariance $\Sigma$. 
\begin{lemma} \label[lemma]{lemma:bound_yj_sigma}
    Let $C,C_5$ be the universal constants from \Cref{theorem:two_point_prob_bound} and $C_1$ from \Cref{lemma:useful_constants}. Denote $\delta_J \coloneqq \frac{d}{2^{J-1}(C+1)}$, and let $J^\dagger\in \mathbb{N}$ be the largest $J$ such that 
    \begin{align}
        C_5 N \min\left(\frac{\delta_J^2}{\Upsilon^2},\frac{\delta_J}{\Upsilon}\right) &\geq 4 \log \mkloc(c\delta_J,2c)\bigvee\log 2. \label{eq:delta_mk_ineq}
        \shortintertext{and}
        \frac{\delta_J^2}{\Upsilon^2} &\geq N^{-1} \vee C_1^2\varepsilon^2\log^2(1/\varepsilon). \nonumber
    \end{align}
    If these conditions do not hold for any $J\in\mathbb{N}$ then set $J^\dagger = 1$. Denote $\mathcal{Y}_J$ as the $J^{th}$ element in the ``winning'' path from \Cref{algorithm:robust_algorithm}, $\mathcal{Y}$ for $1 \leq J \leq J^{\dagger}$.\\ Then
    \begin{align}
        \Prob\left(\|\mathcal{Y}_J - \Sigma\|_F > (C+1)\delta_J\right) \leq 2\indicator{J>1} \expfunc{-\frac{1}{2}C_5N\min\bigg(\frac{\delta_J^2}{\Upsilon^2},\frac{\delta_J}{\Upsilon}\bigg)}. \label{eq:y_sigma_frob}
    \end{align}
\end{lemma}
\begin{delayedproof}{lemma:bound_yj_sigma}
    First note that if $C_5 N \newRHS \geq 4 \log \mkloc(c\delta_J,2c)\bigvee\log 2$ or $\frac{\delta_J^2}{\Upsilon^2}\geq N^{-1} \vee C_1^2\varepsilon^2\log^2(1/\varepsilon)$ never hold, we set $J^\dagger = 1$ and \Cref{eq:y_sigma_frob} holds trivially. 
    \\ Define $A_j \coloneqq \left\{\|\mathcal{Y}_j - \Sigma\|_F> \frac{d}{2^{j-1}}\right\}$. Then by \textcite[Lemma B.4 (i)]{prasadan_information_2025} we have 
    \begin{align*}
        \Prob(A_J) \leq \Prob(A_1) + \Prob\left(A_2 \cap A_1^C\right) + \suml{j=3}{J} \Prob\left(A_j \cap A_{j-1}^C\right).
    \end{align*}
    Because $\frac{\delta_j^2}{\Upsilon^2} \geq N^{-1} \vee C_1^2\varepsilon^2\log^2(1/\varepsilon)$ for every $1 \leq j \leq J \leq J^\dagger$, we can apply \Cref{theorem:covering_bound} to each layer. We will analyze each term in this sum separately: 
    \begin{itemize}
        \item {$\Prob\left(A_j \cap A_{j-1}^C\right)$}
        \begin{align*}
            &\overset{(i)}{\leq} \suml{u \in \mathcal{L}(j-1)\cap B\left(\Sigma,\frac{d}{2^{j-2}}\right)}{} \Prob\left(\|\mathcal{Y}_j - \Sigma\|_F \geq \frac{d}{2^{j-1}},\mathcal{Y}_{j-1} = u\right)
            \\ &\overset{(ii)}{=}\suml{u \in \mathcal{L}(j-1)\cap B\left(\Sigma,\frac{d}{2^{j-2}}\right)}{} \Prob\left(\left\|\argmin_{\Sigma \in \mathcal{O}(u)}T(\delta_j,\Sigma,\mathcal{O}(u)) - \Sigma\right\|_F \geq (C+1)\delta_j,\mathcal{Y}_{j-1} = u\right)
            \\ &\leq \suml{u \in \mathcal{L}(j-1)\cap B\left(\Sigma,\frac{d}{2^{j-2}}\right)}{} \Prob\left(\left\|\argmin_{\Sigma \in \mathcal{O}(u)}T(\delta_j,\Sigma,\mathcal{O}(u)) - \Sigma\right\|_F \geq (C+1)\delta_j\right)
            \\ &\overset{(iii)}{\leq} \sum_{u \in \mathcal{L}(j-1)\cap B\left(\Sigma,\frac{d}{2^{j-2}}\right)} \mkloc\left(\frac{d}{2^{j-2}},2c\right) \expfunc{-C_5 N\min\left(\frac{\delta_j^2}{\Upsilon^2},\frac{\delta_j}{\Upsilon}\right)}
            \\ &\overset{(iv)}{\leq} \left[\mkloc\left(\frac{d}{2^{j-2}},2c\right)\right]^2 \expfunc{-C_5N \min\left(\frac{\delta_j^2}{\Upsilon^2},\frac{\delta_j}{\Upsilon}\right)}
        \end{align*} where the justification for each step is as follows:
        \begin{enumerate}[label=\textit{(\roman*)}]
            \item $\mathcal{Y}_{j-1}$ must be some point $u$ in the discrete set $\mathcal{L}(j-1) \cap B\left(\Sigma,\frac{d}{2^{j-2}}\right)$.
            \item This follows directly from \Cref{definition:tournament_T}.
            \item By \Cref{lemma:tree_properties} $\mathcal{O}(u)$ is a $\frac{d}{2^{j-1}(C+1)}$ covering of $K \cap B\big(\mathcal{Y}_{j-1},\frac{d}{2^{j-2}}\big)$. Then we can apply \Cref{theorem:covering_bound}.
            \item By \Cref{lemma:tree_properties} \Cref{tree_property_3}
        \end{enumerate}
        \item $\mathbb{P}(A_1) = 0$ trivially
        \item $\Prob(A_2 \cap A_1^C)$\\
        Note that from \Cref{lemma:tree_properties} $A_2$ is a $d/c$-covering of $K \cap B(\Sigma,d)$. Then 
        \begin{align*}
            \Prob\left(A_2 \cap A_1^C\right) &= \Prob(A_2)
            \\ &\leq \Prob\left(\left\|\argmin_{\Sigma \in \mathcal{O}(1)} T(\delta_j,\Sigma,\mathcal{O}(1)) - \Sigma\right\|_F \geq (C+1)\delta_2\right)
            \\ &\leq \mkloc(d,c) \expfunc{-C_5N\min\left(\frac{\delta_2^2}{\Upsilon^2},\frac{\delta_2}{\Upsilon}\right)}
            \\ &\leq \left[\mkloc(d,2c)\right]^2 \expfunc{-C_5N\min\left(\frac{\delta_2^2}{\Upsilon^2},\frac{\delta_2}{\Upsilon}\right)}
        \end{align*} where the final inequality is due to the fact that $\mkloc(d,c)$ is a positive integer.
    \end{itemize}
    Finally we combine the above results as 
    \begin{align*}
        \Prob(A_J) &\leq \Prob(A_1) + \Prob\left(A_2 \cap A_1^C\right) + \suml{j=3}{J} \Prob\left(A_j \cap A_{j-1}^C\right)
        \\ &\leq \sum_{j=2}^J \left[\mkloc\left(\frac{d}{2^{j-2}},2c\right)\right]^2 \expfunc{-C_5N\min\left(\frac{\delta_j^2}{\Upsilon^2},\frac{\delta_j}{\Upsilon}\right)}
        \\ &\leq \left[\mkloc\left(\frac{d}{2^{J-2}},2c\right)\right]^2 \sum_{j=2}^J \expfunc{-C_5 N \min\left(\frac{\delta_j^2}{\Upsilon^2},\frac{\delta_j}{\Upsilon}\right)}
        \\ &\overset{(i)}{\leq} \indicator{J>1} \left[\mkloc\left(\frac{d}{2^{J-2}},2c\right)\right]^2
        \frac{a_J}{1-a_J}
    \end{align*}
    where in $(i)$ we set $a_J \coloneqq \expfunc{-C_5 N \min\left(\frac{\delta_J^2}{\Upsilon^2},\frac{\delta_J}{\Upsilon}\right)}$ and see that the sum is bounded by a geometric series since $\min\big\{\frac{(2x)^2}{\Upsilon^2},\frac{2x}{\Upsilon}\big\} \geq 2\min\big\{\frac{(x)^2}{\Upsilon^2},\frac{x}{\Upsilon}\big\}$. Now due to the condition described in \Cref{eq:delta_mk_ineq}, we have $a_J < \frac{1}{2}$ and $\expfunc{\frac{C_5}{2}N\min\left(\frac{\delta_J^2}{\Upsilon^2},\frac{\delta_J}{\Upsilon}\right)} \geq   \left[\mkloc(c\delta_J, 2c)\right]^2$.
    \\ Then
    \begin{align*}
        \Prob(A_J) &\leq 2\indicator{J>1} \left[\mkloc\left(\frac{d}{2^{J-2}},2c\right)\right]^2 a_J
        \\ &\leq 2\indicator{J>1}\expfunc{\frac{C_5}{2}N\min\left(\frac{\delta_J^2}{\Upsilon^2},\frac{\delta_J}{\Upsilon}\right)}\expfunc{-C_5N\min\left(\frac{\delta_J^2}{\Upsilon^2},\frac{\delta_J}{\Upsilon}\right)}
        \\&= 2\indicator{J>1}\expfunc{-\frac{C_5}{2}N\min\left(\frac{\delta_J^2}{\Upsilon^2},\frac{\delta_J}{\Upsilon}\right)}
    \end{align*}
\end{delayedproof}

Now we have established three conditions under which the probability tail bound holds via \\\Cref{lemma:bound_yj_sigma}, namely
\begin{enumerate}[label=\Roman*.]
    \item $C_5 N \min\left(\frac{\delta_J^2}{\Upsilon^2},\frac{\delta_J}{\Upsilon}\right) \geq 4 \log \mkloc(c\delta_J,2c)\bigvee\log 2$, \label{bound_cond1}
    \item $\frac{\delta_J^2}{\Upsilon^2} \geq N^{-1}$, \label{bound_cond2}
    \item $\frac{\delta_J^2}{\Upsilon^2} \geq C_1^2\varepsilon^2\log^2(1/\varepsilon)$. \label{bound_cond3}
\end{enumerate}
We define the largest $J$ such that each condition holds as follows:
\begin{definition}
    Call $J^{\rom{1}}$ the largest $J$ such that condition \ref{bound_cond1} holds, $J^{\rom{2}}$ the largest $J$ such that condition \ref{bound_cond2} holds, and $J^{\rom{3}}$ the largest $J$ such that condition \ref{bound_cond3} holds. For each condition, if $J^r$ never holds set $J^r =0$. Then $J^\star$ from \Cref{theorem:upper_bound} can be rewritten as $J^\star = J^{\rom{1}} \vee 1$ and $J^\dagger$ from \Cref{lemma:bound_yj_sigma} can be rewritten as $J^\dagger = 1 \vee \min\left(J^{\rom{1}},J^{\rom{2}},J^{\rom{3}}\right)$. 
\end{definition}

\begin{lemma}\label[lemma]{lemma:sigma_star_star_bound}
    Denote $\Sigma^\star \coloneqq \mathcal{Y}_{J^\star}$ and the output at an some step $J^{\star\star} \geq J^\star$ as $\Sigma^{\star\star} = \mathcal{Y}_{J^{\star\star}+1}$. Then there exists a constant $\psi_C$ depending only on $C$ such that $\forall x \geq \delta_{J^\dagger}$
    \begin{align*}
        \Prob\left(\left\|\Sigma - \Sigma^{\star\star}\right\|_F > \psi_Cx\right) &\leq 2 \indicator{J^\star > 1}\expfunc{-\frac{1}{8}C_5N \min\left(\frac{x^2}{\Upsilon^2},\frac{x}{\Upsilon}\right)},
    \end{align*}
    and therefore $\forall \delta \geq \delta_{J^\dagger}$
    \begin{align*}
        \E_\xi \|\Sigma - \Sigma^{\star\star}\|_F^2 \lesssim \delta^2 + \frac{\Upsilon^2}{N}
    \end{align*}
\end{lemma}
\begin{delayedproof}{lemma:sigma_star_star_bound}
    For any $x \geq \delta_{J^\dagger}$ we can find a $\tilde{J} \leq J^\dagger$ s.t. $x \in [\delta_{\tilde{J}},\delta_{\tilde{J}-1})$. If we define $J \coloneqq \tilde{J}\vee 1$ then $1 \leq J \leq J^\dagger \leq J^\star \leq J^{\star\star}$. Then consider the decomposition \[\Sigma - \Sigma^{\star\star} = \Sigma - \mathcal{Y}_J + \mathcal{Y}_J - \Sigma^\star + \Sigma^\star - \Sigma^{\star\star}\] and note that \Cref{lemma:tree_properties} part \ref{tree_property_4} implies
    \begin{align}
        \lVert\mathcal{Y}_{J} - \Sigma^\star\rVert_F = \lVert\mathcal{Y}_{J} - \mathcal{Y}_{J^{\star}}\rVert_F &\leq \frac{d(2+4c)}{c2^J} = \frac{5+4C}{C+1}\frac{d}{2^J}, \label{eq:yj_sigma_star}
        \\ \lVert\Sigma^\star - \Sigma^{\star\star}\rVert_F = \lVert\mathcal{Y}_{J^\star} - \mathcal{Y}_{J^{\star\star}+1}\rVert_F &\leq \frac{d(2+4c)}{c2^{J^\star}} = \frac{5+4C}{C+1}\frac{d}{2^{J^\star}}.  \label{eq:sigma_sigma_star}
    \end{align}
    We can define $\psi_C \coloneqq 6 + 5C$ and then note that
    \begin{align}
        \psi_Cx &= (6 + 5C)x = \left(\frac{6 + 5C}{C+1}\right)(C+1)x \nonumber
        \\ &= \left(1+ \frac{5 + 4C}{C+1}\right)(C+1)x \nonumber
        \\ &= \frac{1}{2}\left(2 + \frac{5+4C}{C+1} + \frac{5+4C}{C+1}\right)(C+1)x \nonumber
        \\ &\overset{(i)}{\geq} \left(2 +\frac{5+4C}{C+1} + \frac{5+4C}{C+1}\right) \frac{d}{2^J} \nonumber
        \\ &\overset{(ii)}{\geq} \frac{d}{2^{J-1}} + \frac{5+4C}{C+1} \frac{d}{2^J} + \frac{5 + 4C}{C+1}\frac{d}{2^{J^\star}} \label{eq:psi_bound}
    \end{align} where $(i)$ and $(ii)$ are due to $x \geq \delta_{\tilde{J}} \geq \delta_J \geq \delta_{J^\star}$. 
    \\ Now we can bound $\|\Sigma - \Sigma^{\star\star}\|_F$ as follows
    \begin{align*}
        \Prob\left(\|\Sigma - \Sigma^{\star\star}\|_F > \psi_C x\right) &\leq \Prob\left(\|\Sigma - \mathcal{Y}_{J}\|_F + \|\mathcal{Y}_J - \Sigma^\star\|_F + \|\Sigma^\star - \Sigma^{\star\star}\|_F > \psi_C x\right)
        \\ &\overset{(i)}{\leq} \Prob \left( \|\mathcal{Y}_J - \Sigma\|_F + \frac{5+4C}{C+1}\frac{d}{2^J} + \frac{5+4C}{C+1}\frac{d}{2^{J^\star}} > \psi_C x\right)
        \\ &\overset{(ii)}{\leq} \Prob\left(\|\mathcal{Y}_J - \Sigma \|_F > \frac{d}{2^{J-1}}\right)
        \\ &\overset{(iii)}{\leq} 2\indicator{J>1}\expfunc{-\frac{1}{2}C_5N\min\left(\frac{\delta_J^2}{\Upsilon^2},\frac{\delta_J}{\Upsilon}\right)}
        \\ &\overset{(iv)}{\leq} 2\indicator{J>1}\expfunc{-\frac{1}{2}C_5N\min\left(\frac{\delta_{\tilde{J}}^2}{\Upsilon^2},\frac{\delta_{\tilde{J}}}{\Upsilon}\right)}
        \\ &\overset{(v)}{\leq} 2\indicator{J>1}\expfunc{-\frac{1}{8}C_5N\min\left(\frac{x^2}{\Upsilon^2},\frac{x}{\Upsilon}\right)}
    \end{align*}
    where $(i)$ is due to \Cref{eq:yj_sigma_star} and \Cref{eq:sigma_sigma_star}, $(ii)$ is due to the the definition of $\psi_C$ and the previous bound \Cref{eq:psi_bound}, $(iii)$ holds from \Cref{lemma:bound_yj_sigma}, $(iv)$ holds because $\delta_J = \delta_{\tilde{J}}$ when $J>1$, and finally $(v)$ is due to the fact that $x < \delta_{\tilde{J}-1}=2\delta_{\tilde{J}}$. Our probability bound holds by noting that $\indicator{J>1} \leq \indicator{J^\star> 1}$. We can obtain the expectation bound by integrating the probability bound over $x$:
   
    \begin{align*}
        \E_\xi \lVert\Sigma - \Sigma^{\star\star}\rVert_F^2 &= \int_0^{\infty} \Prob\big(\lVert\Sigma - \Sigma^{\star\star}\rVert_F^2 > x \big) \dx
        \\ &= 2\psi_C^2 \int_0^\infty x \Prob\big(\lVert\Sigma - \Sigma^{\star\star}\rVert_F > \psi_Cx\big) \dx
        \\ &\leq 2\psi_C^2\bigg[\int_0^\delta x \dx + \int_\delta^\infty x \Prob\big(\lVert\Sigma - \Sigma^{\star\star}\rVert_F > \psi_C x\big) \dx\bigg]
        \\ &\leq \psi_C^2\delta^2 + 4\psi_C^2 \indicator{J>1} \int_\delta^\infty x \expfunc{-\frac{1}{8} C_5N\min\bigg(\frac{x^2}{\Upsilon^2},\frac{x}{\Upsilon}\bigg)} \dx 
        \\ &\overset{(i)}{\lesssim} \delta^2 + \frac{\Upsilon^2}{N}
    \end{align*}
    where $(i)$ is due to the following:
    \\\textbf{Case 1: $\Upsilon > \delta$.}
        \begin{align*}
            &\int_\delta^\infty x \expfunc{-\frac{1}{8} C_5N\min\bigg(\frac{x^2}{\Upsilon^2},\frac{x}{\Upsilon}\bigg)} \dx 
            \\ &= \underbrace{\int_{\delta}^{\Upsilon}  x \expfunc{-\frac{1}{8} C_5N\frac{x^2}{\Upsilon^2}} \dx}_{\mathbf{A}} + \underbrace{\int_{\Upsilon}^\infty x \expfunc{-\frac{1}{8} C_5N\frac{x}{\Upsilon}} \dx}_{\mathbf{B}}
            \shortintertext{let $\gamma \coloneqq \frac{1}{8\Upsilon^2}C_5N$}
            \mathbf{A} &= \frac{1}{2\gamma}\bigg(e^{-\delta^2\gamma} - e^{-\Upsilon^2\gamma}\bigg)
            \\ &= \frac{1}{\frac{1}{4\Upsilon^2}C_5N} \bigg(\expfunc{-\delta^2 \frac{1}{8\Upsilon^2}C_5N} - \expfunc{-\Upsilon^2\frac{1}{8\Upsilon^2}C_5N}\bigg)
            \\ &= \frac{4\Upsilon^2}{C_5N} \bigg(\expfunc{-\frac{\delta^2}{8\Upsilon^2} C_5N} - \expfunc{-\frac{1}{8}C_5N}\bigg)
            \shortintertext{let $\nu \coloneqq \frac{1}{8\Upsilon}C_5 N$}
            \mathbf{B} &= \frac{1}{\nu^2} \bigg((\Upsilon \nu + 1)\expfunc{-\Upsilon\nu}\bigg)
            \\ &= \frac{64\Upsilon^2}{C_5^2 N^2}\bigg( \bigg(\Upsilon  \frac{1}{8\Upsilon}C_5 N + 1\bigg) \expfunc{-\Upsilon \frac{1}{8\Upsilon}C_5 N}\bigg) 
            \\ &= \frac{64\Upsilon^2}{C_5^2 N^2}\bigg( \bigg(\frac{1}{8}C_5 N + 1\bigg) \expfunc{-\frac{1}{8}C_5 N}\bigg) 
            \\ &= \frac{4\Upsilon^2}{C_5 N}\bigg( \bigg(2 + \frac{16}{C_5N}\bigg) \expfunc{-\frac{1}{8}C_5 N}\bigg)
            \shortintertext{so}
            \mathbf{A} + \mathbf{B} &= \frac{\Upsilon^2}{N} \times \frac{4}{C_5}\bigg(\expfunc{-\frac{\delta^2}{8\Upsilon^2} C_5N} +  \bigg(1 + \frac{16}{C_5N}\bigg) \expfunc{-\frac{1}{8}C_5 N}\bigg)
        \end{align*}
    \textbf{Case 2: $\Upsilon \leq \delta$.}
        \begin{align*}
            &\int_\delta^\infty x \expfunc{-\frac{1}{8} C_5N\min\bigg(\frac{x^2}{\Upsilon^2},\frac{x}{\Upsilon}\bigg)} \dx 
            \\ &\leq \int_{\Upsilon}^\infty x \expfunc{-\frac{1}{8} C_5N\frac{x}{\Upsilon}} \dx
            \\ &= \mathbf{B}
        \end{align*}
\end{delayedproof}

\begin{lemma}\label[lemma]{lemma:delta_J_123_rates}
    Let $C,C_1,C_5$ be the universal constants from \Cref{lemma:bound_yj_sigma}. Recall conditions $\rom{1}$, $\rom{2}$, $\rom{3}$ and respective $J^{\rom{1}}$,$J^{\rom{2}}$,$J^{\rom{3}}$. Assume that for each case the relevant $J$ case holds for at least $\delta_1$. Recall that $\delta_J \coloneqq \frac{d}{2^{J-1}(C+1)}$ and $c = 2(C+1)$.
    \\ Then
    \begin{align}
        \left(\frac{\delta_{J^{\rom{2}}}}{\Upsilon}\right)^2 &\asymp \frac{1}{N}, \label{eq:delta_J_asymp_2}
        \\\left(\frac{\delta_{J^{\rom{3}}}}{\Upsilon}\right)^2 &\asymp \varepsilon^2\log^2(1/\varepsilon), \label{eq:delta_J_asymp_3}
    \shortintertext{and if $N \gtrsim \sup_{\delta>0} \log \mkloc(\delta,2c)$ then}
        \left(\frac{\delta_{J^{\rom{1}}}}{\Upsilon}\right)^2 &\asymp \min\bigg(\frac{\delta_{J^{\rom{1}}}^2}{\Upsilon^2},\frac{\delta_{J^{\rom{1}}}}{\Upsilon}\bigg) ,\label{eq:delta_J_asymp_4}
        \shortintertext{and}
        \delta_{J^{\rom{1}}} &\gtrsim \delta_{J^{\rom{2}}}. \label{eq:delta_J_asymp_1}
    \end{align}
\end{lemma}

\begin{delayedproof}{lemma:delta_J_123_rates}
    We handle each of these rates in order. 
    First, by condition \ref{bound_cond2}, $\left(\frac{\delta_{J^{\rom{2}}}}{\Upsilon}\right)^2 \lesssim \left(\frac{\delta_{J^{\rom{2}}}}{2\Upsilon}\right)^2 = \left(\frac{\delta_{J^{\rom{2}}+1}}{\Upsilon}\right)^2 \leq N^{-1} \leq \left(\frac{\delta_{J^{\rom{2}}}}{\Upsilon}\right)^2$ and \cref{eq:delta_J_asymp_2} holds, and by condition \ref{bound_cond3} and similar logic \cref{eq:delta_J_asymp_3} holds.
    \\ Now, note that because $J^{\rom{1}}$ is the \textbf{largest} $J$ for which \ref{bound_cond1} holds, we have
    \begin{align*}
       C_5 \min\left(\frac{\delta_{J^{\rom{1}}}^2}{\Upsilon^2},\frac{\delta_{J^{\rom{1}}}}{\Upsilon}\right) 
        \lesssim C_5 \min\left(\frac{\delta_{J^{\rom{1}}}^2}{4\Upsilon^2},\frac{\delta_{J^{\rom{1}}}}{2\Upsilon}\right) 
        \leq \frac{4\log \mkloc\left(\frac{\delta_{J^{\rom{1}}}}{2}c,2c\right) \bigvee \log 2}{N}
    \end{align*}
    If $\delta_{J^{\rom{1}}}\leq \Upsilon$ then $\min\left(\frac{\delta_{J^{\rom{1}}}^2}{\Upsilon^2},\frac{\delta_{J^{\rom{1}}}}{\Upsilon}\right) = \frac{\delta_{J^{\rom{1}}}^2}{\Upsilon^2}$ and the result holds trivially.
    
    Now assume $\delta_{J^{\rom{1}}} > \Upsilon$. 
    \\ Note that $\frac{\delta_{J^{\rom{1}}}^2}{\Upsilon^2} \geq \frac{\delta_{J^{\rom{1}}}}{\Upsilon}$, so it suffices to show $\frac{\delta_{J^{\rom{1}}}^2}{\Upsilon^2} \lesssim \frac{\delta_{J^{\rom{1}}}}{\Upsilon}$.
    Now since $N\gtrsim \sup_{\delta>0} \log \mkloc(\delta,2c)$
    \begin{align*}
        C_5 \frac{\delta_{J^{\rom{1}}}}{2\Upsilon} = \frac{1}{2}C_5 \min\left(\frac{\delta_{J^{\rom{1}}}^2}{\Upsilon^2},\frac{\delta_{J^{\rom{1}}}}{\Upsilon}\right) \lesssim \frac{4\log \mkloc\left(\frac{1}{2}\delta_{J^{\rom{1}}}c,2c\right) \bigvee \log 2}{\sup_{\delta>0}\log \mkloc(\delta,2c) \bigvee 1} \lesssim B 
    \end{align*}
    for constant $B$. Then $\frac{\delta_{J^{\rom{1}}}}{2\Upsilon}\lesssim \frac{B}{C_5}$ so $\frac{\delta_{J^{\rom{1}}}^2}{4\Upsilon^2}\lesssim \frac{B^2}{C_5^2}$ and \Cref{eq:delta_J_asymp_4} follows by noting that 
    \begin{align*}
        N \left(\frac{\delta_{J^{\rom{1}}}}{\Upsilon}\right)^2 \lesssim N \left(\frac{\delta_{J^{\rom{1}}}}{2\Upsilon}\right)^2 \lesssim N \frac{B^2}{C_5^2} \lesssim N \frac{\delta_{J^{\rom{1}}}}{\Upsilon}.
    \end{align*}
    In other words, $\frac{\delta_{J^{\rom{1}}}}{\Upsilon} \lesssim 1$ which implies that $\min\left(\frac{\delta_{J^{\rom{1}}}^2}{4\Upsilon^2},\frac{\delta_{J^{\rom{1}}}}{2\Upsilon}\right) \asymp \frac{\delta_{J^{\rom{1}}}^2}{4\Upsilon^2}$.
    
    Finally, $\left(\frac{\delta_{J^{\rom{1}}}}{\Upsilon}\right)^2 \asymp C_5\min\bigg(\frac{\delta_{J^{\rom{1}}}^2}{\Upsilon^2},\frac{\delta_{J^{\rom{1}}}}{\Upsilon}\bigg) \geq \frac{\log \mkloc\left(c\delta_{J^{\rom{1}}},2c\right) \bigvee 1}{N} \geq N^{-1} \asymp \left(\frac{\delta_{J^{\rom{2}}}}{\Upsilon}\right)^2$ and \Cref{eq:delta_J_asymp_1} holds.
\end{delayedproof}

Now, with the above lemmas established we can complete the proof of \Cref{theorem:upper_bound}.

\begin{delayedproof}{theorem:upper_bound}
    If all three conditions in \Cref{lemma:sigma_star_star_bound} (\ref{bound_cond1}, \ref{bound_cond2}, and \ref{bound_cond3}) hold then \Cref{lemma:sigma_star_star_bound} implies that \Cref{eq:upper_bound} holds. As we traverse the tree in \Cref{algorithm:robust_algorithm} $J$ grows and $\delta_J$ shrinks, implying that there exists some minimal $J$ where at least one of the conditions does not hold (that is, where the condition ``breaks"). We proceed with the proof by analyzing each of these cases separately, first for the case that each never holds and then for the case that all hold at level 1. When all thre cases hold at level 1, note that
    \begin{itemize}
        \item $d^2 \asymp \delta_1^2 \gtrsim \delta_{J^\star}^2$ by the definition of $\delta_J$,
        \item $d^2 \asymp \delta_1^2 \gtrsim \delta_{J^{\rom{2}}}^2 \asymp \frac{\Upsilon^2}{N}$ by definition of $\delta_J$ and \eqref{eq:delta_J_asymp_2}, and
        \item $d^2 \asymp \delta_1^2 \gtrsim \delta_{J^{\rom{3}}}^2 \asymp \Upsilon^2\varepsilon^2\log^2(1/\varepsilon)$ definition of $\delta_J$ and \eqref{eq:delta_J_asymp_3},
    \end{itemize}
    and thus $d^2 \gtrsim \max\left(\delta_{J^\star}^2,\frac{\Upsilon^2}{N},\varepsilon^2\log^2(1/\varepsilon)\right)$. 
    \\\textbf{Case 1: At least one of the conditions never holds.}
    \\\textbf{Case 1A: Condition \ref{bound_cond1} never holds.}
    Here we have that $\delta_{J^\star}^2 = \delta_1^2 \asymp d^2$ and thus
        \begin{align*}
            \E_\xi \|\Sigma - \Sigma^{\star\star}\|_F^2 \leq d^2 = \left(d^2 \vee \varepsilon^2 \log^2(1/\varepsilon) \right) \bigwedge d^2 \asymp \left(\delta_{J^\star}^2 \vee \varepsilon^2 \log^2(1/\varepsilon) \right) \bigwedge d^2.
        \end{align*}
    \\\textbf{Case 1B: Condition \ref{bound_cond3} never holds.}
    The mere fact of condition \ref{bound_cond3} failing gives
        \begin{align*}
            \E_\xi \|\Sigma - \Sigma^{\star\star}\|_F^2 \leq d^2 \asymp \left(\delta_{J^\star}^2 \vee \varepsilon^2 \log^2(1/\varepsilon) \right) \bigwedge d^2.
        \end{align*}
    \\\textbf{Case 1C: Condition \ref{bound_cond2} never holds.}
    If condition $\rom{1}$ holds at level 1 we have that \\ $C_5 N \min\big(\frac{\delta_{J^{\rom{1}}}^2}{\Upsilon^2},\frac{\delta_{J^{\rom{1}}}}{\Upsilon}\big) \geq \log 2$ and therefore $\delta_{J^{\rom{1}}}^2 \geq \frac{\Upsilon^2}{N}\frac{\log 2}{C_5} > \frac{\log 2}{C_5}\delta_1^2$. On the other hand, $\delta_{J^{\rom{1}}} \leq \delta_1$ so $\delta_{J^{\rom{1}}}^2 \asymp \delta_1^2 \asymp d^2$. 
    If condition $\rom{1}$ never holds then $\delta_{J^\star}^2 = \delta_1^2 \asymp d^2$.
    Therefore we obtain our desired bound 
    \begin{align*}
        \E_\xi \|\Sigma - \Sigma^{\star\star}\|_F^2 \leq d^2 \asymp \left(\delta_{J^\star}^2 \vee \varepsilon^2 \log^2(1/\varepsilon) \right) \bigwedge d^2.
    \end{align*}
    \\\textbf{Case 2: All three conditions hold at level 1.}
    \\\textbf{Case 2A:  Condition \ref{bound_cond1} breaks at $J^\dagger+1$}
        Note that $J^\star = J^{\rom{1}}$. Recall that \[d^2\asymp\delta_1^2 \gtrsim \max\left(\delta_{J^\star}^2,\frac{\Upsilon^2}{N},\varepsilon^2\log^2(1/\varepsilon)\right).\] Furthermore, we have \begin{align}\delta_{J^\star}^2 = \delta_{J^{\rom{1}}}^2 = \delta_{J^\dagger}^2 \geq \delta_{J^{\rom{3}}}^2 \asymp \Upsilon^2 \varepsilon^2\log^2(1/\varepsilon).\label{eq:delta_j_star_eps}\end{align} Therefore
        \begin{align*}
            \E_\xi \|\Sigma - \Sigma^{\star\star}\|_F^2 \lesssim \delta_{J^\star}^2 + \frac{\Upsilon^2}{N} \lesssim \delta_{J^\star}^2 \lesssim \left(\delta_{J^\star}^2 \vee\varepsilon^2\log^2(1/\varepsilon)\right)\bigwedge d^2,
        \end{align*}
        where the first inequality follows from \Cref{lemma:sigma_star_star_bound}, the second is due to \Cref{lemma:delta_J_123_rates} \\ $\left(\delta^2_{J^{\rom{1}}} \gtrsim \delta^2_{J^{\rom{2}}} \asymp \frac{\Upsilon^2}{N}\right)$, and the final inequality is due to \Cref{eq:delta_j_star_eps} and the fact that $d^2 \gtrsim \max\left(\delta_{J^\star}^2,\frac{\Upsilon^2}{N},\varepsilon^2\log^2(1/\varepsilon)\right)$.
    \\\textbf{Case 2B: Condition \ref{bound_cond3} breaks at $J^\dagger+1$.}
        Recall that this implies $J^\star = J^{\rom{1}}$. Note that $\delta^2_{J^{\rom{3}}} = \delta_{J^\dagger}^2 \geq \delta_{J^\star}^2 \vee \delta_{J^{\rom{2}}}^2$, so $\delta_{J^{\rom{3}}}^2 \geq \delta_{J^{\rom{2}}}^2 \asymp \Upsilon^2/N$ and $\varepsilon^2\log^2(1/\varepsilon) \asymp \delta_{J^{\rom{3}}}^2 \geq \delta_{J^\star}^2$. Therefore 
        \begin{align*}
            \E_\xi \|\Sigma - \Sigma^{\star\star}\|_F^2 \lesssim \delta^2_{J^{\rom{3}}} + \frac{\Upsilon^2}{N} \asymp \varepsilon^2\log^2(1/\varepsilon) \asymp \left(\delta_{J^\star}^2 \vee \varepsilon^2\log^2(1/\varepsilon)\right)\bigwedge d^2.
        \end{align*}
    \textbf{Case 2C: Condition \ref{bound_cond2} breaks at $J^\dagger+1$.}
        Now, assume without loss of generality that $J^{\rom{1}}\bigwedge J^{\rom{3}} \geq 1$. If this is not true, follow the relevant previous proof. Recall one final time that this implies $J^\star = J^{\rom{1}}$. Now we have $\delta^2_{J^{\rom{2}}} = \delta_{J^\dagger}^2 \geq \delta_{J^\star}^2 \vee \delta_{J^{\rom{3}}}^2$ which, with \Cref{lemma:delta_J_123_rates}, implies $\delta_{J^\star}^2 \gtrsim  \delta^2_{J^{\rom{2}}} \geq \delta_{J^{\rom{3}}}^2 \asymp \varepsilon^2\log^2(1/\varepsilon)$. This all gives us $\delta_{J^\star}^2 \asymp \delta_{J^{\rom{2}}}^2 \asymp \Upsilon^2/N \gtrsim \varepsilon^2\log^2(1/\varepsilon)$. Thus, we achieve our desired bound 
        \begin{align*}
            \E_\xi \|\Sigma - \Sigma^{\star\star}\|_F^2 \leq \delta_{J^{\rom{2}}}^2 + \frac{\Upsilon^2}{N} \asymp \left(\delta_{J^\star}^2 \vee \varepsilon^2 \log^2(1/\varepsilon) \right)
            \bigwedge d^2.
        \end{align*}
\end{delayedproof}

\section{Proofs for Section \ref{sec:examples} (examples)}\label{appendix:ex_proofs}

\begin{delayedproof}{lemma:Uq_star_shaped}
    Note that trivially $c_0\mathbb{I} \in \mathcal{U}_q$. Consider an arbitrary point $\Sigma^0 \in \mathcal{U}_q(s_0(n))$ and 
\begin{align*}
    \Sigma^{1} &= \alpha \Sigma^0 + (1-\alpha)c_0\mathbb{I}.
\end{align*}
First note that 
\begin{align*}
    \lambda_{\min}(\Sigma^1) &= \lambda_{\min}(\alpha \Sigma^0 + (1-\alpha)c_0\mathbb{I}) 
    \\ &= \alpha \lambda_{\min}(\Sigma^0) + (1-\alpha)c_0
    \\ &\geq c_0
\end{align*}
and
\begin{align*}
    \lambda_{\max}(\Sigma^1) &= \lambda_{\max}(\alpha \Sigma^0 + (1-\alpha)c_0\mathbb{I}) 
    \\ &= \alpha \lambda_{\max}(\Sigma^0) + (1-\alpha)c_0
    \\ &\leq c_1.
\end{align*}
We conclude by showing separately that the column bounds hold when $q \in (0,1)$ and $q=0$.
\\\textbf{Case 1: $q > 0$.}
For an arbitrary $i \in n$ we have 
\begin{align*}
    \sum_{j=1}^n |\sigma^1_{ij}|^q &= \sum_{j=1}^n |\alpha \sigma_{ij}^0 + (1-\alpha) c_0\indicator{i=j}|^q
    \\ &= |\alpha\sigma_{ii}^0 + (1-\alpha)c_0|^q + \sum_{j\neq i}\alpha^q|\sigma_{ij}^0|^q
    \\ &\leq |\alpha\sigma_{ii}^0 + (1-\alpha)\sigma_{ii}^0|^q + \sum_{j\neq i}\alpha^q|\sigma_{ij}^0|^q
    \\ &\leq \sum_{j=1}^n |\sigma^0_{ij}|^q \leq s_0(n),
\end{align*}
so the column bounds hold.
\\\textbf{Case 2: $q =0$} Note that in this case shrinking $\Sigma^0$ and adding to the diagonal does not incease the support, because $\lambda_{\min}(\Sigma^0) \geq c_0 > 0$ the diagonal of $\Sigma^0$ must be positive. 

\end{delayedproof}

\begin{delayedproof}{lemma:delta_properties}

Plug in our value of $\delta$ to \eqref{eq:n_to_log_n}see that
    \begin{align*}
    \log \bigg( n^{2-\frac{2}{2-q}}s_0(n)^{-\frac{2}{2-q}}\delta^{\frac{2q}{2-q}} \bigg) &\lesssim \log\bigg(n^{2-\frac{2}{2-q}}s_0(n)^{-\frac{2}{2-q}}\bigg(ns_0(n)\left(\frac{\log(n)}{N}\right)^{\frac{2-q}{2}}\bigg)^{\frac{q}{2-q}} \bigg)
    \\&= \log\bigg(n^{2-\frac{2}{2-q}+\frac{q}{2-q}}s_0(n)^{\frac{q-2}{2-q}}\bigg(\frac{\log n}{N}\bigg)^{q/2}\bigg)
    \\&= \log\bigg(\frac{n}{s_0(n)}\bigg(\frac{\log n}{N}\bigg)^{q/2}\bigg)
    \\ &= \log n - \log s_0(n) + \frac{q}{2}\log\bigg(\frac{\log n}{N}\bigg)
    \\ &\lesssim \log n + \frac{q}{2}\log\bigg(\frac{\log n}{N}\bigg)
\end{align*}
which is bounded above by a constant multiple of $\log (n)$ if $\frac{\log n}{N} \lesssim  n$ (which is weaker than our required $N \gtrsim n^2$), and gives \eqref{eq:n_to_log_n}. 
Next, plug in the value of $\delta$ to \eqref{eq:Rdelta_bound} get that
\begin{align*}
    R^{-1}\delta &= n^{-1/q}s_0(n)^{-1/q}\delta 
    \\ &=  n^{-1/q}s_0(n)^{-1/q} n^{1/2} s_0(n)^{1/2} \bigg(\frac{\log n}{N}\bigg)^{1/2-q/4}
    \\ &= n^{1/2-1/q} s_0(n)^{1/2-1/q}\bigg(\frac{\log n}{N}\bigg)^{1/2-q/4} 
\end{align*}
and we see that %the first condition of $R^{-1}\delta \gtrsim n^{1-2/q}$ \eqref{eq:Rdelta_bound} holds if
$R^{-1}\delta \gtrsim n^{1-2/q}$ \eqref{eq:Rdelta_bound} holds if
\begin{align*}
    n^{1-2/q} &\lesssim n^{1/2-1/q}s_0(n)^{1/2-1/q} \bigg(\frac{\log n}{N}\bigg)^{1/2-q/4}
    \\ n^{1/2 - 1/q} &\lesssim s_0(n)^{1/2-1/q} \bigg(\frac{\log n}{N}\bigg)^{1/2-q/4}
\end{align*}
and we have that 
\begin{align*}
\bigg(\frac{n}{s_0(n)}\bigg)^{1/2-1/q} &\lesssim \bigg(\frac{\log n}{N}\bigg)^{1/2-q/4}
\\ \bigg(\frac{s_0(n)}{n}\bigg)^{(2-q)/(2q)} &\lesssim \bigg(\frac{\log n}{N}\bigg)^{(2-q)/4}
\\ \frac{s_0(n)}{n} &\lesssim \bigg(\frac{\log n}{N}\bigg)^{q/2}
\\ s_0(n) & \lesssim n\bigg(\frac{\log n}{N}\bigg)^{q/2}
\end{align*}
and note that this is a weaker condition than $s_0(n) \lesssim nN^{-q/2}\log^{q/2-1}n$. 
% Now, the second condition holds if 
% \begin{align*}
%     n^{1-2/q} &\lesssim \frac{n^{1-1/q} s_0(n)^{-1/q}}{N^{1/2}}
%     \\ n^{-1/q}&\lesssim \frac{1}{\sqrt{N}}s_0(n)^{-1/q}
%     \\ s_0(n)^{1/q} &\lesssim n^{1/q} N^{-1/2}
%     \\ s_0(n) &\lesssim n N^{-q/2}
% \end{align*}

\end{delayedproof}

\begin{delayedproof}{lemma:spiked_local_entropy}
Fix an arbitrary
\[
    \Sigma_0=I+x_0x_0^\top\in K^+,
\]
and let
\[
    \{\Sigma_i=I+x_ix_i^\top\}_{i=1}^M
    \subseteq
    B_F(\Sigma_0,\delta)\cap K^+
\]
be a $\delta/c$-packing in Frobenius norm.  We will bound $M$
uniformly over $\Sigma_0$.  Write
\[
    a\coloneqq \|x_0\|_2,
    \qquad
    b_i\coloneqq \|x_i\|_2.
\]
For every $i$,
\begin{equation}
    |b_i^2-a^2|
    =
    \left|
        \|x_ix_i^\top\|_F-\|x_0x_0^\top\|_F
    \right|
    \leq
    \|x_ix_i^\top-x_0x_0^\top\|_F
    \leq \delta.
    \label{eq:spike_norm_control}
\end{equation} 
We separate the proof into two cases.

\medskip
\noindent
\textbf{Case 1: $a^2\leq 2\delta$.}
By \eqref{eq:spike_norm_control},
\[
    b_i^2\leq a^2+\delta\leq 3\delta,
\]
and therefore
\[
    x_i\in B_2(0,\sqrt{3\delta})\cap K^0.
\]
On the other hand, for $i\neq j$,
\begin{align*}
    \frac{\delta}{c}
    &\leq
    \|x_ix_i^\top-x_jx_j^\top\|_F \\
    &=
    \|(x_i-x_j)x_i^\top+x_j(x_i-x_j)^\top\|_F \\
    &\leq
    (b_i+b_j)\|x_i-x_j\|_2 \\
    &\leq
    2\sqrt{3\delta}\,\|x_i-x_j\|_2.
\end{align*}
Hence
\[
    \|x_i-x_j\|_2
    \geq
    \frac{\sqrt{\delta}}{2\sqrt{3}\,c}.
\]
Thus $\{x_i\}_{i=1}^M$ is a
$\sqrt{\delta}/(2\sqrt{3}c)$-packing of
$B_2(0,\sqrt{3\delta})\cap K^0$, and consequently
\[
    M
    \leq
    \mathcal{M}_{K^0}^{\operatorname{loc}}
    (\sqrt{3\delta},6c).
\]
Since $\delta\leq K_{\max}^2$,
\[
    \sqrt{3\delta}
    \geq
    \frac{\delta}{K_{\max}}.
\]
Because $K^0$ is star-shaped, its local metric entropy is
non-increasing in the local radius.  Therefore
\begin{equation}
    M
    \leq
    \mathcal{M}_{K^0}^{\operatorname{loc}}
    \left(
        \frac{\delta}{K_{\max}},
        6c
    \right).
    \label{eq:small_spike_case}
\end{equation}

\medskip
\noindent
\textbf{Case 2: $a^2>2\delta$.}
In this case \eqref{eq:spike_norm_control} gives
\begin{equation}
    \frac{a}{\sqrt{2}}
    \leq b_i
    \leq
    \sqrt{\frac32}\,a
    \qquad
    \text{for every }i.
    \label{eq:bi_comparable_a}
\end{equation}
For each $i$, choose $s_i\in\{-1,1\}$ so that
$s_i\langle x_i,x_0\rangle\geq 0$.  We claim that
\begin{equation}
    \|x_i-s_ix_0\|_2
    \leq
    \frac{2\delta}{a}.
    \label{eq:local_inverse_rank_one}
\end{equation}
To see this, set
\[
    t_i
    \coloneqq
    \frac{|\langle x_i,x_0\rangle|}{ab_i}
    \in[0,1].
\]
The identity
\begin{align*}
    \|x_ix_i^\top-x_0x_0^\top\|_F^2
    &=
    b_i^4+a^4-2\langle x_i,x_0\rangle^2 \\
    &=
    (b_i^2-a^2)^2
    +
    2a^2b_i^2(1-t_i^2)
\end{align*}
and the fact that the left-hand side is at most $\delta^2$ implies
\[
    2a^2b_i^2(1-t_i^2)\leq \delta^2.
\]
Moreover,
\[
    |b_i-a|
    =
    \frac{|b_i^2-a^2|}{a+b_i}
    \leq
    \frac{\delta}{a}.
\]
Therefore, using $1-t_i\leq 1-t_i^2$ and
$b_i\geq a/\sqrt{2}$,
\begin{align*}
    \|x_i-s_ix_0\|_2^2
    &=
    (b_i-a)^2+2ab_i(1-t_i) \\
    &\leq
    \frac{\delta^2}{a^2}
    +
    2ab_i(1-t_i^2) \\
    &\leq
    \frac{\delta^2}{a^2}
    +
    \frac{\delta^2}{ab_i} \\
    &\leq
    (1+\sqrt{2})\frac{\delta^2}{a^2}
    <
    \frac{4\delta^2}{a^2},
\end{align*}
which proves \eqref{eq:local_inverse_rank_one}.

Partition the packing indices according to the sign $s_i$:
\[
    I_+\coloneqq\{i:s_i=1\},
    \qquad
    I_-\coloneqq\{i:s_i=-1\}.
\]
By \eqref{eq:local_inverse_rank_one},
\[
    \{x_i:i\in I_+\}
    \subseteq
    B_2\left(x_0,\frac{2\delta}{a}\right)\cap K^0.
\]
If $I_-\neq\varnothing$, choose any $i_0\in I_-$ and put
$z=x_{i_0}\in K^0$.  Then, again by
\eqref{eq:local_inverse_rank_one},
\[
    \{x_i:i\in I_-\}
    \subseteq
    B_2\left(z,\frac{4\delta}{a}\right)\cap K^0.
\]
Thus both sign classes are contained in a Euclidean ball, centered
at a point of $K^0$, of radius at most $4\delta/a$.

It remains to control their packing separation.  For $i\neq j$,
\begin{align*}
    \frac{\delta}{c}
    &\leq
    \|x_ix_i^\top-x_jx_j^\top\|_F \\
    &\leq
    (b_i+b_j)\|x_i-x_j\|_2 \\
    &\leq
    \sqrt{6}\,a\,\|x_i-x_j\|_2,
\end{align*}
where the last inequality follows from
\eqref{eq:bi_comparable_a}.  Hence
\[
    \|x_i-x_j\|_2
    \geq
    \frac{\delta}{\sqrt{6}\,ca}
    =
    \frac{4\delta/a}{4\sqrt{6}\,c}.
\]
It follows that each sign class has cardinality at most
\[
    \mathcal{M}_{K^0}^{\operatorname{loc}}
    \left(
        \frac{4\delta}{a},
        4\sqrt{6}\,c
    \right).
\]
Since $a\leq K_{\max}$,
\[
    \frac{4\delta}{a}
    \geq
    \frac{\delta}{K_{\max}},
\]
so the monotonicity of local metric entropy again yields
\begin{equation}
    M
    \leq
    2\,
    \mathcal{M}_{K^0}^{\operatorname{loc}}
    \left(
        \frac{\delta}{K_{\max}},
        4\sqrt{6}\,c
    \right).
    \label{eq:large_spike_case}
\end{equation}

Combining \eqref{eq:small_spike_case} and
\eqref{eq:large_spike_case}, and taking for example $C=10$,
gives \eqref{eq:spiked_local_entropy}.  Since the bound is uniform
over $\Sigma_0\in K^+$, taking the supremum over the local center
completes the proof.
\end{delayedproof}

\end{document}